\documentclass[11pt, twoside]{amsart}

\title[Finite-dimensional quantum groups of type Super A]{
Finite-dimensional quantum groups of type Super A and non-semisimple modular categories
}

\author{Robert Laugwitz}
\address{School of Mathematical Sciences,
University of Nottingham, University Park, Nottingham, NG7 2RD, UK}
\email{robert.laugwitz@nottingham.ac.uk}

\author{Guillermo Sanmarco}
\address{Department of Mathematics, University of Washington, Seattle, WA 98195, USA}
\email{sanmarco@uw.edu}

\date{February 3, 2026}

\usepackage{mathabx}
\usepackage{import}
\usepackage{amsmath}
\usepackage{amsfonts}
\usepackage{amsthm}
\usepackage{amssymb,bbm}
\usepackage{dsfont}
\usepackage[alphabetic, initials, nobysame]{amsrefs}
\usepackage[english]{babel}
\usepackage{url}
\usepackage{fancyhdr}
\usepackage{graphicx}
\usepackage[
colorlinks=true,
linkcolor=black, 
anchorcolor=black,
citecolor=black,
urlcolor=black, 
]{hyperref}
\usepackage[all]{xy}
\usepackage{geometry}
\usepackage{microtype} 
\usepackage[dvipsnames]{xcolor}
\usepackage{cleveref}
\usepackage{caption}
\usepackage{subcaption}

\usepackage{tikz}
\usetikzlibrary{arrows,positioning,decorations.markings}
\usepackage{tikz-cd}

\allowdisplaybreaks

\definecolor{forest}{rgb}{0.0, 0.5, 0.0}

\usepackage[justification=centering]{caption}

\providecommand{\abs}[1]{\lvert#1\rvert}

\newcommand{\leftexpsub}[3]{{\vphantom{#3}}^{#1}_{#2}{#3}}
\newcommand{\lYD}[1]{\leftexpsub{#1}{#1}{\mathsf{YD}}}
\newcommand{\Set}[1]{\left\lbrace #1\right\rbrace}
\newcommand{\inner}[1]{\left\langle #1\right\rangle}

\newcommand{\op}[1]{\mathrm{#1}}

\newcommand{\ov}[1]{\overline{#1}}
\newcommand{\un}[1]{\underline{#1}}
\newcommand{\lmod}[1]{#1\text{-}\mathsf{mod}}
\newcommand{\glmod}[1]{#1\text{-}\mathsf{mod}^{\mathbb{Z}}}
\newcommand{\lcomod}[1]{#1\text{-}\mathsf{comod}}

\newcommand{\Aut}{\operatorname{Aut}}
\newcommand{\ch}{\operatorname{ch}}

\newcommand{\coev}{\mathsf{coev}}
\newcommand{\deri}{\operatorname{d}}
\newcommand{\Drin}{\operatorname{Drin}}
\newcommand{\ev}{\mathsf{ev}}
\newcommand{\End}{\operatorname{End}}
\newcommand{\Hom}{\mathsf{Hom}}
\newcommand{\ide}{\mathsf{Id}}
\newcommand{\Ind}{\mathsf{Ind}}
\newcommand{\isomorph}{\stackrel{\sim}{\to}}
\newcommand{\one}{\mathds{1}}
\newcommand{\ord}{\operatorname{ord}}
\newcommand{\Res}{\operatorname{Res}}
\newcommand{\rad}{\operatorname{rad}}
\newcommand{\rev}{\mathrm{rev}}
\newcommand{\Comp}{\operatorname{Comp}}

\newcommand{\BB}{\mathfrak{B}}
\newcommand{\II}{\mathfrak{I}}
\newcommand{\superqa}[3]{{\bf A}_{#1}(#2|#3)}
\newcommand{\superqam}[1]{{\bf t}^{#1}}
\newcommand{\qs}{\mathbf{q}}

\newcommand{\Vect}{\mathsf{vect}_\Bbbk}
\newcommand{\svect}{\mathsf{vect}_{-1}}

\providecommand{\fr}[1]{\mathfrak{#1}}

\newcommand{\mI}{\mathbb{I}}
\newcommand{\mJ}{\mathbb{J}}

\newcommand{\mZ}{\mathbb{Z}}
\newcommand{\mN}{\mathbb{N}}

\newcommand{\cA}{\mathcal{A}}
\newcommand{\cC}{\mathcal{C}}
\newcommand{\cD}{\mathcal{D}}

\newcommand{\cI}{\mathcal{I}}
\newcommand{\cL}{\mathcal{L}}
\newcommand{\cT}{\mathcal{T}}
\newcommand{\cN}{\mathcal{N}}
\newcommand{\cZ}{\mathcal{Z}}

\newcommand{\rI}{\mathrm{I}}

\newcommand{\ru}{\mathrm{u}}

\newcommand{\bfp}{\mathbf{p}}

\newtheoremstyle{defstyle}
  {0.5cm}                   
  {0.5cm}                   
  {\normalfont}           
  {}     
  {\normalfont\bfseries}  
  {:}                     
  {0.3cm}              
  {\thmname{#1}\thmnumber{ #2}\thmnote{ (#3)}}

\numberwithin{equation}{subsection}

\newtheorem*{rep@theorem}{\rep@title}
\newcommand{\newreptheorem}[2]{%
\newenvironment{rep#1}[1]{%
 \def\rep@title{#2 \ref{##1}}%
 \begin{rep@theorem}}%
 {\end{rep@theorem}}}
\makeatother

\newtheorem{theorem}{Theorem}[section]

\newtheorem{proposition}[theorem]{Proposition}
\newreptheorem{proposition}{Proposition}
\newtheorem{corollary}[theorem]{Corollary}
\newreptheorem{corollary}{Corollary}
\newtheorem{lemma}[theorem]{Lemma}
\newreptheorem{lemma}{Lemma}
\newtheorem{conjecture}[theorem]{Conjecture}

\newtheorem{theorem*}{Theorem}
\newreptheorem{theorem}{Theorem}

\theoremstyle{definition}
\newtheorem{definition}[theorem]{Definition}

\newtheorem{example}[theorem]{Example}
\newtheorem{remark}[theorem]{Remark}

\newtheorem{question}[theorem]{Question}

\makeatletter
\@namedef{subjclassname@2020}{%
  \textup{2020} Mathematics Subject Classification}
\makeatother

\subjclass[2020]{Primary 18M15, 18M20, 17B37; Secondary 57K14, 57K16}
\keywords{braided Drinfeld double, modular tensor category, Nichols algebra, relative monoidal center, relative Drinfeld center, super-type quantum group, highest weight category, link invariant}

\begin{document}

\begin{abstract}
We construct a  series of finite-dimensional quantum groups as braided Drinfeld doubles of Nichols algebras of type Super A, for an even root of unity, and classify ribbon structures for these quantum groups. Ribbon structures exist if and only if the rank is even and all simple roots are odd. In this case, the quantum groups have a unique ribbon structure which comes from a non-semisimple spherical structure on the negative Borel Hopf subalgebra.
Hence, the categories of finite-dimensional modules over these quantum groups provide examples of non-semisimple modular categories. In the rank-two case, we explicitly describe all simple modules of these quantum groups. We finish by computing link invariants, based on generalized traces, associated to a four-dimensional simple module of the rank-two quantum group. These knot invariants distinguish certain knots indistinguishable by the Jones or HOMFLYPT polynomials and are related to a specialization of the Links-Gould invariant.
\end{abstract}

\maketitle

\tableofcontents


\section{Introduction}

\subsection{Motivation}

Modular categories are tensor categories of special interest to quantum algebra, low-dimensional topology, and quantum field theory. Modular tensor categories are usually assumed to be semisimple, i.e., they are modular fusion categories; in this case, they provide invariants of framed links and $3$-manifold through topological quantum field theories (TQFTs) of surgery or Reshetikhin--Turaev type \cite{RT91}. Notable examples are obtained using the semisimple quotient of categories of tilting modules over the quantum group associated to a semisimple complex Lie algebra $\fr{g}$ and a root of unity $q$, see \cite{AP} or, e.g., \cite{BK}*{Section 3.3}.

The non-semisimple category of $u_q(\fr{g})$-modules, where $q$ is a root of unity of odd order $\ell\geq 3$,\footnote{When $\fr{g}$ is of type $G_2$, assume $l$ is coprime to $3$} is a finite non-degenerate ribbon category and, thus, possesses all features of a modular fusion category besides semisimplicity. Moreover, the non-semisimple category  of $u_q(\fr{g})$-modules possesses richer structure due to non-split extensions and projective objects of quantum dimension zero which vanish in the semisimple quotient. It was already observed by V.~Lyubashenko \cites{Lyu,KL} that invariants of $3$-manifolds, and representations of mapping class groups, can be obtained from a non-semisimple modular category. Lyubashenko's constructions were later renormalized and extended to a $3$-dimensional TQFT \cites{DGGPR,DGGPR2}; see also \cites{LMSS,SW} for extensions of Lyubashenko's construction of mapping class group actions to cochain complexes. More recently, $3$-dimensional quantum field theories were constructed from non-semisimple categories of quantum group representations \cite{CDGG}.

Examples of non-semisimple modular categories arising from different sources can be found in the literature. These include examples obtained from the representation theory of quantum groups \cites{LO,GLO,CGR,Negron}, categorical constructions of Drinfeld centers \cite{Shi1} or M\"uger centralizers and relative centers \cite{LW2}, and categories of modules over logarithmic vertex operator algebras (see \cites{GSTF,FT,FS1,CGR, Len,CLR,GN} and references therein). Compared to the more established  modular fusion categories, there is still a lack of examples of non-semisimple modular categories, especially, examples related to Cartan data of super-type. However, we note that a finite-dimensional restricted super quantum group of type $\fr{gl}(1|1)$ was constructed in \cite{AGPS}*{Section~4} and shown to be ribbon and factorizable.

This paper contributes a new series of non-semisimple modular categories obtained as modules over finite-dimensional Hopf algebras $\ru_q(\fr{sl}_{r,\mI})$ constructed from generalized Lie-theoretic data of type Super A and an even root of unity $q$. These Hopf algebras $\ru_q(\fr{sl}_{r,\mI})$ are defined as braided Drinfeld doubles of Nichols algebras which have root systems associated to those of Lie superalgebras $\fr{sl}(m|r+1-m)$ where $m$ is either $\frac{r}{2}$ when $r$ is even or $\frac{r+1}{2}$ when $r$ is odd \cite{AA}. However, only the case when $r$ is even, related to $\fr{sl}(\tfrac{r}{2}|\tfrac{r}{2}+1)$ leads to modular categories. Such general series of non-semisimple modular categories associated to Cartan data of super-type have, to our knowledge, not previously appeared in the literature. 

We interpret the $\ru_q(\fr{sl}_{r,\mI})$ as analogues of small quantum groups of type Super A.  Unlike other approaches to quantum groups of super-type \cites{KT,Yam}, which are  infinite dimensional Hopf superalgebras, our quantum groups are finite-dimensional Hopf algebras over $\Bbbk$. We argue that considering braided Drinfeld doubles of finite Nichols algebras, drawing from their strong connections to Lie theory \cites{Hec2,AA}, gives a suitable general framework for constructing quantum groups of non-Cartan type. In type Super A and for $q$ an even root of unity, the examples we found are characterized as the only examples obtained this way that have a modular category of representations.

In non-semisimple ribbon categories, invariants of (framed) links are extended to objects of zero quantum dimension using \emph{generalized traces} \cite{GKP}. These invariants often carry interesting topological information not accessible through the semisimple Reshetikhin--Turaev invariants \cite{RT90}. Examples of link invariants obtained from quantum dimension zero objects include the multivariable Alexander polynomial, Kashaev’s invariant and the ADO invariant \cite{CGP}. The present paper provides a new supply of such link invariants (valued in $\mZ[t,t^{-1}]$) and investigates an example of such an invariant which is shown to distinguish certain knots that are not distinguished by the Jones and HOMFLYPT polynomials.

\subsection{Statement of results}

We will now summarize the main results of the paper. The classification of Nichols algebras of diagonal type in \cite{Hec1} includes a class of finite-dimensional braided Hopf algebras $\BB_\qs$ with generalized Dynkin diagrams of type Super A \cite{AA}*{Section 5.1}. These Nichols algebras are determined by a matrix $\qs=(q_{ij})$ of size $r\times r$, where $r\geq 2$, with parameters $q_{ij}$ in $\Bbbk^\times$ based on a single root of unity $q$, and a non-empty subset $\mJ\subseteq \mI=\Set{1,\ldots,r}$ of \emph{odd} simple roots,  see \Cref{def:super-A} for a precise definition.

We take $q$ to be a primitive even root of unity of order $N=2n$ and define the braided monoidal category $\cA_\qs$ of comodules over the abelian group $G=\langle g_1,\ldots, g_r\;|\; g_1^N=\ldots=g_r^N=1\rangle\cong \mZ_N^{ r}$ with braiding given  by the dual quasitriangular structure 
$$r^\qs(g_i,g_j)=q_{ji}.$$
Here, we suppose that $q_{ij}=1$ for $i>j$.  Our setup is justified in \Cref{rem:justify-setup}.

The Nichols algebra $\BB_\qs$ is realized as a Hopf algebra object in $\cA_\qs$. Thus, the bosonization $\BB_\qs\rtimes \Bbbk[G]$ of $\BB_\qs$ with the dual group algebra $\Bbbk[G]$ is a Hopf algebra over $\Bbbk$. We prove the following characterization of its unimodularity.

\begin{repproposition}{prop:unimodularity}
The bosonization $\BB_\qs\rtimes \Bbbk[G]$ of the Nichols algebra $\BB_\qs$ of type Super A is unimodular if and only if  $q_{ii}=-1$ for all $i$ (i.e. $\mJ=\mI$) and $r\geq 2$ is even.
\end{repproposition}

Thus, in the chosen setup, the only cases of Nichols algebras of type Super A that can admit a spherical structure (in the sense of \cite{DSS}) have only odd simple roots and a generalized Dynkin diagram of the following form:
$$
\vcenter{\hbox{\begin{tikzpicture}
\node [circle,draw,label=above:$-1$] (1){};
\node [circle,draw,label=above:$-1$] (2)[right of=1,node distance=1.5cm]{};
\node [circle,draw,label=above:$-1$] (3)[right of=2,node distance=1.5cm]{};
\node [circle,draw,label=above:$-1$] (4)[right of=3,node distance=1.5cm]{};
\node [circle,draw,label=above:$-1$] (5)[right of=4,node distance=1.5cm]{};
\node [circle,draw,label=above:$-1$] (6)[right of=5,node distance=1.5cm]{};
\draw  (1.east) -- (2.west) node [above,text centered,midway]
{$q^{-1}$};
\draw  (2.east) -- (3.west) node [above,text centered,midway]
{$q$};
\draw[dotted] (3.east) -- (4.west) node []{};
\draw  (4.east) -- (5.west) node [above,text centered,midway]
{$q^{-1}$};
\draw  (5.east) -- (6.west) node [above,text centered,midway]
{$q$};
\end{tikzpicture}}}
$$
This root data is different from usual approaches in the literature \cite{Yam} where only one of the simple roots is odd. Generalized root systems of the above form are, in a certain sense, related to Cartan data associated with the Lie superalgebra $\fr{sl}(\tfrac{r}{2}|\tfrac{r}{2}+1)$, cf. \Cref{rem:slnm-connection}.

Using the characterization of ribbon structures on Drinfeld doubles of \cite{KR93} and a combinatorial argument, we prove that  the above unimodular Hopf algebras are precisely those admitting a spherical structure.
\begin{reptheorem}{theorem:spherical}
The category $\lmod{\BB_\qs}(\cA_\qs)=\lmod{\BB_\qs\rtimes \Bbbk[G]}$ admits a spherical structure if and only if $\BB_\qs\rtimes \Bbbk[G]$ is unimodular if and only if $q_{ii}=-1$ for all $i$ (i.e. $\mJ=\mI$) and $r\geq 2$ is even. In this case, the spherical structure is determined by the pivotal element 
$$a=\sum_{\bf j} (-1)^{j_1+\ldots+j_r}\delta_{\bf j}\in \Bbbk[G]\subset \BB_\qs\rtimes \Bbbk[G].$$
\end{reptheorem}
Given the existence of such a spherical structure, the Drinfeld center $\cZ(\lmod{\BB_\qs}(\cA_\qs))$ is a non-degenerate finite ribbon category (i.e., a modular category) by \cite{Shi2}*{Theorem 5.11}. We further prove that these are the only instances in which ribbon structures exist on $\cZ(\lmod{\BB_\qs}(\cA_\qs))$.

\begin{reptheorem}{theorem:ribbon}
The category $\cZ(\lmod{\BB_\qs}(\cA_\qs))$ admits a ribbon structure if and only if $\BB_\qs\rtimes \Bbbk[G]$ in unimodular if and only if $\mJ=\mI$ and  $r\geq 2$ is even. 
In this case, there exist precisely $2^r$ ribbon structures only one of which corresponds to the spherical structure from \Cref{theorem:spherical}.
\end{reptheorem}

It was shown in \cite{LW2}, see also \Cref{thm:LW-center-relcenter} below, that if $\cA_\qs$ is a non-degenerate braided tensor category, there is an equivalence of ribbon categories
$$\cZ(\lmod{\BB_\qs}(\cA_\qs))\simeq \cZ_{\cA_\qs}(\lmod{\BB_\qs}(\cA_\qs))\boxtimes \cA_\qs^\rev.$$
Here, $\cZ_{\cA_\qs}(\lmod{\BB_\qs}(\cA_\qs))$ is the \emph{relative} Drinfeld center studied in \cites{Lau3,LW2}. The relative Drinfeld center is a braided tensor category equivalent to that of finite-dimensional modules over a quasi-triangular Hopf algebra 
$$\ru_q(\fr{sl}_{r,\mJ})=\Drin_{\Bbbk[G]}(\BB_\qs,\BB_\qs^*),$$
a \emph{braided Drinfeld double} (or \emph{double bosonization} of \cite{Maj99}).
A presentation for the quasi-triangular Hopf algebra $\ru_q(\fr{sl}_{r,\mJ})$ is given in \Cref{def:uqsl-pres}. Our main result is the following theorem.

\begin{reptheorem}{thm:modularunique}
Let $q$ be a primitive root of unity of order $N=2n$. The following statements are equivalent for the braided tensor category $\cC=\cZ_{\cA_\qs}(\lmod{\BB_\qs}(\cA_\qs)) \simeq \lmod{\ru_q(\fr{sl}_{r,\mJ})}$:
\begin{enumerate}
\item[(i)] $\cC$ is a ribbon category.
\item[(ii)] $\cC$ is a  modular category.
\item[(iii)] $r$ is even and $\mJ=\mI$.
\end{enumerate}
In case the equivalent statements hold, the ribbon category structure on $\cC$ is uniquely determined.
\end{reptheorem}

The category $\lmod{\ru_q(\fr{sl}_{r,\mJ})}^\mZ$ of \emph{graded} $\ru_q(\fr{sl}_{r,\mJ})$-modules is a highest weight category by results of Bellamy--Thiel \cite{BT}. The Grothedieck ring $K_0(\lmod{\ru_q(\fr{sl}_{r,\mJ})}^\mZ)$ is isomorphic to the subring of $\mZ \Lambda [t,t^{-1}]$ with basis given by the symbols of all (shifts of) simple module, where $\Lambda=\mZ^r_N$, see \Cref{sec:Koexmpl}. In the rank-two case, we obtain explicit results, including tensor product decompositions of standard modules in \Cref{prop:tensordec}. 

Already in the case of minimal rank $r=2$, with $\mJ=\mI$, which is built from super Cartan data related to $\fr{sl}(1|2)$, the  Hopf algebra $\ru_q(\fr{sl}_{r,\mI})$ displays a rich representation theory. The simple modules fall into two classes. First, for $i=0,\ldots, N-1$, there are simple modules $L(i,0)$ and their duals $L(0,i)$ of dimension $2i+1$ and quantum dimension $(-1)^i$. Second, we find simple modules $L(i,j)$ with $0< i,j < N$ of dimension $4(i+j)$ if $i+j\leq N$ and $4(i+j-N)$ otherwise, which have quantum dimension zero, see \Cref{thm:rank2}. This list of simple modules can be matched with a subset of integral weight simple typical or atypical modules over the Lie superalgebra $\fr{sl}(1|2)$, see \cites{Kac-77,Kac-78,FScS}.

As an application, we study the  link invariant associated to the four-dimensional simple $\ru_q(\fr{sl}_{2,\mI})$-module $W=L(n,n+1)$. As the quantum dimension of $W$ is zero, these link invariants are obtained via the unique generalized trace on its endomorphism ring and take values $\rI_W(\cL)$ in $\mZ[t,t^{-1}]$ for any link $\cL$. 
The braiding $\Psi$ on this module $W$ is a $16\times 16$-matrix with minimal polynomial
$$\Psi^3+(2+q^{-1})\Psi^2 +(1+2q^{-1})\Psi + q^{-1}\ide = 0,$$
akin to a skein relation. This recovers the partial skein relation of a specialization of the two-variable Links--Gould polynomial \cites{GLZ,GLO}.
With the aid of a computer, we find the knot invariant $\rI_W(\cL)$  for all knots $\cL$ with at most $7$ crossing from the Rolfson table, showing that $\rI_W$ distinguishes these knots and their mirror images. Moreover, we provide a closed formula for the invariants associated with torus links $\cT_{2,a}$, for $a\in \mZ$. Finally, we show that $\rI_W$ distinguishes the knots $5_1$ and $10_{132}$ which have the same HOMFLYPT polynomials, and the link $LL_2(2)$ which has the same Jones  polynomial as the two-component unlink.

\subsection{Summary}

We start with the necessary background definitions on results on ribbon and modular categories in \Cref{sec:modular-back}. The next section, \Cref{sec:Nichols}, contains necessary background material on Nichols algebras. We use the approach of realizing diagonal type Nichols algebras $\BB_\qs$ as Hopf algebra objects in braided categories $\cA_\qs$ of comodules over abelian groups in order to define their braided Drinfeld doubles. The general structural results on classification of spherical structures for the bosonizations $\BB_\qs\rtimes \Bbbk[G]$ and ribbon structures for their doubles are obtained in \Cref{sec:spher-ribbon}, which contains the technical core of the paper.

In \Cref{sec:superquantum}, we give presentations for braided Drinfeld doubles of Nichols algebras of type Super A, which define the finite-dimensional quantum groups $\ru_q(\fr{sl}_{r,\mJ})$. \Cref{sec:modularresults} contains the main result of the paper, which characterizes which quantum groups $\ru_q(\fr{sl}_{r,\mJ})$ of type Super A give modular tensor categories of representations. In order to give a closed formula for the universal $R$-matrix and to aid the study of the representation theory of the rank-two examples $\ru_q(\fr{sl}_{2,\mI})$, we introduce a basis of divided powers. The representation theory of $\ru_q(\fr{sl}_{r,\mJ})$, with particular focus on the rank-two case, is studied in \Cref{sec:reps}. Here, we give bases for the simple modules in the rank-two case and compute composition series and tensor product decompositions for the standard modules in the highest weight category of graded modules. We conclude \Cref{sec:reps} with some partial results on the semisimplification of the category of $\ru_q(\fr{sl}_{2,\mI})$-modules.
We conclude the paper with the above-mentioned applications to knot theory in \Cref{sec:linkinvariants}.

\subsection{Acknowledgements}
The authors thank Iv\'an Angiono, Azat Gainutdinov, and Simon Lentner for helpful discussions. We further thank an anonymous reviewer for very helpful comments that improved the manuscript. 
R.~Laugwitz was supported by a Nottingham Research Fellowship. G.~Sanmarco is partially supported by an AMS-Simons Travel Grant and kindly acknowledges the warm hospitality of the University of Nottingham, where part of this work was carried out during a research visit.

\section{Modular tensor categories}
\label[section]{sec:modular-back}

Throughout the paper, we work over an algebraically closed field $\Bbbk$ of characteristic zero. The group of invertible elements is $\Bbbk^\times$; the order of an element $q$ is $\ord (q)$.
The monoidal categories considered here are finite tensor categories, i.e. finite abelian $\Bbbk$-linear categories with a monoidal structure exact in both variables such that every object has a left and right dual, unless otherwise stated. We refer the reader to \cite{EGNO} for background material on this class of categories.

\subsection{Ribbon and non-semisimple spherical tensor categories}
\label[section]{sec:ribbonback}

A braided tensor category $(\cC, \otimes, \one, c)$  is \emph{ribbon} (or \emph{tortile}) if it comes equipped with a \emph{twist} (or \emph{ribbon structure}), i.e., a natural isomorphism $\theta_X\colon X \overset{\sim}{\to} X$ which satisfies 
\begin{equation}\label{eq:ribbon-def}
\theta_{X \otimes Y} = (\theta_X \otimes \theta_Y)  \circ c_{Y,X} \circ c_{X,Y}\qquad \text{and}\qquad (\theta_X)^* = \theta_{X^*}
\end{equation}
for all objects $X,Y$ in $\cC$. A \emph{functor of ribbon categories} is a functor $F\colon \cC\to \cD$ of braided  tensor  categories such that $F(\theta_X^{\cC})=\theta^{\cD}_{F(X)}$, for any $X\in \cC$. We refer to \cite{BK}*{Chapter~2}, \cite{EGNO}*{Sections~8.9--8.11}, or  \cite{TV}*{Section~3.3} for more details.

For a ribbon category $\cC$, consider the \emph{Drinfeld isomorphism}:
\begin{align}\label{Driniso}
\phi_X=(\ide_{X^{**}}\otimes \ev_X)(c_{X^*,X^{**}}\otimes \ide_{X})(\coev_{X^*}\otimes\ide_X)\colon X \overset{\sim}{\to} X^{**}.
\end{align}
It can be used to define a pivotal structure on $\cC$ via
\begin{align}\label{eq:ribbonpivotal} 
j_X:=\phi_X\theta_X\colon X\isomorph X^{**}.
\end{align}

The category $\lmod{H}$ for a finite-dimensional quasitriangular $\Bbbk$-Hopf algebra  $H = (H,R)$ is a ribbon category if and only if $H$ is a ribbon Hopf algebra \cite{Majid}*{Corollary~9.3.4}. Here, $H$ is a \emph{ribbon Hopf algebra} if there exists a central invertible element $v\in H$ satisfying
\begin{align}\label{ribbonelement}
\Delta(v)=(R_{21}R)^{-1}(v\otimes v), \qquad \varepsilon(v)=1, \qquad S(v)=v.
\end{align}
In this case, $v$ is a \emph{ribbon element}, and the ribbon twist  in $\lmod{H}$ is given by the action of $v^{-1}$.

In fact, ribbon elements for a Hopf algebra can be classified using certain grouplike elements.
Given a quasi-triangular Hopf algebra $H$, denote by $u=S(R^{(2)})R^{(1)}$ the \emph{Drinfeld element}. Note that for $\cC=\lmod{H}$, the Drinfeld isomorphism $\phi_X$ from \Cref{Driniso} is given by the action of $u$ followed by the canonical pivotal structure $\tau_X\colon X \stackrel{\sim}{\to} X^{**}$ of $\Vect$.

\begin{lemma}[{\cite{Rad}*{Theorem~12.3.6}}]\label[lemma]{lem:ribbon-elts}
Let $\cC=\lmod{H}$ for $H$ a quasi-triangular Hopf algebra. Then there is a bijection between the set of ribbon structures on $\cC$ and the set of elements
$$\mathrm{RPiv}(H)=\Set{l\in G(H)\,\vert \, l^2=g, S^2(h)=lhl^{-1}, \forall h\in H },$$
where $g=uS(u^{-1})$. The bijections maps $l\in \mathrm{RPiv}(H)$ to the 
 ribbon element $v=l^{-1}u=ul^{-1}$ for $H$.
\end{lemma}

We note that the set of grouplike elements of $H$ is a linearly independent subset of $H$. Hence, the set $\mathrm{RPiv}(H)$ is linearly independent. In this paper, we will find this set $\mathrm{RPiv}(H)$ corresponding to the set of ribbon structures for certain Hopf algebras (see \Cref{theorem:ribbon} and \Cref{thm:modularunique}).

Next, we recall a theorem of Kauffman--Radford describing the set of ribbon elements for the Drinfeld double of a finite-dimensional Hopf algebra $H$. For this, fix a non-zero left integral $\Lambda$ for $H$ and a non-zero right integral $\lambda$ of $H^*$.
We recall the \emph{distinguished grouplike elements}  
$g_{H}\in G(H)$ and $\alpha_{H}\in G(H^*)$ which are uniquely determined by the equations 
\begin{gather}\label{eq:distinguished}
p\lambda = \ev(p\otimes g_{H})\lambda,\qquad \text{for all $p\in H^*$},\qquad 
\alpha_{H}(h)\Lambda=\Lambda h,\qquad  \text{for all $h\in H$}.
\end{gather}

\begin{theorem}[\cite{KR93}*{Theorem~3}]\label{thm:KR}
Let $g_H\in H$ and $\alpha_H\in H^*$ be the distinguished grouplike elements.
Then  there is a bijection between the set 
\begin{gather*}
\left\{(\zeta,a)\in G(H^*)\times G(H)\;\big|\; \zeta^2=\alpha_H, \;  a^2=g_H, \; S^2(h)=\zeta^{-1}(h_{(1)})ah_{(2)}a^{-1}\zeta(h_{(3)}), \; \forall h\in H \right\},
\end{gather*}
and the set of ribbon elements of the Drinfeld double, $\Drin(H)$, cf. \eqref{ribbonelement}.  
The bijection is given by sending a pair $(\zeta,a)$ to $u_{\Drin(H)}(\zeta^{-1} \otimes a^{-1})$.
\end{theorem}
Given a ribbon element $v$ of $H$, then, by \cite{Rad}*{Theorem 12.3.6}, any other ribbon element is of the form $zv$ for $z$ an element of 
\begin{equation}\label{eq:Z}
    Z=\Set{z\in Z(H)\cap G(H) \;|\; z^2=1}.
\end{equation}

\Cref{thm:KR} was generalized to the center of any finite tensor category by Shimizu.
In any finite tensor category $\cC$ there is a distinguished invertible object $D$ equipped with a natural isomorphism
\begin{align}\label{eq:Radford-iso}
\xi=\lbrace \xi_X\colon D\otimes X\isomorph X^{4*}\otimes D\rbrace_{X\in \cC},
\end{align}
see \cite{EGNO}*{Sections~6.4, 7.18, 7.19}. If $\cC=\lmod{H}$ then $D=\Bbbk v$ with $h\cdot v=\alpha_H^{-1}(h)v$ for the distinguished grouplike element $\alpha_H\in H^*$ from \Cref{eq:distinguished}, and 
$\xi$ is given by acting with $g_H$.  

\begin{definition}[\cite{LW2}*{Definition 3.3}] \label[definition]{def:sqrtD}
For a finite tensor category $\cC$ and the pair $(D, \xi)$ from \eqref{eq:Radford-iso}, we define $\mathsf{Sqrt}_{\cC}(D,\xi)$ as the set of equivalence classes of pairs $(V,\sigma)$ where $V\in\cC$ and
\[\sigma=\lbrace\sigma_X \colon V\otimes X \to X^{**} \otimes V \rbrace_{X\in \cC} \]
is a natural isomorphism, in $X$, compatible with the tensor product, for which there exist an isomorphism $\nu \colon V^{**}\otimes V \stackrel{\sim}{\to}D$ such that the following diagram commutes 
\begin{align*}
\vcenter{\hbox{
\xymatrix{
V^{**}\otimes V \otimes X \ar[rrr]^{\ide_{V^{**}}\otimes \sigma_X}\ar[d]_{\nu\otimes \ide_X}&&& V^{**}\otimes X^{**}\otimes V \ar[rrr]^{\sigma_{X}^{**}\otimes \ide_{V}} &&& X^{4*}\otimes V^{**} \otimes V \ar[d]^{\ide_{X^{4*}}\otimes \nu}\\
D \otimes X\ar[rrrrrr]^{\xi_X} &&&&&& X^{4*}\otimes D.
}}}
\end{align*}
\end{definition}
In this setup, Shimizu's generalization of \Cref{thm:KR} takes the following form.

\begin{theorem}\cite{Shi2}
If $\cC$ is a finite tensor category, then there is a bijection between the set  $\mathsf{Sqrt}(D,\xi)$ and the set of ribbon structures on $\cZ(\cC)$.
\end{theorem}

We recall the following definition of sphericality adapted to allowing non-semisimple tensor categories from \cite{DSS}*{Definition 3.5.2}.
In case when $\cC$ is semisimple
this recovers the definition from \cite{BW} in terms of left and right traces being equal, see \cite{DSS}*{Proposition 3.5.4}.

 \begin{definition}  \label[definition]{def:spherical}
A pivotal finite tensor category $(\cC, \otimes, \one, j)$ is \emph{spherical} if  there is an isomorphism $\nu\colon \one \overset{\sim}{\to} D$ making the following diagram commute:
\begin{align}\label{eq:spherical}
\vcenter{\hbox{
\xymatrix{
X\ar[rr]^{j_{X}^{**}\; j_{X}}\ar[d]_{\nu\otimes \ide_X}&&X^{4*}\ar[d]^{\ide_{X^{4*}}\otimes \nu}\\
D \otimes X\ar[rr]^{\xi_X} && X^{4*}\otimes D.
}}}
\end{align}
\end{definition}

The results of \cite{KR93} recalled in \Cref{thm:KR} imply the following parametrization of spherical structures for representation categories over a  finite-dimensional Hopf algebra $H$. 

\begin{proposition} \label[proposition]{prop:Hspherical} The tensor category $\cC=\lmod{H}$ is spherical  if and only if  $\alpha_H=\varepsilon$ and 
\begin{align*}
    \mathsf{SPiv}(H):= \big\{ a\in G(H)\mid a^2=g_H,\; S^2(h)=aha^{-1}, \text{ for all $h\in H$}\big\}\neq \varnothing.
\end{align*}
In this case, there is a bijection between the set of pivotal structures $j$ on $\cC$ giving spherical structures and $\mathsf{SPiv}(H)$.
\end{proposition}

In this paper, we will construct new examples of spherical categories in \Cref{theorem:spherical}.
We recall that $\cC$ is unimodular if $D=\one$ and that any unimodular ribbon tensor category is spherical by \cite{LW2}*{Proposition~3.13}. Note that \cite{Rad}*{Proposition~12.3.3} shows that, indeed, if $H$ is unimodular, then $g_H=uS(u^{-1})$.

We conclude this section by including results on ribbon structures on Deligne tensor products required later in the paper.
Given ribbon structures $\theta^{\cC}$, $\theta^\cD$ on finite braided tensor categories $\cC$ and $\cD$, it follows that
$$\theta_{X\boxtimes Y}:=\theta^{\cC}_X\boxtimes \theta^{\cD}_Y$$
extends, by the universal property of the Deligne tensor product, cf. \cite{DSS2}*{Definition~3.2}, to a ribbon structure on $\cC\boxtimes \cD$.

\begin{lemma}\label[lemma]{lem:ribHopftensor}
Let $\cC$ and $\cD$ be finite braided tensor categories. 
Then every ribbon structure on $\cC\boxtimes \cD$ is obtained from a unique pair of ribbon structures on $\cC$ and $\cD$ as above.
\end{lemma}
\begin{proof}
Given a ribbon structure $\theta$ on $\cC\boxtimes \cD$, we can evaluate it on the object $X\boxtimes \one$ to define $\theta^\cC_X$ to be the unique element in 
$\Hom_\cC(X,X)$ such that $\theta_{X\boxtimes \one}=\theta^\cC_X\boxtimes \ide_\one$. This produces a ribbon structure $\theta^\cC$ on $\cC$. Proceeding similarly, we find a ribbon structure $\theta^\cD$ on $\cD$. 

Now, $X\boxtimes\one$ centralizes $\one \boxtimes Y$, i.e. $c^2_{X\boxtimes\one,\one \boxtimes Y}=\ide$. Thus, 
$$\theta_{X\boxtimes Y}=\theta_{X\boxtimes \one}\otimes \theta_{\one\boxtimes Y}=\theta^\cC_X\boxtimes \theta^\cD_Y.$$
Note that $\theta^\cC$ and $\theta^\cD$ uniquely determine $\theta$ from the requirement that $\theta_{X\boxtimes Y}=\theta^\cC_X\boxtimes \theta^\cD_Y$.
\end{proof}

\subsection{Modular tensor categories}
\label[section]{sec:modular}
A generalization of the definition of a modular category to not necessarily semisimple categories can be found in Kerler--Lyubashenko \cite{KL}, see also Shimizu \cite{Shi1} for equivalent conditions. Here, non-degeneracy of the $S$-matrix, as required for semisimple modular categories is replaced by the condition that the \emph{M\"uger center} $\cC'$ of $\cC$ is equivalent to $\Vect$ as a tensor category \cite{EGNO}*{Definition~8.19.2 and Theorem~8.20.7}. In this case, we say that $\cC$ is \emph{non-degenerate}. A braided finite tensor category $\cC$ with braiding $c$ is non-degenerate if and only if it is factorizable, i.e. the natural functor
$$\cC\boxtimes \cC^\rev\to \cZ(\cC), \quad X\boxtimes Y\mapsto (X,c_{X,-})\otimes (Y,c_{-,Y}^{-1})$$
gives an equivalence of braided tensor categories
\cite{Shi1}*{Theorem~4.2} to the \emph{Drinfeld  center} (or, \emph{monoidal center}) $\cZ(\cC)$. Here, $\cC^\rev$ denotes the braided opposite $(\cC, c_{Y,X}^{-1}\colon X\otimes Y \stackrel{\sim}{\to} Y\otimes X)$.

Next, we recall the definition of modularity for not necessarily semisimple categories.
\begin{definition}
 A braided finite tensor category is called {\it modular} if it is non-degenerate and ribbon.
\end{definition}

In particular, the category $\lmod{H}$ for $H$ a finite-dimensional quasi-triangular Hopf algebra over $\Bbbk$ is modular if and only if $H$ is ribbon and factorizable \cite{EGNO}*{Proposition~8.11.2 and Example~8.6.4}. 

If $\cC$ is a finite tensor category, then $\cZ(\cC)$ is factorizable \cite{EGNO}*{Proposition 8.6.3} and hence non-degenerate. Thus, the conditions in \cite{Shi2}*{Theorem~5.4} ensuring that $\cZ(\cC)$ is a ribbon category imply that $\cZ(\cC)$ is a modular category.
In particular, the center $\cZ(\cC)$ of a spherical finite tensor category $\cC$ (cf. \Cref{def:spherical}) is modular by \cite{Shi2}*{Theorem~5.10}. In the semisimple case, this is due to \cite{Mue2}*{Theorem 2}. These results were extended to \emph{relative} monoidal centers in \cite{LW2}*{Theorem~4.14}. We briefly discuss these results in the following section.

\subsection{Relative Drinfeld centers}
\label[section]{sec:relcen}
Modularity of relative centers, and more generally of M\"uger centralizers, was investigated in \cite{LW2}. Here, we briefly recall the general results related to this paper's content.

Consider a braided monoidal category $\cA$ with braiding  $c_{X,Y} \colon X\otimes Y \stackrel{\sim}{\to} Y\otimes X$. Recall that an $\cA$-\emph{central} structure on a monoidal category $\cC$ is a faithful braided monoidal functor $G\colon \cA^\rev \to \cZ(\cC)$, $A\mapsto (A,c^{G(A)})$. The \emph{relative center} $\cZ_{\cA}(\cC)$ is defined as the full subcategory of $\cZ(\cC)$ that centralizes the image of the central structure.

\begin{definition}
The relative center $\cZ_{\cA}(\cC)$ is the monoidal full subcategory of $\cZ(\cC)$ containing all objects $(V,c^V\colon V\otimes \ide_{\cC} \stackrel{\sim}{\to}\ide_{\cC}\otimes V)$ such that the half braiding $c^V$ is compatible with the $\cA$-central structure, i.e, $c^{G(A)}_{V}\circ c^{V}_{G(A)}=\ide$ for all $A\in \cA^\rev$.
\end{definition}

The braiding of $\cZ(\cC)$ restricts to a braiding on $\cZ_{\cA}(\cC)$. Moreover, if $\cA$ is a finite braided tensor category, $\cC$ is a finite tensor category and the central functor $G$ is a tensor functor, then $\cZ_{\cA}(\cC)$ is a finite braided tensor category. 
The next result gives sufficient conditions for modularity of the relative center. 
Recall the distinguished invertible element and the Radford isomorphism $(D,\xi)$ from \eqref{eq:Radford-iso} and the set  $\mathsf{Sqrt}_{\cC}(D,\xi)$ from \Cref{def:sqrtD}.

\begin{theorem}[\cite{LW2}*{Theorems~1.4, 1.5}]\label{thm:LW-center-relcenter}
Let $\cC$ be a finite tensor category with $\cA$-central structure $G\colon \cA^\rev \to \cZ(\cC)$. Assume further that $\cA$ is non-degenerate and that $G\left(\cA^\rev\right)$ is closed under subquotients and finite direct sums. If  $\mathsf{Sqrt}_{\cC}(D,\xi_D) \ne \varnothing$, then the relative center $\cZ_{\cA}(\cC)$ is modular and we have an equivalence of ribbon categories 
$\cZ(\cC)\simeq\cZ_{\cA}(\cC)\boxtimes\cA^\rev$.
\end{theorem}
The above result justifies working with relative centers as they appears as a factor in a Deligne tensor product decomposition of $\cZ(\cC)$, omitting the extra copy of $\cA^\rev$. This way, $\cZ_{\cA}(\cC)$ might be a \emph{prime} modular category (cf. \cite{LW2}*{Section~4.4}), i.e., not decomposable as a Deligne tensor product of proper modular tensor subcategories, while $\cZ(\cC)$ is not prime.

\section{Nichols algebras and their braided Drinfeld doubles}
\label[section]{sec:Nichols}

In this section, we include the necessary definitions from the theory of Nichols algebras with focus on type Super A, and fix the setup to be used throughout the paper. We then realize such Nichols algebras in base braided categories of comodules over an abelian group of the same rank in \Cref{prop:BBqs-graded}. Next, we define braided Drinfeld doubles of Nichols algebras, whose representation categories are equivalent to the relative Drinfeld center of the category of modules over the bosonization of the Nichols algebra (see \Cref{prop:DrinYD}). 

The following notation will be used in this and the following sections. If  $k < r$ are non-negative integers, put $\mI_{k, r} = \Set{n\in \mZ \,\middle|\, k\le n \le r }$, and $\mI_{r} = \mI_{1, r}$. We use $\mI=\mI_r$ when there is no possible confusion.
The canonical basis of $\mZ^{\mI}:=\mZ^{r}$ is denoted by $(\alpha_i)_{i\in \mI_{r}}$.
 
\subsection{Nichols algebras of diagonal type}\label[section]{sec:Nichols-diagonal}

Nichols algebras constitute a large class of examples of Hopf algebras in braided tensor categories and are a central tool in the classification of pointed Hopf algebras, see e.g. \cites{Hec1, AA,HS}. Most notably, the nilpotent parts $u_q(\fr{n}^+)$ of small quantum groups appear as examples of finite-dimensional Nichols algebras of diagonal Cartan type. In this section, we recall basic definitions of the theory of Nichols algebras of diagonal type, with focus on  type Super A.  The notation used throughout this paper for Nichols algebras of type Super A is detailed in \Cref{subsubsec:setup}.

The \emph{Nichols algebra} $\BB_\qs$ of a matrix $\qs = (q_{ij}) \in (\Bbbk^\times)^{\mI \times \mI}$ is a $\mZ^\mI$-graded braided Hopf algebra \cites{AS,HS,Tak}. The bialgebra structure of $\BB_\qs$ can be constructed starting from a braided vector space $V^\qs$ with basis $\{x_i\}_{i \in \mI}$ and braiding $c^\qs (x_i \otimes x_j) = q_{ij} x_j \otimes x_i$. The algebra $\BB_\qs$ admits a PBW-type basis with $\mZ^\mI$-homogeneous generators. The set of \emph{positive roots} of $\BB_\qs$ is the collection $\varDelta_+^{\qs}$ consisting on the $\mZ^\mI$-degrees of these PBW generators, counted with multiplicities. The matrix $\qs$ is \emph{arithmetic} if $\varDelta_+^{\qs}$ is finite; in such case $\varDelta_+^{\qs}$ does not depend on the choice of the PBW generators and $\BB_\qs$ has finite Gelfand--Kirillov dimension.

A fundamental step in the theory of Nichols algebras (and their role in the classification of pointed Hopf algebras) was achieved in \cite{Hec2}, where Dynkin diagrams of arithmetic matrices $\qs$ were classified. 
Remember that the \emph{(generalized) Dynkin diagram} of $\qs = (q_{ij}) \in (\Bbbk^\times)^{\mI \times \mI}$ is the decorated graph with vertices $i\in\mI$ labelled by $q_{ii}$, and there is an edge between distinct vertices $i$ and $j$ if $\widetilde{q}_{ij}\coloneq q_{ij}q_{ji} \ne 1$; such edge is labelled by $\widetilde{q}_{ij}$. Later on, this classification was organized in Lie-theoretic terms in \cite{AA}, to which we refer for generalities on Nichols algebras of diagonal type. 

Let us fix some notation associated to a matrix $\qs$. We have a bilinear form 
\begin{align}\label{eq:bilinear-form}
    \qs \colon \mZ^\mI \times \mZ^\mI \to \Bbbk^\times, &&  \qs(\alpha_i,\alpha_j)=q_{ij},
\end{align}
where $(\alpha_i)$ is the canonical basis of $\mZ^\mI$. Given $\beta \in\mZ^\mI$ write $N_\beta\coloneq \ord \qs(\beta, \beta) \in \mN \cup \{\infty\}$. 

If $\qs$ is arithmetic then $\BB_\qs$ has a PBW basis of the form $\Set{\prod_{\beta \in \varDelta_+^{\qs}} x_{\beta}^{n_\beta}\middle| 0\le n_\beta< N_\beta}$, where each $x_\beta$ is a \emph{root vector}, and the product is taken according to a \emph{convex} order on the set of positive roots.
Assume moreover that $\BB_\qs$ is finite dimensional. In this case the top $\mZ$-degree $\BB_\qs^{\operatorname{top}}$ of $\BB_\qs=\bigoplus_{i=0}^{\operatorname{top}} \BB^i_\qs$ is one dimensional, it coincides with the subspace of left (and right) integrals, and is $\mZ^\mI$-homogeneous of degree $\sum_{\beta \in \varDelta_+^{\qs}}(N_\beta -1) \beta$.

\subsubsection{Defining relations}\label[section]{subsubsec:relations}
Once the classification of arithmetic braiding matrices was achieved, the next crucial problem in the theory was to find an explicit presentation of the corresponding Nichols algebras by generators and relations. Indeed, $\BB_\qs$ is a braided Hopf quotient of the tensor algebra of the vector space $V^\qs$ with basis vectors $\{x_i\}_{i \in \mI}$ by an ideal $\II_\qs$ for which Angiono found a set of $\mZ^\mI$-homogeneous generators in \cites{Ang1, Ang2}.

We recall the $\qs$-commutators, which are employed to exhibit both root vectors and defining relations of $\BB_\qs$. For elements $x$, $y$ in the tensor algebra of degrees $\alpha$ and $\beta$, respectively, we denote 
\begin{align}\label{eq:q-comm}
    [x,y]_\qs = xy - \qs(\alpha, \beta) yx,
\end{align}
where $\qs(\alpha,\beta)$ was defined in \eqref{eq:bilinear-form}.
For indexes $i_1, \dots, i_k$ we define recursively 
$   x_{i_1 \dots i_k} = [x_{i_1}, x_{i_2 \dots i_k}]_\qs$ and for $i<j$ put $x_{(ij)}=x_{i i+1 \dots j}$. 
To illustrate, we compute
\begin{align}\label{eq:general-qserre}
\begin{aligned}
x_{ij} &= x_i x_j - q_{ij}x_j x_i, \\
x_{iij}&=x_ix_{ij} -q_{ii}q_{ij}x_{ij}x_i= x_i^2 x_j - q_{ij} (1+q_{ii})x_i x_j x_i + q_{ij}^2 q_{ii} x_j x_i^2,
\end{aligned}
&&i\ne j \in \mI.
\end{align}

\subsubsection{Realizations}\label{subsubsec:realizations}
Fix a finite abelian group $G$ and $\qs = (q_{ij}) \in (\Bbbk^\times)^{\mI \times \mI}$. A \emph{realization} of $\qs$ over $G$ is a family of generators $(g_i)_{i \in \mI}$ of $G$ and a family of characters $(\chi_i)_{i \in \mI}$ such that $\chi_i(g_j) = q_{ji}$ for all $i,j \in \mI$. In this case $\BB_\qs$ is a Hopf algebra in the braided category $\lYD{\Bbbk G}$ of Yetter--Drinfeld modules over the group algebra $\Bbbk G$ (these are just $G$-graded $G$-modules as $G$ is abelian). Via bosonization we get a Hopf algebra $\BB_\qs \rtimes \Bbbk G$.
By \cite{AA}*{Proposition 2.42} the distinguished group-like element of $\BB_\qs \rtimes \Bbbk G$ is the image $g_{\ell}$ of $\sum_{\beta \in \varDelta_+^{\qs}}(N_\beta -1) \beta$ under the group map $\mZ^\mI \to G$ determined by $\alpha_i\mapsto g_i$.

For the purpose of defining braided Drinfeld doubles of $\BB_\qs$ later, let us realize $\qs$ in the category of $\Bbbk G$-\emph{comodules}, which embeds into the larger category of Yetter--Drinfeld modules to recover the above realization.
For this, consider the braiding $\Psi$ on the category $\lcomod{\Bbbk G}$ via the dual $R$-matrix 
\begin{align}\label{eq:supera-Rmatrix}
r_{\qs} \colon \Bbbk G \otimes  \Bbbk G \to \Bbbk^\times, && r_{\qs}(g_i \otimes g_j)\coloneq q_{ji}, && i,j \in \mI.
\end{align}
In particular, the braiding on $\lcomod{\Bbbk G}$ is given by
\begin{equation}\label{eq:braiding-graded}
    \Psi_{V,W}(v\otimes w)=r_\qs(w^{(-1)}\otimes v^{(-1)})w^{(0)}\otimes v^{(0)}.
\end{equation}
For instance, if $v_i$ has degree $g_i$ and $w_j$ has degree $g_j$, $$\Psi_{V,W}(v_i\otimes w_j)=q_{ij}w_j\otimes v_i.$$

\begin{definition}\label[definition]{def:Aq}
We use $\cA_{\qs}$ to denote the pointed braided fusion category $\lcomod{\Bbbk G}$ with braiding given by $r_\qs$. 
\end{definition}

Consider the symmetric pairing 
\begin{align}\label{eq:supera-bilinear-form}
b_{\qs} \colon  G \times   G \to \Bbbk^\times, && b_{\qs}(g_i,g_j)\coloneq r_{\qs}(g_i \otimes g_j)r_{\qs}(g_j \otimes g_i), && i,j \in \mI.
\end{align}

\begin{lemma}[{\cite{EGNO}*{Example 8.13.5}}] The $S$-matrix of $\cA_{\qs}$ is $\left(b_{\qs}(g,h)\right)_{g,h \in G}$. Hence, $\cA_\qs$ is non-degenerate if and only if the symmetric pairing $b_{\qs}$ is non-degenerate. \qed
\end{lemma}

\begin{remark}\label[remark]{remark:choice}
We note that $\cA_\qs$ depends on the concrete choice of the matrix $\qs$ rather than the generalized Dynkin diagram of the Nichols algebra (which determines the Nichols algebra as an algebra). For example, let $q$ be a root of unity of even order $N$ and $q^{1/2}$ be a square root of $q$, which is necessarily a primitive $2N$-th root of unity. Then the matrices
$$\qs_1=\begin{pmatrix}
-1&q\\
1&-1
\end{pmatrix}, \qquad \qs_2=\begin{pmatrix}
-1&q^{1/2}\\
q^{1/2}&-1
\end{pmatrix}$$
lead to non-equivalent braided tensor categories $\cA_{\qs_1}$ and $\cA_{\qs_2}$ of comodules over $G=\mZ_{2N}^2$, which are both degenerate. The dual $r$-matrix $r_{\qs_1}$ can be defined over the quotient $G'=\mZ_{N}^2$ and gives a non-degenerate braiding on $G'$-comodules. This is not possible for $r_{\qs_2}$ since $(q^{1/2})^N=q^{N/2}\neq 1$ and hence the dual $R$-matrix $r_\qs$ of \eqref{eq:supera-Rmatrix} cannot be defined over $G'$. 

If two different matrices $\cA_{\qs_1}$ and $\cA_{\qs_2}$ are defined over a group $G$, then non-degeneracy of the braiding does not depend on this choice since the symmetric pairing $b_\qs$ of \eqref{eq:supera-bilinear-form}
only depends on the fixed parameters $\widetilde{q}_{ij}= q_{ij}q_{ji}$. 
\end{remark}

Recall the definition of quadratic forms over abelian groups and their associated bilinear forms from, e.g. \cite{EGNO}*{Section 8.4}. 

\begin{proposition}\label[proposition]{prop:ribbon-Aq}
Let $G$ be an abelian group together with a bilinear form $\qs\colon G\times G \to \Bbbk^\times$. Then $\qs$ determines a braided monoidal category $\cA_\qs$. The set of ribbon structures for $\cA_\qs$ is parametrized by the set of quadratic forms $\vartheta\colon G\to \Bbbk^\times$ such 
\begin{equation}\label{eq:ribbon-group}
    b_\qs(g,h)=\frac{\vartheta(gh)}{\vartheta(g)\vartheta(h)}.
\end{equation}
\end{proposition}
\begin{proof}
Given a ribbon structure $\theta$ on $\cA_\qs$, then $\theta_{\Bbbk_g}$ is a multiple of the identity, define $\vartheta(g)$ via
$$\theta_{\Bbbk_g}=\vartheta(g)\ide_{\Bbbk_g}.$$
Now the defining axioms of a ribbon structure, see \Cref{eq:ribbon-def} imply \Cref{eq:ribbon-group} and
$\vartheta(g)=\vartheta(g^{-1})$.
\end{proof}

\begin{remark}\label[remark]{rem:group-ribbon}
Let $G$ be a finite abelian group with a bilinear form $r$ which defines a braiding on  $\cC=\lcomod{\Bbbk G}$. By \cite{EGNO}*{Remark~8.10.4}, $\theta(g)=b(g,g)$ defines a quadratic form and a ribbon structure on $\cC$. Now, by \cite{Rad}*{Theorem~12.3.6}, see \Cref{eq:Z}, the number of ribbon structures on $\cC$ is given by 
$$|Z|=\lvert\Set{g\in G\,|g^2=1\,}\rvert.$$
\end{remark}

\begin{remark}
Two quadratic forms $\vartheta_1$ and $\vartheta_2$ are equivalent if there exists a group automorphism $\phi\colon G\to G$ such that 
$$\vartheta_2(g)=\vartheta_1(\phi(g)).$$
In other words, $\vartheta_1$ and $\vartheta_2$ are in the same orbit under the action of $\Aut(G)$. Such an automorphism $\phi$ induces an equivalence of monoidal categories $$F_\phi\colon \cA_\qs\to \cA_\qs,\qquad  \Bbbk_g \mapsto \Bbbk_{\phi(g)},$$
together with a structural isomorphisms $\mu^{F_\phi}\colon F_\phi(V)\otimes F_\phi(W)\to F_\phi(V\otimes W)$ which is determined by a group $2$-cocycle $\mu\colon G\times G\to \Bbbk^\times$ such that 
$$\mu^{F_\phi}_{\Bbbk_g,\Bbbk_h}=\ide_{\Bbbk_{\phi(gh)}}\mu_{g,h}.$$
The equivalence $F_\phi$ is one of braided categories if and only if $\phi$ leaves the (fixed) braiding invariant, i.e., 
$$\mu_{g,h}r_\qs(g,h)=\mu_{h,g}r_\qs(\phi(g),\phi(h)), \qquad \text{for all $g,h\in G$}.$$
In this case, $F_\phi$ is an equivalence of ribbon categories between $\cA_\qs$ with ribbon structures $\vartheta_1$, respectively, $\vartheta_2$.

If $\phi=\ide_G$, the $2$-cocycle $\mu=\Set{\mu_{g,h}}$ is symmetric and hence trivial by \cite{EGNO}*{8.4.13}, thus we may disregard this additional data which does not impact the ribbon structure. 
\end{remark}

Using the dual $R$-matrix from \eqref{eq:supera-Rmatrix}, we may realize the Nichols algebra $\BB_\qs$ as a braided Hopf algebra in $\cA_\qs$. Indeed, there is a full and faithful functor of braided monoidal categories
\begin{align*}
    \Phi\colon \cA_\qs \longrightarrow \lYD{\Bbbk G},
\end{align*}
which sends a $\Bbbk G$-comodule $V$ to itself,
preserving the coaction, and using the $\Bbbk G$-action given by 
\begin{equation}
g\cdot v=r_\qs(|v|\otimes g) v,
\end{equation}
if $v$ is homogeneous of degree $|v|$. This functor is a right inverse to the forgetful functor. 

\begin{proposition}[{\cite{LW2}*{Lemma 5.7}}]\label[proposition]{prop:BBqs-graded}
The Nichols algebra $\BB_\qs$ is a braided Hopf algebra  in $\cA_\qs$, where the generator $x_i$ has degree $g_i\in G$. \qed
\end{proposition}

We can apply $\Phi$ to $\BB_\qs$ to obtain the usual realization of $\BB_\qs$ as a braided Hopf algebra in $\lYD{\Bbbk G}$.
The action of the group $G$ is given on generators by 
\begin{equation}
g_i\cdot x_j=r_\qs(g_j\otimes q_i) x_j=q_{ij}x_j.    
\end{equation}

\subsubsection{Weyl equivalence} \label{subsubsec:groupoid}
The action of the Weyl group on (the positive part of) the quantum group can be generalized to the setting of Nichols algebras, but a subtle difference needs to be taken into consideration. For details in the following construction, we refer to \cite{Hec3}. 

Given a braiding matrix $\qs=(q_{ij})_{i, j \in \mI}$ as above with finite dimensional Nichols algebra, one obtains a \emph{generalized} Cartan matrix $C^{\qs}$ by setting $c_{ii}=2$ and
\begin{align}\label{eq:Cartan-matrix}
c_{ij}^{\qs}&=-\min \Set{m \in \mN_0 \middle| (m+1)_{q_{ii}}(q_{ii}^mq_{ij}q_{ji} -1)=0}, &i\ne j & \in \mI,
\end{align}
where we use the notation
\begin{equation}(m)_{q}=\begin{cases} 1+q+q^2+\ldots + q^{m-1}, & \text{if $m\geq 1$}\\
0, & \text{if $m=0$}.
\end{cases}
\end{equation}
For each $i\in \mI$, one can now define a reflection $s^{\qs}_i \colon \mZ^{\mI} \to \mZ^{\mI}$ given by $s^{\qs}_i(\alpha_j)=\alpha_j - c_{ij}^{\qs} \alpha_i$ for $j\in\mI$. This reflection in turn gives rise to a new braiding matrix $\rho_i(\qs)$ with entries 
\begin{align}\label{eq:reflection-matrix}
\rho_i(\qs)_{kj} &= \qs (s^{\qs}_i(\alpha_k), s^{\qs}_i(\alpha_j)), &k, j & \in \mI.
\end{align}
In general, $\rho_i(\qs)$ will differ from $\qs$; however, we have an algebra isomorphism
\begin{align*}
T_i \colon \Drin\big(\BB_{\qs}\rtimes\Bbbk\mZ^\mI\big) \to \Drin\big(\BB_{\rho_i(\qs)}\rtimes \Bbbk\mZ^\mI\big).
\end{align*}
The fact that these two braiding matrices might be different is the main reason why, for Nichols algebras, one needs to consider Weyl \emph{groupoids} rather than just groups. 

The \emph{Weyl class} of $\qs$ is then defined as the set of reflections $\rho_{i_1}\dots \rho_{i_k} (\qs)$ for arbitrary $i_{1},\dots i_k \in \mI$. Note that, even though the Drinfeld doubles of the Nichols algebras for $\qs$ and $\qs'$ in the same Weyl class are isomorphic as algebras, there are, in general, non-isomorphic as Hopf algebras.

\subsubsection{Nichols algebras of type Super A} \label[section]{subsec:NicholstypeA} We introduce a family of Nichols algebras related to  Lie superalgebras of type $A$, cf. \cite{AA}*{Section~5}. Recall that $r \geq 2$ and $\mI = \mI_r = \{1, \dots, r\}$.

\begin{definition}\label[definition]{def:super-A}
Consider a non-empty subset $\mJ\subset \mI$ and $q\in \Bbbk^\times$ with $q^2\ne 1$. We say that $\qs=(q_{ij})$ is of super-type $\superqa{r}{q}{\mJ}$ if it satisfies:
\begin{enumerate}
\item \label{item:def-super-A-1} $q_{r r}^2\widetilde{q}_{r -1 r} = q$.
\item \label{item:def-super-A-2} If $\abs{i-j} \ge 2$ then $\widetilde{q}_{ij}=1$.
\item \label{item:def-super-A-3} If $i \in \mJ$ then $q_{ii}=-1$ and $\widetilde{q}_{i-1 i}= \widetilde{q}_{i i+1}^{\ -1}$.
\item \label{item:def-super-A-4} If $i \notin \mJ$ then $q_{ii}=q^{\pm 1}$ and $\widetilde{q}_{i-1 i}= q_{ii}^{-1}=\widetilde{q}_{i i+1}$.
\end{enumerate}
\end{definition}

\begin{remark}\label[remark]{rem:super-A-diagram} Note that the Dynkin diagram of any $\qs$ of type $\superqa{r}{q}{\mJ}$ is determined by $q$ and $\mJ$. Indeed, one first determines $q_{r r}$ and $\widetilde{q}_{r -1 r}$: if $r \in \mJ$ then $\widetilde{q}_{r -1 r}=q$ by \eqref{item:def-super-A-1}; if $r \notin \mJ$ then  $q_{r r}=q=\widetilde{q}_{r -1 r}^{\ -1}$ by \eqref{item:def-super-A-4} and \eqref{item:def-super-A-1}. Now we can use \eqref{item:def-super-A-3} and \eqref{item:def-super-A-4} to determine $q_{r-1  r-1}$ and $\widetilde{q}_{r -2 r-1}$. In this way the entire diagram is computed. In particular, we get $\widetilde{q}_{i i+1} = q^{\pm 1}$ for all $i<r$. 
\end{remark}

The set of positive roots of this Nichols algebra $\BB_\qs$ is
\begin{align}\label{eq:supera-positiveroots}
\varDelta_+^{\qs}=\Set{ \alpha_{ij} \middle| i\leq j \in \mI}, && \alpha_{ij}\coloneq\sum_{k=i}^j \alpha_k.
\end{align}
The \emph{parity} map associated to $\mJ$ is the group homomorphism
\begin{align}\label{eq:parity-map}
\bfp_\mJ \colon \mZ^{\mI} \to \{\pm 1\}, &&  \bfp_\mJ(\alpha_i)=-1 \iff i \in \mJ.
\end{align}
An element $\beta \in \mZ^\mI$ is \emph{even} if $\bfp_\mJ(\beta)=1$, and it is \emph{odd} otherwise. It is easy to see that $N_\beta = \ord q$ if $\beta$ is even and $N_\beta= 2$ if $\beta$ is odd. So the PBW basis of $\BB_\qs$ is of the form
\begin{align} \label{eq:supera-PBW}
\Set{\prod_{i\le j} x_{(ij)}^{n_{ij}}\middle| 0\le n_{ij}< \ord q \text { if } \alpha_{ij} \text { is even}, 0\le n_{ij}< 2 \text { if } \alpha_{ij} \text { is odd} }.
\end{align}

By \cite{AA}*{\S 5.1.9}, the Nichols algebra $\BB_\qs$ of type $\superqa{r}{q}{\mJ}$ for $\ord q=N$ is presented by generators $x_1,\dots, x_r$ subject to the relations
\begin{align}
\begin{aligned}
x_{ij}=0 \quad (i < j-1), \qquad x_{i i i\pm 1}=0  \quad (i \notin \mJ), \qquad x_i^2=0 \quad (i \in \mJ),\\
    [x_{(i-1 i+1)},x_{i}]_\qs =0 \quad (i \in \mJ), \qquad x_{(ij)}^N=0 \quad (\alpha_{ij} \text{ even root}).
\end{aligned}
\end{align}
For example, using the definition of the braided commutators \eqref{eq:q-comm}, for $i \in \mJ$ we have
\begin{align}\label{eq:relrank4}
\begin{aligned} \ 
[x_{(i-1 i+1)},x_{i}]_\qs= x_{i-1} x_i x_{i+1} x_i - q_{i-1, i} (1+\widetilde{q}_{i, i+1}) x_i x_{i-1} x_{i+1} x_i+ q_{i-1, i} q_{i+1, i}x_{i} x_{i-1} x_{i}x_{i+1} \\
\qquad \qquad + q_{i-1, i} q_{i-1, i+1}q_{i, i+1}x_{i+1} x_{i} x_{i-1}x_{i} 
 + q_{i-1, i}^2 q_{i-1, i+1} \widetilde{q}_{i, i+1} x_{i} x_{i+1} x_{i}x_{i-1}.
\end{aligned}
\end{align}

The collection of generalized Cartan matrices $C^\qs$ from \eqref{eq:Cartan-matrix} remains constant as $\qs$ specializes to any diagram of type $\superqa{r}{q}{\mJ}$. Actually, $C^\qs$  is the usual $r\times r$ Cartan matrix of type A, and the corresponding $s_i^\qs$ are the usual reflections of $\mZ^r$. This simplifies the computation of the Weyl classes in type Super A, which we briefly summarize next.

If $\mJ=\{i_1, \dots, i_k\}$ with $i_1< \dots<i_k$, we consider the quantity $S_\mJ=\left\vert \sum_{j=1}^{k} (-1)^j i_j\right\vert$. Weyl classes of Nichols algebras of type Super A of rank $r$ based on a root of unity $q$ are parametrized by integers $j$ with $1\leq j \leq \lfloor\frac{r+1}{2}\rfloor$. In case $j<\frac{r+1}{2}$, the class associated to $j$ contains all the Dynkin diagrams $\superqa{r}{q}{\mJ}$ with $S_\mJ=j$, and all the $\superqa{r}{q^{-1}}{\mJ}$ with $S_\mJ=r+1-j$. See \cite{AA}*{Section 5.1} for the class of $j=\frac{r+1}{2}$ when $r$ is odd.

\begin{example}
In rank $r=2$, for a fixed root of unity $q$ there is a unique Weyl class, corresponding to $j=1$. This class contains the diagrams $\superqa{2}{q}{\Set{1,2}}$, $\superqa{2}{q}{\Set{1}}$ and $\superqa{2}{q^{-1}}{\Set{2}}$. The following picture illustrates the action of the reflections $\rho_1, \rho_2$ from \eqref{eq:reflection-matrix} on each of these diagrams.
\[
\begin{tikzcd}
\superqa{2}{q}{\Set{1}} \ar[loop left, distance=3em, start anchor={[yshift=-1ex]west}, end anchor={[yshift=1ex]west}, "\rho_2"]  \ar[<->, "\rho_1"] {rr} && \superqa{2}{q}{\Set{1,2}} \ar[<->, "\rho_2"] {rr}
 && \superqa{2}{q^{-1}}{\Set{2}} \ar[loop right, distance=3em, start anchor={[yshift=1ex]east}, end anchor={[yshift=-1ex]east}, "\rho_1"]
\end{tikzcd} 
\]
\end{example}

\begin{remark}\label[remark]{rem:slnm-connection}
In contrast with the theory of simple Lie algebras, to a single simple Lie superalgebra one can associate a collection of distinct Dynkin diagrams, due to the existence of so-called odd reflections. However, there is a distinguished choice. For the Lie superalgebra $\mathfrak{sl}(m\vert n)$, the distinguished Dynkin diagram has $m+n-1$ vertices, with a unique odd vertex in the $m$-th position, see e.g. \cites{Kac-77} and \cite{Musson}*{Section 3.4.5}. The same choice is usually adopted in the study of quantized enveloping algebras \cites{KT, Yam}.

Among all Nichols algebras of type $\superqa{r}{q}{\mJ}$, later on we will become particularly interested in the case $\mJ=\mI$, i.e, all simple roots are odd. In this situation, $S_\mI= \frac{r}{2}$ if $r$ is even, and $S_\mI= \frac{r+1}{2}$ when $r$ is odd. This means that $\superqa{r}{q}{\mI}$ is Weyl equivalent to $\superqa{r}{q}{\Set{S_\mI}}$, a Nichols algebra of type Super A with a unique odd simple root in the position  $S_\mI$. Thus, we may say that the examples $\superqa{r}{q}{\mI}$, with $r$ even are of super-type $\fr{sl}(\tfrac{r}{2}|\tfrac{r}{2}+1)$.
\end{remark}

\subsubsection{Our setup}\label{subsubsec:setup}
Here we establish the assumption to be used throughout the rest of the paper:
\begin{align}\label{eq:setup}
\qs = (q_{ij})\text{ is of type }\superqa{r}{q}{\mJ}, 
&& 2<2n=N\coloneq \ord q  \text{ is even}, 
&& q_{ij}=1 \text{ for all } j\ne i, i+1.
\end{align} 
Hence $\dim \BB_\qs<\infty$ and $\qs$ admits a realization over the finite abelian group 
\begin{align}\label{eq:supera-group}
G=G_{r, N}\coloneq \langle g_1,\ldots, g_r \,\vert\, g_i^{N}=1 \ \forall i \rangle.
\end{align}

Since $q$ has even order $N=2n>2$, we can define the matrix
${\bf u}^{\qs}=(u_{ij}^{\qs})\in \mZ^{\mI\times\mI}$ by
\begin{align}\label{eq:supera-matrix-power2}
u_{ii}^{\qs}=
\begin{cases} 
n, &\text{ if } i \in \mJ,\\ 
1, &\text{ if }  q_{ii}= q, \\ 
- 1, &\text{ if }  q_{ii}= q^{-1}; \end{cases} &&
u_{ij}^{\qs}=
\begin{cases} 
0, &\text{ if } j\neq i+1, \\ 
1, &\text{ if }  j= i+1 \text{ and } \widetilde{q}_{i i+1} = q, \\ 
-1, &\text{ if }  j= i+1 \text{ and } \widetilde{q}_{i i+1} = q^{-1}; \end{cases}&& i \ne j \in \mI.
\end{align}
We note that ${\bf u}^\qs$ is upper triangular and, by \eqref{eq:setup}, it contains all information about $\qs$ as 
\begin{align*}
q_{ij}=q^{u_{ij}^\qs}, && i,j\in \mI.
\end{align*} 
Denote by $\superqam{\qs}$ the symmetrization of ${\bf u}^\qs$, i.e.~$\superqam{\qs}=({\bf u}^\qs)^t{\bf u}^\qs$. Explicitly, $\superqam{\qs}=(t_{ij}^{\qs})\in \mZ^{\mI\times\mI}$ is given by
\begin{align}\label{eq:supera-matrix-power}
t_{ii}^{\qs}=
\begin{cases} 
N, &\text{ if } i \in \mJ, \\ 
2, &\text{ if }  q_{ii}= q, \\ 
- 2, &\text{ if }  q_{ii}= q^{-1}; \end{cases} &&
t_{ij}^{\qs}=
\begin{cases} 
0, &\text{ if } j\neq i\pm 1, \\ 
1, &\text{ if } j= i \pm 1 \text{ and } \widetilde{q}_{i j} = q, \\ 
-1, &\text{ if } j= i \pm 1 \text{ and } \widetilde{q}_{i i+1} = q^{-1}; \end{cases}&& i \ne j \in \mI.
\end{align}
The importance of this (symmetric, tridiagonal) matrix resides in the following:
\begin{align*}
b_{\qs}(g_i,g_j)=q_{ij}q_{ji}=q^{t_{ij}^{\qs}}, && i,j\in \mI.
\end{align*}
Justification of our setup is given in \Cref{rem:justify-setup} after non-degeneracy of the base category has been investigated.

\subsection{The braided Drinfeld double of a Nichols algebra} \label[section]{sec:general-Drin}

In this section, we describe how to construct quasitriangular Hopf algebras, called \emph{braided Drinfeld doubles}, from Nichols algebras of diagonal type (cf. \cites{Lau2,Lau3,LW1,LW2}). We display the braided Drinfeld double as a quotient of the Drinfeld double of the bosonization in \Cref{prop:Drin-quotient}. We note that Drinfeld double of bosonization of Nichols algebras were studied in \cites{Hec3,AY,Vay} and other papers. We choose to work with braided Drinfeld doubles in order to capture the smaller relative Drinfeld centers through their representations, justified by \Cref{{sec:relcen}}.

To fix notation, let $\qs=(q_{ij})$ denote a matrix of non-zero scalars such that the associated Nichols algebra $\BB_\qs$ is finite-dimensional and admits a realization over a finite abelian group $G=\langle g_1,\ldots, g_r\mid g_i^{n_i}=1,  i=1,\ldots, r \rangle$ as in \Cref{subsubsec:realizations}. 

On one hand, consider the category $\lYD{\BB_\qs}(\cA_\qs)$ of Yetter--Drinfeld modules over $\BB_\qs$ in the braided category $\cA_\qs$. This category is equivalent to the relative center $\cZ_{\cA_\qs}(\lmod{\BB_\qs\rtimes \Bbbk[G]})$, see \cite{Lau3}*{Section~4.2}, \cite{LW2}*{Proposition~5.11}, and versions of it already appeared in \cites{Bes,Maj99}.

On the other hand, we consider $\Lambda=\mZ_{n_i}\times\ldots \times \mZ_{n_r}$, a copy of $G$ written additively, with generators $\alpha_1,\ldots, \alpha_r$ (we keep the notation used for the canonical basis of $\mZ^r$).
We can use tuples ${\bf i}=(i_1,\ldots, i_r)=i_1\alpha_1+\ldots+i_r \alpha_r$ of $\Lambda$ to enumerate the elements 
$g_{\bf i}=g_1^{i_1}\ldots g_r^{i_r}$ of $G$. Thus, $\{g_{\bf i}~|~ {\bf i}\in \Lambda\}$ is the usual basis for the 
group algebra $\Bbbk G$. 
Further, we use the dual basis $\{\delta_{\bf i}~|~ {\bf i}\in \Lambda\}$ for $\Bbbk[G]=(\Bbbk G)^*$ and fix notation for the following group-like elements in $\Bbbk[G]$:
\begin{align}\label{eq:gammais}
    \gamma_i=\sum_{\bf j}r_{\qs}(g_{\bf j}\otimes g_i)\delta_{\bf j}, \qquad    \ov{\gamma}_i=\sum_{\bf j}r_{\qs}(g_i\otimes g_{\bf j})\delta_{\bf j}.
\end{align}
Let $V=\Bbbk\langle x_1,\ldots, x_r \rangle$, $c(x_i\otimes x_j) = q_{ij} x_j\otimes x_i$, denote the braided vector space determined by $\qs$. By \cite{HS}*{Section 7.2} both $V$ and the dual braided vector space $V^*= \Bbbk\langle y_1,\ldots, y_r \rangle$, $c(y_i\otimes y_j)=q_{ij} y_j \otimes y_i$, are objects in $\cA_\qs$ in such a way that the categorical braidings coincide with the given vector space braidings. The associated Nichols algebras are denoted by $\BB_\qs$, as in \Cref{sec:Nichols-diagonal}, and $\BB_\qs^*$. Moreover, the evaluation map $V^*\otimes V \to \Bbbk$ extends uniquely to a non-degenerate pairing 
$$(-,-)\colon \BB_{\qs}^*\otimes \BB_\qs \to \Bbbk$$
of Hopf algebras in $\cA_\qs$.

There is a set of homogeneous generators for the defining ideals $\II_\qs$ of the Nichols algebra $\BB_\qs$, respectively, the ideal $\II_\qs^*$ of relations of $\BB_\qs^*$. We refer to those generators as the \emph{Nichols relations}, see \Cref{subsubsec:relations}.

\begin{definition}[Braided Drinfeld double]\label[definition]{def:braided-Drin}
The braided Drinfeld double $\Drin_{\Bbbk [G]}(\BB_\qs,\BB_\qs^*)$ is the Hopf algebra generated as an algebra $x_i$, $y_j$, and $\delta_{\bf i}$, for ${\bf i}\in \Lambda$, $i,j=1,\ldots, r$, subject to the relations
\begin{gather}
    \delta_{\bf i}\delta_{\bf j}=\delta_{{\bf i},{\bf j}}\delta_{\bf i},\qquad 
    \delta_{\bf i}x_j=x_j \delta_{{\bf i}-\alpha_j}, \qquad    \delta_{\bf i}y_j=y_j \delta_{{\bf i}+\alpha_j},\label{drinrel1} \\
    y_ix_j-q_{ji}x_jy_i=\delta_{i,j}(1-\ov{\gamma}_i\gamma_i),\label{drinrel2}
\end{gather}
and the Nichols relations in $\II_\qs$ and $\II_\qs^*$.

The coproduct, counit, and antipode of $\Drin_{\Bbbk [G]}(\BB_\qs,\BB_\qs^*)$ is given on generators by
\begin{gather}
    \Delta(\delta_{\bf i})=\sum_{{\bf a}+{\bf b}={\bf i}}\delta_{\bf a}\otimes \delta_{\bf b}, \qquad \Delta(x_i)=x_i\otimes 1+\gamma_i\otimes x_i, \qquad \Delta(y_i)=y_i\otimes 1+\ov{\gamma}_i\otimes y_i,\label{drinrel3}\\
    \varepsilon(\delta_{\bf i})=\delta_{{\bf i},0}, \qquad \varepsilon(x_i)=\varepsilon(y_i)=0,\label{drinrel4}\\
S(\delta_{\bf i})= \delta_{\bf -i},\qquad    S(x_i)=-\ov{\gamma}_i^{-1}x_i, \qquad  S(y_i)=-\gamma_i^{-1}y_i. \label{drinrel5}
\end{gather} 
\end{definition}
In particular, we derive the additonal relations
\begin{align}
        \gamma_i x_j&=q_{ij}x_j\gamma_i, & \gamma_i y_j &= q_{ij}^{-1}y_j \gamma_i,\\
    \ov{\gamma}_i x_j&=q_{ji}x_j\ov{\gamma}_i, & \ov{\gamma}_i y_j &= q_{ji}^{-1}y_j. \gamma_i.
\end{align}

A presentation of $\Drin_{\Bbbk [G]}(\BB_\qs,\BB_\qs^*)$ as above was given in \cite{LW2}*{Proposition~5.9}\footnote{Note the present paper corrects the presentation to use $y_ix_j-q_{ji}x_jy_i$ instead of $y_ix_j-q_{ji}^{-1}x_jy_i$ as in \cite{LW2}.}. Note that Equation \eqref{drinrel2} implies that $\Drin_{\Bbbk [G]}(\BB_\qs,\BB_\qs^*)$ has a triangular decomposition as a vector space as the tensor product $\BB_\qs\otimes \Bbbk [G]\otimes \BB_\qs^*$. We do not prove directly that $\Drin_{\Bbbk[G]}(\BB_\qs,\BB_\qs^*)$
 is a Hopf algebra. This statement will be a consequence of the following proposition and  \cite{Majid}*{Section 9.4}.

\begin{proposition}\label[proposition]{prop:DrinYD}
There is an equivalence of braided monoidal categories between $\lYD{\BB_\qs}(\cA_\qs)$ and the category of left modules over $\Drin_{\Bbbk [G]}(\BB_\qs,\BB_\qs^*)$.
\end{proposition}
\begin{proof}
The category $\lYD{\BB_\qs}(\cA_\qs)$ consists of objects $W$ in $\cA_\qs$ with a left action $a\colon B\otimes W\to W$ and coaction $\delta\colon W\to B\otimes W$ by $B=\BB_\qs$, which are compatible through the YD condition
\begin{equation}\label{YDcond}
(m\otimes a)(\ide_B\otimes \Psi_{B,B}\otimes \ide_W)(\Delta\otimes \delta)=(m\otimes \ide_W)(\ide_B\otimes \Psi_{W,B})(\delta a\otimes \ide_W)(\ide_B\otimes \Psi_{B,W})(\Delta\otimes \ide_W),
\end{equation}
cf. \cite{Lau2}*{Definition 2.1}, \cite{Bes}*{Section 3.3}. The left $\Bbbk G$-coaction becomes a left $\Bbbk[G]$-action via
$\delta_{\bf i}\cdot w:= \delta_{{\bf i},|w|}w$,
where $|w|$ is the $G$-degree of $w$. The assignment $\ev(y_i\otimes x_j):=\delta_{i,j}$ extends to a non-degenerate pairing of Hopf algebras in $\cA_\qs$, $\ev\colon \BB_\qs^*\otimes \BB_\qs \to \Bbbk$, where $\BB_\qs^*$ is the Nichols algebra of the comodule dual to $\qs$ \cite{HS}*{Theorem 7.2.3} with basis $y_1,\ldots, y_r$. The left $B$-coaction $\delta$ induces a left action of $(\BB_\qs^*)^{\mathrm{cop}}$, the braided Hopf algebra   in $\cA_\qs^\rev$ with braided opposite coproduct $\Delta^{\mathrm{cop}}=c^{-1}\Delta$, via 
\begin{align}
    y\cdot w=\ev(y\otimes w^{(-1)})w^{(0)}, && y\in \BB_\qs, w\in W,~~\text{where}~~ \delta(w)=w^{(-1)}\otimes w^{(0)}.
\end{align}
The proof now proceeds as in \cite{Lau2}*{Proposition~3.6} to show that
the obtained actions of $\BB_\qs$, $\Bbbk[G]$, and $\BB_\qs^*$ define an action of the algebra $\Drin_{\Bbbk[G]}(\BB_\qs,\BB_\qs^*)$ (denoted by $\Drin_{\Bbbk[G]}(\BB_\qs^*,\BB_\qs)$ in \cite{Lau2}) on $W$ which gives an equivalence of monoidal categories 
$$\lYD{\BB_\qs}(\cA_\qs)\isomorph \lmod{\Drin_{\Bbbk[G]}(\BB_\qs,\BB_\qs^*)}.$$
It remains to specify the relations derived in \cite{Lau2} to the case of Nichols algebras over groups.
The universal $R$-matrix for $\Bbbk[G]$ obtained by dualizing the dual $R$-matrix of $\Bbbk G$ is given by 
\begin{align}\label{eq:r-matrixG}
    R=r_*=\sum_{{\bf i}, {\bf j}}r_{\qs}(g_{\bf j},g_{\bf i})\delta_{\bf i}\otimes \delta_{\bf j} =\sum_{\bf j} \ov{\gamma}_{\bf j}\otimes \delta_{\bf j} = \sum_{\bf i}\delta_{\bf i}\otimes \gamma_{\bf i}.
\end{align}  With this notation, \cite{Lau2}*{Equation (3.18)} specializes to
\begin{align*}
y_ix_j-q_{ji}x_jy_i&= \ev(y_i\otimes x_j)\Big(1-\sum_{{\bf i},{\bf k}}r_{\qs}(g_{\bf i}\otimes g_i^{-1})^{-1} \delta_{\bf i}r_{\qs}(g_j\otimes g_{\bf k})\delta_{\bf k}\Big)\\
&= \delta_{i,j}\Big(1-\sum_{\bf i}r_{\qs}(g_{\bf i}\otimes g_i) r_{\qs}(g_j\otimes g_{\bf i})\delta_{\bf i}\Big)= \delta_{i,j}(1-\gamma_i\ov{\gamma}_i),
\end{align*}
recovering \eqref{drinrel2}. As $|x_i|=g_i$ and $|y_i|=g_i^{-1}$, \eqref{drinrel1} holds. 

The coproducts of the Nichols algebra generators are given by
\begin{align*}
    \Delta(x_i)&=x_i\otimes 1+ \sum_{\bf i}r_\qs(g_{i}\otimes g_{\bf j})\delta_{\bf j}\otimes x_i= x_i\otimes 1 + \gamma_i\otimes x_i,\\
        \Delta(y_i)&=y_i\otimes 1+ \sum_{\bf i}r_\qs(g_{\bf i}\otimes g_{i}^{-1})^{-1}\delta_{\bf i}\otimes y_i= y_i\otimes 1 + \ov{\gamma}_i\otimes y_i,
\end{align*}
verifying \eqref{drinrel3}. The remaining relations \eqref{drinrel4}--\eqref{drinrel5} are direct consequences of the coproduct formulas on the generators, using that $\gamma_i,\ov{\gamma}_i$ are grouplike elements and $\Bbbk[G]$ is a Hopf subalgebra of $\Drin_{\Bbbk[G]}(\BB_\qs,\BB_\qs^*)$.
\end{proof}

As another consequence of \Cref{prop:DrinYD}, we obtain an   $R$-matrix for this braided double.

\begin{corollary}\label[corollary]{cor:R-matrix-Drin}
If $\BB_\qs$ is finite-dimensional, then the Hopf algebra $\Drin_{\Bbbk[G]}(\BB_\qs,\BB_\qs^*)$ is quasi-triangular, with universal $R$-matrix given by 
\begin{align}
    R_{\Drin}=\sum_{\alpha,{\bf i},{\bf j}}r_{\qs}(g_{\bf i}\otimes g_{\bf j})\delta_{\bf i}y_{\alpha}\otimes x_{\alpha}\delta_{\bf j}= \sum_{\alpha, {\bf j}}\ov{\gamma}_{\bf j}y_\alpha \otimes x_\alpha \delta_{\bf j} = \sum_{\alpha, {\bf i}}\delta_{\bf i}y_\alpha \otimes x_\alpha \gamma_{\bf i},
\end{align}
where $\{y_\alpha\}$ is the  basis for $\BB_\qs^*$ which is dual to a basis $\{x_\alpha\}$ for $\BB_\qs$.
\end{corollary}

Since $\lYD{\BB_\qs}(\cA_\qs)$ is also equivalent to the relative center $\cZ_{\cA_\qs}(\lmod{\BB_\qs\rtimes \Bbbk[G]})$, see \cite{Lau3}*{Section~4.2}, \cite{LW2}*{Proposition~5.11}, we can later use \Cref{thm:LW-center-relcenter} to investigate modularity of the braided Drinfeld double, as done in \cite{LW2}*{Proposition~5.15}.

In the remainder of this section, we will display the braided Drinfeld double $\Drin_{\Bbbk[G]}(\BB_\qs,\BB_\qs^*)$ as a Hopf algebra quotient of the Drinfeld double of the bosonization $H:=\BB_\qs\rtimes \Bbbk[G]$. In order to construct this quotient homomorphism, we employ the following functor \cite{LW1}*{Example~3.11}.

\begin{lemma}\label[lemma]{lem:YD-functor}
There is a fully faithful functor of braided tensor categories
$$I\colon \lYD{\BB_\qs}(\cA_\qs)\to \lYD{H}(\Vect).$$
\end{lemma}
\begin{proof}
The functor $I$ sends a $\BB_\qs$-Yetter--Drinfeld module $V=\bigoplus_{g\in G} V_g$ in $\cA_\qs$, described by a left $\BB_\qs$-action $a_V$ and coaction $\delta_V$,
\begin{align*}
    a_V\colon \BB_\qs\otimes V\to V, \quad x\otimes v\mapsto xv \qquad   \delta_V\colon V\to \BB_\qs \otimes V,\quad v\mapsto v^{(-1)}\otimes v^{(0)},
\end{align*}
to the the $H$-Yetter--Drinfeld module $V$ together with left $H$-action and coaction given by 
\begin{align*}
    (x\otimes \delta_g)\otimes v_h= \delta_{g,h}xv_h, \qquad \delta^H(v_h)=(v^{(-1)}\otimes \ov{\gamma}_{h})\otimes v_h,
\end{align*}
for all $v_h\in V_h$. On morphisms, $I$ is given by the identity which is clearly faithful. Any morphism of $H$-Yetter--Drinfeld modules preserves the $G$-grading and commutes with the $\BB_\qs$-action and coaction. Thus, $I$ is full.
\end{proof}

The Hopf algebra dual $H^*$ is defined on the vector space $\Bbbk G\otimes \BB_\qs^*$ with duality pairing given by 
$$ (-,-)\colon (\Bbbk G\otimes \BB_\qs^*)\otimes (\Bbbk[G]\otimes \BB_{\qs})\to \Bbbk, \quad (g\otimes y, f\otimes x)=f(g)(y,x),$$
where $(y,x)$ is the value of $x\otimes y$ under the pairing $(-,-)\colon \BB_{\qs}^*\otimes \BB_\qs\to \Bbbk$ of braided Hopf algebras described in \Cref{sec:general-Drin}.

\begin{proposition}\label[proposition]{prop:Drin-quotient}
There is a  surjective homomorphism of Hopf algebras over $\Bbbk$
$$\varphi\colon \Drin(\BB_{\qs}\rtimes \Bbbk [G])\twoheadrightarrow \Drin_{\Bbbk[G]}(\BB_\qs,\BB_\qs^*),$$
given by sending $g_{\bf i}\in G$ to $\ov{\gamma}_{\bf i}\in \Bbbk[G]\subset \Drin_{\Bbbk[G]}(\BB_\qs,\BB_\qs^*)$ and  by identities on $\Bbbk[G]$, $\BB_\qs$, and $\BB_\qs^*$.
\end{proposition}
\begin{proof}
By reconstruction theory (see e.g. \cite{Majid}*{Chapter~9}) functor $I$ from \Cref{lem:YD-functor}, which is compatible with the respective forgetful functors to $\Vect$, induces a surjective morphism of $\Bbbk$-Hopf algebras $\varphi$ as stated. This uses the equivalences of categories from \Cref{prop:DrinYD} appearing as vertical arrows in the commutative diagram
$$
\xymatrix{
\lYD{\BB_\qs}(\cA_\qs)\ar[rr]^I\ar[d]^{\sim}&&\lYD{H}(\Vect)\ar[d]^{\sim}\\
\lmod{\Drin_{\Bbbk[G]}(\BB_\qs,\BB_\qs^*)}\ar[rr]^{\Res_\varphi}&&
\lmod{\Drin(\BB_{\qs}\rtimes \Bbbk [G])},
}
$$
where $\Res_\varphi$ is the restriction functor along $\varphi$. Tracing the image of a $\BB_\qs$-YD module $V$ in $\cA_\qs$ through both paths of the diagram, we see that elements $x\in \BB_\qs$, $y\in \BB_\qs$, and $\delta_g\in \Bbbk[G]$ act the same way whether regarded as an element of $\Drin(\BB_{\qs}\rtimes \Bbbk [G])$ or $\Drin_{\Bbbk[G]}(\BB_\qs,\BB_\qs^*)$. Namely,
$x$ acts via the given $\BB_\qs$ action, $y$ acts by dualizing the $\BB_\qs$-coaction, i.e., $y\cdot v=(y,v^{(-1)})v^{(0)}$, and $\delta_g\cdot v=\delta_{g,|v|}v$. 

Starting with an element $g_{\bf i}\in G$, we can regard $g_{\bf i}$ as an element of $\Drin(\BB_{\qs}\rtimes \Bbbk [G])$; as such, it acts on $I(V)$ via
$$g_{\bf i}\cdot v=r(g_{\bf i},|v|)v= \ov{\gamma}_{\bf i}\cdot v.$$
The right hand side describes the action of $\ov{\gamma}_{\bf i}\in \Drin_{\Bbbk[G]}(\BB_\qs,\BB_\qs^*)$ on $V$. As this equality holds for a general $\BB_\qs$-YD module $V$ in $\cA_\qs$, it follows that $\varphi(g_{\bf i})=\ov{\gamma}_{\bf i}$.
\end{proof}


\section{Classification of spherical and ribbon structures}\label[section]{sec:spher-ribbon}

We now restrict to considering Nichols algebras associated to parameters $\qs$ of type Super A, using the setup from \Cref{subsubsec:setup}.
In this section, we first establish conditions for non-degeneracy of the base category $\cA_\qs$ in \Cref{sec:nondeg} before characterizing when the bosonizations $\BB_\qs\rtimes \Bbbk[G]$ of the Nichols algebras $\BB_\qs$ of super-type $\superqa{r}{q}{\mJ}$ are unimodular and admit a non-semisimple spherical structure in Sections \ref{subsec:unimodularity} and \ref{sec:spherical}. Finally, we classify ribbon structures for the Drinfeld double of these bosonizations in \Cref{sec:ribbon}.

\subsection{Non-degeneracy of the base category}\label[section]{sec:nondeg}

In this section, we establish criteria for the base category $\cA_\qs$ to be non-degenerate.
We use the notation introduced in \Cref{subsubsec:setup}. In particular, $\qs$ is of  super-type $\superqa{r}{q}{\mJ}$, where $q$ is a root of unity of even order $N=2n>2$.

\begin{proposition}\label[proposition]{prop:base-non-degeneracy}
The braided category $\cA_{\qs}$ is non-degenerate if and only if the determinant of $\superqam{\qs}$ is relative prime to $N$. 
\end{proposition}

\begin{proof}
We know from \cite{EGNO}*{Example 8.13.5} that $\cA_{\qs}$ is non-degenerate if and only if so is the pairing $b_{\qs}$. Notice that a generic element $g=\prod_{i=1}^{r}g_i^{m_i}$ of $G$ is in the radical of $b_{\qs}$ if and only if $b_{\qs}(g,g_j)=1$ for all $j=1,\dots, r$. But  
\begin{align*}
b_{\qs}(g,g_j)=q^{\sum_{i}t^{\qs}_{ij} m_i}=q^{\sum_{i}t^{\qs}_{ji} m_i}, && j=1,\dots, r.
\end{align*}
This means that $g=\prod_{i=1}^{r}g_i^{m_i}$ is in the radical if and only if $\superqam{\qs}$ annihilates the vector $[m_1 \dots m_r]^t $ modulo $\ord q=N$.
Hence the radical is trivial precisely when the reduction modulo $N$ of $\superqam{\qs}$ is invertible in $(\mZ/N\mZ)^{r \times r}$.
\end{proof}

Next, we derive an explicit formula for the determinant of $\superqam{\qs}$. We start with a very particular case that will be revisited in \Cref{subsec:unimodularity}. 

\begin{proposition}\label[proposition]{prop:base-I=J-non-degeneracy}
Let $\qs$ of type $\superqa{r}{q}{\mI_r}$ (i.e $q_{ii}=-1$ for all $i \in \mI_r$).
Then the determinant of $\superqam{\qs}$ is $0$ if $r$ is odd and $(-1)^{r/2}$ otherwise.
In particular, $\cA_{\qs}$ is non-degenerate if and only $r$ is even.
\end{proposition}

\begin{proof}
Denote by $\superqam{(r)}$ the matrix associated to $\superqa{r}{q}{\mI_r}$. It is easy to see that $\det \superqam{(r)} = -\det \superqam{(r - 2)}$. Since $\det \superqam{(2)} = -1$ and $\det \superqam{(3)} = 0$, the claim follows inductively.
\end{proof}

To attack the general case, we will need more vocabulary. Recall that the \emph{companion matrix} $\Comp^\alpha$ of a monic polynomial $\alpha(x)=x^{r}+a_{r-1} x^{r-1}+\dots + a_1x + a_0$ is 
\begin{align*}
\Comp_\alpha=
\begin{pmatrix} 
0 & 0 &\cdots &0&-a_0\\
1 & 0 &\cdots &0&-a_1 \\
0 & 1 &\cdots &0&-a_2 \\
\vdots & \vdots & \ddots & \vdots &\vdots\\
0 & 0 &\cdots &1&-a_{r-1}
\end{pmatrix},
\end{align*}
and the characteristic polynomial of $\Comp_\alpha$ is $\alpha$.

Given $\qs$ of type $\superqa{r}{q}{\mJ}$, using the parity morphism $\bfp_{\mJ}$ from \eqref{eq:parity-map}, we define a family of signs 
\begin{align}\label{eq:supera-upper-signs}
\varepsilon^\qs = (\varepsilon^\qs_1, \dots, \varepsilon^\qs_{r})\in\{\pm 1\}^{r}, 
&& \varepsilon^\qs_{j}\coloneq\bfp_{\mJ}(\alpha_{j r})=\bfp_{\mJ}(\alpha_{j}+ \dots+ \alpha_r), && 1\leq j \leq r.
\end{align}
Associated to them we have an upper triangular matrix $E^{\qs}$ and a monic polynomial $\alpha^\qs$ given by
\begin{align}
E^{\qs}=
\begin{pmatrix}
\varepsilon^\qs_{1} & \varepsilon^\qs_{1} & \dots & \varepsilon^\qs_{1}\\
0 & \varepsilon^\qs_{2} & \dots & \varepsilon^\qs_{2}\\
\vdots & \vdots& \ddots & \vdots \\
0& 0 &\dots& \varepsilon^\qs_{r}
\end{pmatrix}, &&
\alpha^{\qs}(x)=x^{r}+\varepsilon^\qs_{r} x^{r-1}+\dots + \varepsilon^\qs_{2}x + \varepsilon^\qs_{1}.
\end{align}

Next we establish relations between some entries of $E^{\qs}$ and $\superqam{\qs}$ defined in \eqref{eq:supera-matrix-power}.

\begin{lemma}\label[lemma]{lem:base-general-non-degeneracy}
These signs satisfy 
$\varepsilon^\qs_{j}=-\superqam{\qs}_{j-1,j}$ for $2\leq j\leq r$ and $\varepsilon^\qs_{1}=\superqam{\qs}_{1,1}+\superqam{\qs}_{1,2}$.
\end{lemma}

\begin{proof}
We verify the first equality by induction on $j$, starting with $j=r$ and going all the way back to $j=2$. For the base case we need to show that $\varepsilon^\qs_{r}=-\superqam{\qs}_{r-1,r}$, which becomes evident when considering separately the cases $r \in \mJ$ and $r \notin \mJ$.
Assume now that $\varepsilon^\qs_{j+1}=-\superqam{\qs}_{j,j+1}$ for some $2\leq j< r$. If $j \in \mJ$ then $\superqam{\qs}_{j-1,j} = - \superqam{\qs}_{j,j+1} $ by condition \eqref{item:def-super-A-3} of Definition \ref{def:super-A}, so we have $\varepsilon^\qs_{j}=-\varepsilon^\qs_{j+1}=\superqam{\qs}_{j,j+1} = - \superqam{\qs}_{j,j-1}$. On the other hand, if $j \notin \mJ$ then $\superqam{\qs}_{j-1,j} =  \superqam{\qs}_{j,j+1} $ by condition \eqref{item:def-super-A-4}, and now we have $\varepsilon^\qs_{j}=\varepsilon^\qs_{j+1}=-\superqam{\qs}_{j,j+1} = - \superqam{\qs}_{j,j-1}$.

Finally, we  use the equality above for $j=2$ to show that $\varepsilon^\qs_{1}=\superqam{\qs}_{1,1}+\superqam{\qs}_{1,2}$. If $1 \in \mJ$, since $\superqam{\qs}_{1,1}=0$ we get $\varepsilon^\qs_{1}=-\varepsilon^\qs_{2}= \superqam{\qs}_{1,2}=\superqam{\qs}_{1,1}+\superqam{\qs}_{1,2}$. For the case $1 \notin \mJ$ we have  $\superqam{\qs}_{1,1}=-2\superqam{\qs}_{1,2}$ by \eqref{item:def-super-A-4} of Definition \ref{def:super-A}, thus $\varepsilon^\qs_{1}=\varepsilon^\qs_{2}= -\superqam{\qs}_{1,2}=\superqam{\qs}_{1,1}+\superqam{\qs}_{1,2}$, as claimed.
\end{proof}

The determinant of $\superqam{\qs}$ can be computed using the following.

\begin{proposition}\label[proposition]{prop:base-general-non-degeneracy}
We have $E^{\qs}\superqam{\qs} = \ide - \Comp_{\alpha^\qs}$. In particular $\det \superqam{\qs} = (\det E^{\qs})\alpha^\qs(1)=\pm \alpha^\qs(1)$.
\end{proposition}

\begin{proof}
We proceed by induction on $r \geq 2$. For $r=2$ we verify this case by case:
\begin{itemize}
\item $\superqa{2}{q}{\{1\}}$: $E^{\qs}\superqam{\qs}=
\begin{pmatrix}1 & -1 \\-1 & 2\end{pmatrix} = \ide - \Comp_{x^2+x-1}$;

\item $\superqa{2}{q}{\{2\}}$: $E^{\qs}\superqam{\qs}=
\begin{pmatrix} 1 & -1 \\ -1 & 0 \end{pmatrix} = \ide - \Comp_{x^2-x-1}$;

\item $\superqa{2}{q}{\{1,2\}}$: $E^{\qs}\superqam{\qs}=
\begin{pmatrix}1& 1 \\ -1 & 0 \end{pmatrix} = \ide - \Comp_{x^2-x+1}$.
\end{itemize}
Now let $r > 2$. Denote by $\qs_{>1}$ the matrix obtained from $\qs$ by erasing the first row and column, which is of super-type $\superqa{r-1}{q}{\mJ-\{1\}}$. Now we compute
\begin{align*}
 &E^{\qs}\superqam{\qs}=
\begin{pmatrix}
\varepsilon^\qs_{1} & \varepsilon^\qs_{1} & \dots & \varepsilon^\qs_{1}\\
0 & &  & \\
\vdots & & E^{\qs_{>1}} & \\
0&  & & 
\end{pmatrix} 
\begin{pmatrix}
t^\qs_{11} & t^\qs_{12} & \dots & 0\\
t^\qs_{21} & &  & \\
\vdots & & \superqam{\qs_{>1}} & \\
0&  & & 
\end{pmatrix} \\
=& \begin{pmatrix}
\varepsilon^\qs_{1} (t^\qs_{11} +t^\qs_{21}) & \varepsilon^\qs_{1} (t^\qs_{12} +t^\qs_{22}+t^\qs_{32}) & \dots & \varepsilon^\qs_{1} (t^\qs_{r-2 r-1} +t^\qs_{r-1 r-1}+t^\qs_{r r-1})& \varepsilon^\qs_{1}(t^\qs_{r-1 r} +t^\qs_{r r})\\
\varepsilon^\qs_{2}t^\qs_{21} & &  & \\
\vdots & & E^{\qs_{>1}}\superqam{\qs_{>1}}  & \\
0&  & & 
\end{pmatrix}
\end{align*}
Inductively, we have $E^{\qs_{>1}}\superqam{\qs_{>1}} = \ide - \Comp_{\alpha^{\qs_{>1}}}$. 
From the conditions \eqref{item:def-super-A-3} and \eqref{item:def-super-A-4} in Definition \ref{def:super-A} it follows that $t^\qs_{i-1 i} +t^\qs_{ii}+t^\qs_{i+1 i} = 0$ for all $1<i<r$, and condition \eqref{item:def-super-A-1} implies that $t^\qs_{r-1 r} +t^\qs_{r r} = 1$.
Now Lemma \ref{lem:base-general-non-degeneracy} warrants that $\varepsilon^\qs_{2}t^\qs_{21}=-1$ and $\varepsilon^\qs_{1} (t^\qs_{11} +t^\qs_{21})=1$, as desired.
\end{proof}

\begin{corollary}\label[corollary]{cor:Aq-nondeg} Let $\qs$ be of type $\superqa{r}{q}{\mJ}$ and consider the category $\cA_\qs$ from \Cref{def:Aq}.
\begin{enumerate}
    \item If $r$ is even, $\det \superqam{\qs}$ is a non-zero odd integer. In particular, there exist a root of unity $q$ such that the braiding on $\cA_\qs$ is non-degenerate.
    \item If $r$ is odd, $\det \superqam{\qs}$ is an even integer. In particular, the resulting braiding on $\cA_\qs$ is degenerate for all $q$. 
\end{enumerate}
\end{corollary}
\begin{proof}
By \Cref{prop:base-general-non-degeneracy} it follows that 
$\det \superqam{\qs}=\pm(1+ k-h)$ where $k+h=r$ is a decomposition of natural numbers. If $r$ is even, then so is $k-h$ and hence $\pm(1+k-h)$ is an odd integer. However, if $r$ is odd, then $k-h$ is odd so $\pm(1+k-h)$ is even. In the case of $r$ even, this implies that $\det \superqam{\qs}$ is invertible modulo $N=\ord q$ (or $2N$ if $N$ is odd) as long as $\det \superqam{\qs}$ is coprime to $N$. Similarly, if $r$ is odd, $\det \superqam{\qs}$ is never coprime to the even number $N$. 
\end{proof}
Note that  $\det \superqam{\qs}$ might not be coprime to the order of $q$ causing the braiding on $\cA_\qs$ to be degenerate. For example, if $q$ is a $3m$-th root of unity, then braidings of type $\superqa{4}{q}{\Set{1}}$ give a degenerate braiding on $\cA_\qs$ (see \Cref{tab:theta4}).

\begin{example}For $r=1$, there are no generalized Dynkin diagrams of type Super A.
In \Cref{tab:theta2}, \Cref{tab:theta3}, and \Cref{tab:theta4} we collect the generalized Dynkin diagrams and values of $\det \superqam{\qs}$ for all $\qs$ of type $\superqa{r}{q}{\mJ}$ with $r=2, 3$, and $4$, respectively, omitting symmetries of inverting the vertex order and interchanging $q$ with $q^{-1}$.
\end{example}

\begin{table}[htb]
\begin{align*}
\begin{array}{c|c|c}
 \text{Type}&\text{Generalized Dynkin diagram}& \det \superqam{\qs} \\\hline
    \superqa{2}{q}{\Set{1}} & 
    \vcenter{\hbox{\begin{tikzpicture}
\node [circle,draw,label=above:$-1$] (1){};
\node [circle,draw,label=above:$q$] (2)[right of=1,node distance=1.5cm]{};
\draw  (1.east) -- (2.west) node [above,text centered,midway]
{$q^{-1}$};
\end{tikzpicture}}}
& \det \superqam{\qs}=(-1)^1(1+1-1)=-1\\
  \superqa{2}{q}{\mI_2} & 
    \vcenter{\hbox{\begin{tikzpicture}
\node [circle,draw,label=above:$-1$] (1){};
\node [circle,draw,label=above:$-1$] (2)[right of=1,node distance=1.5cm]{};
\draw  (1.east) -- (2.west) node [above,text centered,midway]
{$q$};
\end{tikzpicture}}}
& \det \superqam{\qs}=(-1)^{1}(1-1+1)=-1
\end{array}
\end{align*}
    \caption{Generalized Dynkin diagrams of type Super A and rank $r=2$}
    \label{tab:theta2}
\end{table}

\begin{table}[htbp]
\begin{align*}
\begin{array}{c|c|c}
 \text{Type}&\text{Generalized Dynkin diagram}& \det \superqam{\qs} \\\hline
    \superqa{3}{q}{\Set{1}} & 
    \vcenter{\hbox{\begin{tikzpicture}
\node [circle,draw,label=above:$-1$] (1){};
\node [circle,draw,label=above:$q$] (2)[right of=1,node distance=1.5cm]{};
\node [circle,draw,label=above:$q$] (3)[right of=2,node distance=1.5cm]{};
\draw  (1.east) -- (2.west) node [above,text centered,midway]
{$q^{-1}$};
\draw  (2.east) -- (3.west) node [above,text centered,midway]
{$q^{-1}$};
\end{tikzpicture}}}
& \det \superqam{\qs}=(-1)^1(1+1+1-1)=-2\\
    \superqa{3}{q}{\Set{2}} & 
    \vcenter{\hbox{\begin{tikzpicture}
\node [circle,draw,label=above:$q^{-1}$] (1){};
\node [circle,draw,label=above:$-1$] (2)[right of=1,node distance=1.5cm]{};
\node [circle,draw,label=above:$q$] (3)[right of=2,node distance=1.5cm]{};
\draw  (1.east) -- (2.west) node [above,text centered,midway]
{$q$};
\draw  (2.east) -- (3.west) node [above,text centered,midway]
{$q^{-1}$};
\end{tikzpicture}}}
& \det \superqam{\qs}=(-1)^{2}(1+1-1-1)=0\\
    \superqa{3}{q}{\Set{1,2}} & 
    \vcenter{\hbox{\begin{tikzpicture}
\node [circle,draw,label=above:$-1$] (1){};
\node [circle,draw,label=above:$-1$] (2)[right of=1,node distance=1.5cm]{};
\node [circle,draw,label=above:$q$] (3)[right of=2,node distance=1.5cm]{};
\draw  (1.east) -- (2.west) node [above,text centered,midway]
{$q$};
\draw  (2.east) -- (3.west) node [above,text centered,midway]
{$q^{-1}$};
\end{tikzpicture}}}
& \det \superqam{\qs}=(-1)^{1}(1+1-1+1)=-2\\
    \superqa{3}{q}{\Set{1,3}} & 
    \vcenter{\hbox{\begin{tikzpicture}
\node [circle,draw,label=above:$-1$] (1){};
\node [circle,draw,label=above:$q^{-1}$] (2)[right of=1,node distance=1.5cm]{};
\node [circle,draw,label=above:$-1$] (3)[right of=2,node distance=1.5cm]{};
\draw  (1.east) -- (2.west) node [above,text centered,midway]
{$q$};
\draw  (2.east) -- (3.west) node [above,text centered,midway]
{$q$};
\end{tikzpicture}}}
& \det \superqam{\qs}=(-1)^2(1-1-1+1)=0\\
    \superqa{3}{q}{\mI_3} & 
    \vcenter{\hbox{\begin{tikzpicture}
\node [circle,draw,label=above:$-1$] (1){};
\node [circle,draw,label=above:$-1$] (2)[right of=1,node distance=1.5cm]{};
\node [circle,draw,label=above:$-1$] (3)[right of=2,node distance=1.5cm]{};
\draw  (1.east) -- (2.west) node [above,text centered,midway]
{$q^{-1}$};
\draw  (2.east) -- (3.west) node [above,text centered,midway]
{$q$};
\end{tikzpicture}}}
& \det \superqam{\qs}=(-1)^2(1-1+1-1)=0\\
\end{array}
\end{align*}
    \caption{Generalized Dynkin diagrams of type Super A and rank $r=3$}
    \label{tab:theta3}
\end{table}

\begin{table}[htb]
\begin{align*}
\begin{array}{c|c|c}
 \text{Type}&\text{Generalized Dynkin diagram}& \det \superqam{\qs} \\\hline
    \superqa{4}{q}{\Set{1}} & 
    \vcenter{\hbox{\begin{tikzpicture}
\node [circle,draw,label=above:$-1$] (1){};
\node [circle,draw,label=above:$q$] (2)[right of=1,node distance=1.5cm]{};
\node [circle,draw,label=above:$q$] (3)[right of=2,node distance=1.5cm]{};
\node [circle,draw,label=above:$q$] (4)[right of=3,node distance=1.5cm]{};
\draw  (1.east) -- (2.west) node [above,text centered,midway]
{$q^{-1}$};
\draw  (2.east) -- (3.west) node [above,text centered,midway]
{$q^{-1}$};
\draw  (3.east) -- (4.west) node [above,text centered,midway]
{$q^{-1}$};
\end{tikzpicture}}}
& \det \superqam{\qs}=(-1)^1(1+1+1+1-1)=-3\\
    \superqa{4}{q}{\Set{2}} &  
    \vcenter{\hbox{\begin{tikzpicture}
\node [circle,draw,label=above:$q^{-1}$] (1){};
\node [circle,draw,label=above:$-1$] (2)[right of=1,node distance=1.5cm]{};
\node [circle,draw,label=above:$q$] (3)[right of=2,node distance=1.5cm]{};
\node [circle,draw,label=above:$q$] (4)[right of=3,node distance=1.5cm]{};
\draw  (1.east) -- (2.west) node [above,text centered,midway]
{$q$};
\draw  (2.east) -- (3.west) node [above,text centered,midway]
{$q^{-1}$};
\draw  (3.east) -- (4.west) node [above,text centered,midway]
{$q^{-1}$};
\end{tikzpicture}}}
& \det \superqam{\qs}=(-1)^2(1+1+1-1-1)=1\\
    \superqa{4}{q}{\Set{1,2}} &  
    \vcenter{\hbox{\begin{tikzpicture}
\node [circle,draw,label=above:$-1$] (1){};
\node [circle,draw,label=above:$-1$] (2)[right of=1,node distance=1.5cm]{};
\node [circle,draw,label=above:$q$] (3)[right of=2,node distance=1.5cm]{};
\node [circle,draw,label=above:$q$] (4)[right of=3,node distance=1.5cm]{};
\draw  (1.east) -- (2.west) node [above,text centered,midway]
{$q$};
\draw  (2.east) -- (3.west) node [above,text centered,midway]
{$q^{-1}$};
\draw  (3.east) -- (4.west) node [above,text centered,midway]
{$q^{-1}$};
\end{tikzpicture}}}
& \det \superqam{\qs}=(-1)^1(1+1+1-1+1)=-3\\
    \superqa{4}{q}{\Set{1,3}} & 
    \vcenter{\hbox{\begin{tikzpicture}
\node [circle,draw,label=above:$-1$] (1){};
\node [circle,draw,label=above:$q^{-1}$] (2)[right of=1,node distance=1.5cm]{};
\node [circle,draw,label=above:$-1$] (3)[right of=2,node distance=1.5cm]{};
\node [circle,draw,label=above:$q$] (4)[right of=3,node distance=1.5cm]{};
\draw  (1.east) -- (2.west) node [above,text centered,midway]
{$q$};
\draw  (2.east) -- (3.west) node [above,text centered,midway]
{$q$};
\draw  (3.east) -- (4.west) node [above,text centered,midway]
{$q^{-1}$};
\end{tikzpicture}}}
& \det \superqam{\qs}=(-1)^2(1+1-1-1+1)=1\\
    \superqa{4}{q}{\Set{2,3}} &  
    \vcenter{\hbox{\begin{tikzpicture}
\node [circle,draw,label=above:$q$] (1){};
\node [circle,draw,label=above:$-1$] (2)[right of=1,node distance=1.5cm]{};
\node [circle,draw,label=above:$-1$] (3)[right of=2,node distance=1.5cm]{};
\node [circle,draw,label=above:$q$] (4)[right of=3,node distance=1.5cm]{};
\draw  (1.east) -- (2.west) node [above,text centered,midway]
{$q^{-1}$};
\draw  (2.east) -- (3.west) node [above,text centered,midway]
{$q$};
\draw  (3.east) -- (4.west) node [above,text centered,midway]
{$q^{-1}$};
\end{tikzpicture}}}
& \det \superqam{\qs}=(-1)^{1}(1+1-1+1+1)=-3\\
    \superqa{4}{q}{\Set{1,4}} &  
    \vcenter{\hbox{\begin{tikzpicture}
\node [circle,draw,label=above:$-1$] (1){};
\node [circle,draw,label=above:$q^{-1}$] (2)[right of=1,node distance=1.5cm]{};
\node [circle,draw,label=above:$q^{-1}$] (3)[right of=2,node distance=1.5cm]{};
\node [circle,draw,label=above:$-1$] (4)[right of=3,node distance=1.5cm]{};
\draw  (1.east) -- (2.west) node [above,text centered,midway]
{$q$};
\draw  (2.east) -- (3.west) node [above,text centered,midway]
{$q$};
\draw  (3.east) -- (4.west) node [above,text centered,midway]
{$q$};
\end{tikzpicture}}}
& \det \superqam{\qs}=(-1)^{3}(+1-1-1-1+1)=1\\
    \superqa{4}{q}{\Set{1,2,3}} &  
    \vcenter{\hbox{\begin{tikzpicture}
\node [circle,draw,label=above:$-1$] (1){};
\node [circle,draw,label=above:$-1$] (2)[right of=1,node distance=1.5cm]{};
\node [circle,draw,label=above:$-1$] (3)[right of=2,node distance=1.5cm]{};
\node [circle,draw,label=above:$q$] (4)[right of=3,node distance=1.5cm]{};
\draw  (1.east) -- (2.west) node [above,text centered,midway]
{$q^{-1}$};
\draw  (2.east) -- (3.west) node [above,text centered,midway]
{$q$};
\draw  (3.east) -- (4.west) node [above,text centered,midway]
{$q^{-1}$};
\end{tikzpicture}}}
&\det \superqam{\qs}=(-1)^{2}(+1+1-1+1-1)=1\\
    \superqa{4}{q}{\Set{1,3,4}} & 
    \vcenter{\hbox{\begin{tikzpicture}
\node [circle,draw,label=above:$-1$] (1){};
\node [circle,draw,label=above:$q$] (2)[right of=1,node distance=1.5cm]{};
\node [circle,draw,label=above:$-1$] (3)[right of=2,node distance=1.5cm]{};
\node [circle,draw,label=above:$-1$] (4)[right of=3,node distance=1.5cm]{};
\draw  (1.east) -- (2.west) node [above,text centered,midway]
{$q^{-1}$};
\draw  (2.east) -- (3.west) node [above,text centered,midway]
{$q^{-1}$};
\draw  (3.east) -- (4.west) node [above,text centered,midway]
{$q$};
\end{tikzpicture}}}
& \det \superqam{\qs}=(-1)^{2}(+1-1+1+1-1)=1\\
    \superqa{4}{q}{\mI_4} &
    \vcenter{\hbox{\begin{tikzpicture}
\node [circle,draw,label=above:$-1$] (1){};
\node [circle,draw,label=above:$-1$] (2)[right of=1,node distance=1.5cm]{};
\node [circle,draw,label=above:$-1$] (3)[right of=2,node distance=1.5cm]{};
\node [circle,draw,label=above:$-1$] (4)[right of=3,node distance=1.5cm]{};
\draw  (1.east) -- (2.west) node [above,text centered,midway]
{$q$};
\draw  (2.east) -- (3.west) node [above,text centered,midway]
{$q^{-1}$};
\draw  (3.east) -- (4.west) node [above,text centered,midway]
{$q$};
\end{tikzpicture}}}
& \det \superqam{\qs}=(-1)^{2}(+1-1+1-1+1)=1
\end{array}
\end{align*}
    \caption{Generalized Dynkin diagrams of type Super A and rank $r=4$}
    \label{tab:theta4}
\end{table}

\begin{remark}\label[remark]{rem:justify-setup}
    To conclude this section, we justify our choice of setup that $G\cong \mZ_N^r$, for $N$ even, $q$ being an even root of unity, and of the specific realization of \eqref{eq:supera-matrix-power2}. 

Note that non-degeneracy of $\cA_\qs$ does not depend directly on $\qs$ but only on the symmetrization $b_\qs$. That is, it only depends on the invariants $\widetilde{q}_{ij}=q^{t^\qs_{ij}}$ of the Nichols algebra rather than on the choice of $q_{ij}=q^{u^\qs_{ij}}$.
However, the braiding of $\cA_\qs$, a priori, depends on the choice of $\qs$ via the dual $R$-matrix $r_\qs\colon G\times G\to \Bbbk^\times$ from \eqref{eq:supera-Rmatrix}. 

Assume that the generator $g_i$ has order $N_i$, for $i=1,\ldots, r$. Then at least one of the $N_i$ has to be even, since $q_{ii}=-1$. By \Cref{rem:super-A-diagram}, we have that $q_{i-1,ii}$ and $q_{ii+1}$ are equal to $q^{\pm 1}$ and all other $\widetilde{q}_{ij}=1$. Thus, the order $N$ of $q$ divides $N_i$. But then $g_i^N$ is in the radical of $b_\qs$ and, for non-degeneracy, $g_i^N=1$. Hence, $N=N_i$ is even, say, $N=2n$. Now, for $b_\qs$ to be a well-defined pairing we need that the order of $g_j$, for any $j=1,\ldots, r$, is divisible by $N$. Using the same reasoning, for non-degeneracy of $b_\qs$, the order of each $g_j$ needs to be $N$. We now distinguish two cases
\begin{enumerate}
\item[(1)] In the case that $r$ is even, by non-degeneracy of $b_\qs$ over the group $\mZ_N^r$ from \Cref{cor:Aq-nondeg}, there cannot be any other relations among the generators $g_1,\ldots, g_r$. Hence, $G\cong \mZ_N^r$. 
\item[(2)] If $r$ is odd, the braiding is degenerate over the group $\mZ_N^r$ by \Cref{cor:Aq-nondeg}. A further quotient might still give a non-degenerate braiding, but we do not explore this here.
\end{enumerate}

Next, consider the dual $R$-matrix $r_\qs$. For it to be well-defined, each $q_{ij}$ has to be a root of unity of order $N=2n$. We can write $q_{ij}=q^{w_{ij}}$ for some $w_{ij}\in \mZ_N$. The requirement that 
$$\widetilde{q}_{ij}=q^{t_{ij}}=q^{w_{ij}+w_{ji}}$$
fixes the values $w_{ii}$ for all $i=1,\ldots, r$. 
We further have that 
\begin{align*}
w_{ii+1}+w_{i+1i}&=\pm 1 \mod N,&&\text{for $i=1,\ldots, r-1$,}\\
w_{ij}+w_{ji}&=0 \mod N, && \text{for $j\neq i-1,i$.}
\end{align*}
We claim that all choices of such ${\bf w}=(w_{ij})$ give equivalent braided monoidal categories. For this, denote $\qs'=(q^{w_{ij}})\in \mZ^{\mI\times \mI}$ and $\qs=(q^{u_{ij}})$ as in \Cref{eq:supera-matrix-power2}. We equip the identity functor $I\colon \cA_\qs\to\cA_{\qs'}$ with the monoidal structure
$$ 
\mu^I_{\Bbbk_{g_i},\Bbbk_{g_j}}\colon I(\Bbbk_{g_i})\otimes(\Bbbk_{g_j})\isomorph I(\Bbbk_{g_i}\otimes \Bbbk_{g_j}),
$$
where $\Bbbk_{g_i}$ denotes the one-dimensional simple $\Bbbk G$-comodule concentrated in degree $g_i$. 
Since the tensor product of both $\cA_\qs$ and $\cA_{\qs'}$ is  $\Bbbk_{g_i}\otimes \Bbbk_{g_j}=\Bbbk_{g_ig_j}$, it follows that $\mu^I_{\Bbbk_{g_i},\Bbbk_{g_j}}= \lambda(g_i,g_j)\ide$  is a multiple of the identity. Setting 
$$\lambda(g_i,g_j):=\begin{cases}(q'_{ij})^{-1}q_{ij}=q^{u_{ij}^\qs-w_{ij}}, & i<j,\\
1, & i\geq j,
\end{cases}
$$
induces a well-defined pairing
$$\lambda\colon \Bbbk G\times \Bbbk G\to \Bbbk^\times.$$
and hence a well-defined monoidal structure for the functor $I$. 

Next, we check that $I$ is a braided monoidal functor, see \cite{EGNO}*{Definition 8.1.7} for the definition. Using the braiding from \Cref{eq:braiding-graded}, this is equivalent to the equations
\begin{equation}
\lambda(g_i,g_j)q'_{ij}=q_{ij}\lambda(g_j,g_i),\label{eq:condition-braided}
\end{equation}
for all $i,j=1,\ldots, r$. With the above choice of $\lambda$, both sides of \eqref{eq:condition-braided} evaluate to $q_{ij}$ when $i<j$. Similarly, when $i=j$, \eqref{eq:condition-braided} holds since $q_{ii}=q_{ii}'$. When  $j<i$, \eqref{eq:condition-braided} evaluates to 
$$q'_{ij}=q_{ij} (q'_{ji})^{-1}q_{ji} \quad \Longleftrightarrow\quad  q'_{ij}q'_{ji}=q_{ij} q_{ji},$$
which holds since both $\qs$ and $\qs'$ realize the same Nichols algebra datum. Thus, we have shown that all choices $\qs$ of realizations of the Nichols algebra datum over the group $G=\mZ_N^r$ give equivalent braided monoidal categories $\cA_\qs$.

We have shown that if $r$ is even, then, up to equivalence of braided monoidal structures, the choices of $G=\mZ^r_N$, $q$ a root of unity of even order $N$, and the realization of \eqref{subsubsec:setup} are the only possibilities that will give a non-degenerate base category.
\end{remark}

\subsection{Characterization of unimodular Nichols algebras of type Super A} \label[section]{subsec:unimodularity}
Let $\qs$ be of type $\superqa{r}{q}{\mJ}$ as in \eqref{eq:setup}. Here we assume that $\ord(q)=N=2n$ is even.\footnote{If $N$ is odd, we can realize $\BB_\qs$ over $(\mZ/2N\mZ)^r$ but then we cannot achieve non-degeneracy of the base category. Thus, we do not consider this case here.}

We first recall a general observation about the distinguished grouplike elements of the bosonization $\BB_\qs\rtimes \Bbbk[G]$ of a Nichols algebra due to \cite{Bur}. Here, we use the statement in the form of \cite{LW2}*{Lemma~5.14}. Recall that $g_\ell$ denotes the $G$-degree of the top $\mZ$-degree element $x_\ell$ of $\BB_\qs$. Let us denote by $\delta_{{\bf i}_\ell} \in \Bbbk[G]$ the dual element to $g_\ell$, with ${\bf i}_\ell \in \Lambda=\mZ^{\times r}_N$.

\begin{lemma}\label[lemma]{lem:unimodularNichols}
The distinguished grouplike elements of $H=\BB_\qs\rtimes \Bbbk[G]$ are given by 
\begin{align*}
    g_H=1\otimes \gamma_{{\bf i}_\ell}, \qquad \alpha_H(x\otimes \delta_{\bf i})=\varepsilon(x)\delta_{{\bf i},-{\bf i}_\ell}.
\end{align*}
In particular, $\lmod{H}$ is unimodular if and only if ${\bf i}_\ell=0$, which then implies that $g_H=1$.
\end{lemma}

\begin{lemma}\label[lemma]{lem:gtop}
The group degree $g_\ell$ of the top $\mZ$-degree element $x_\ell$ of $\BB_\qs$ is given by 
\begin{align}
    g_\ell &= \prod_{j\in \mJ} g_j \prod_{i\in \mI\setminus \mJ} g_i^{N-1} \prod_{i<j \in \mI} (g_{i}\ldots g_j)^{n_{ij}}, &\text{for} &&
    n_{ij}
&=\begin{cases} 
1 & \text{if $|\mI_{i,j}\cap \mJ|$ is odd},\\
-1 & \text{if $|\mI_{i,j}\cap \mJ|$ is even}.
\end{cases}
\end{align}
\end{lemma}
\begin{proof}
The PBW basis from \eqref{eq:supera-PBW} implies that, up to reordering the factors, we have
\begin{align*}
    x_\ell = \prod_{j\in \mJ} x_j \prod_{i\in \mI\setminus \mJ} x_i^{N-1} \prod_{i<j \in \mI} x_{(ij)}^{n_{ij}},
\end{align*}
where
$$n_{ij}
=\begin{cases} 
1 & \text{if $\alpha_{ij}$ is an odd root}\\
N-1 & \text{if $\alpha_{ij}$ is an even root}
\end{cases}
\quad =\quad 
\begin{cases} 
1 & \text{if $|\mI_{ij}\cap \mJ|$ is odd}\\
N-1 & \text{if $|\mI_{ij}\cap \mJ|$ is even}.
\end{cases}
$$
This implies the claimed formula, using $g_i^{N}=1$.
\end{proof}

Now we study unimodularity of $\BB_\qs\rtimes \Bbbk[G]$. As in \Cref{sec:nondeg}, we start with a particular case.

\begin{lemma}\label[lemma]{lem:unimodularfamily}
Assume that  $\mI=\mJ$ and $\ord(q)=2n, n>1$. Then the bosonization $\BB_\qs\rtimes \Bbbk[G]$ is unimodular if and only if $r$ is even.
\end{lemma}
\begin{proof}
Under use of \Cref{lem:unimodularNichols}, we prove the claim by investigating when $g_\ell=1$. In order to verify the claim, we count the number of \emph{connected} subsets of $\mI$ (i.e, of the form $\mI_{k,j}$ with $1\leq k \leq j\leq r$) which contain a fixed vertex $i\in \mI$ of prescribed length~$l$. These subsets correspond to root vectors of $\mZ$-degree~$l$ whose $G$-degrees contribute a power of $g_i$ to $g_\ell$.

Consider the case when $i\leq \lfloor r/2\rfloor$. 
\begin{enumerate}
    \item For $l=1,\ldots, i$ there exist $l$ such subsets, namely, the subsets $(i-a,\ldots, i-a+l-1)$ for any $a=0,\ldots, l-1$. 
    \item For $l=i+1,\ldots, r-i+1$, there are $i$ such subsets, namely, the subsets $(i-a,\ldots, i-a+l-1)$ for any $a=0,\ldots, i-1$.
    \item For $l=r-j+1$, with $1\leq j\leq i-1$, there are $j$ such subsets, namely, the subsets $(i-a,\ldots, i-a+l-1)$ for any $i-j\leq a \leq i-1$.
\end{enumerate}
Hence, there are equally many subsets of $\mI$ containing $i$ of length $l$ as of length $r-l+1$.
The same symmetry holds for the subsets containing a fixed vertex $i> \lfloor r/2\rfloor$.

Using Lemma \ref{lem:gtop}, we can now compute the power $a_i$ of $g_i$ in $g_\ell$. If $r$ is even, then for any $1 \leq l\leq  r/2$, $l$ is even if and only if $r-l+1$ is odd. Thus, as there is an equal number of subsets of size $l$ and $r-l+1$ contributed by the above list, the power $g_i^{\pm 1}$ contributed by a subset of length $l$ cancels with the power contributed by a subset of length $r-l+1$. Thus, $a_i=0$. However, if $r$ is odd, then 
\begin{align*}
    a_1=\sum_{j=1}^{r}(-1)^{j}=-1.
\end{align*}
and hence $g_\ell \neq 1$. Thus, $g_\ell =1$ if and only if $r$ is even.
\end{proof}

\begin{proposition}\label[proposition]{prop:unimodularity} Let $\qs$ as in \eqref{eq:setup}.
The bosonization $\BB_\qs\rtimes \Bbbk[G]$ of the Nichols algebra $\BB_\qs$ of type Super A is unimodular if and only if  $q_{ii}=-1$ for all $i$ (i.e. $\mJ=\mI$) and $r$ is even.
\end{proposition}
\begin{proof} 
We first establish the forward implication using a series of steps. For this, assume that $\BB_\qs\rtimes \Bbbk[G]$ is unimodular. Then $g_\ell=1$  for the group degree 
$$g_\ell=g_1^{a_1}\ldots g_r^{a_r}$$ of the integral $x_\ell$ of $\BB_\qs\rtimes \Bbbk[G]$.

\emph{Step~0:} 
We consider the collection of all connected subsets of $\mI$ which contain a fixed element $j\in \mI$, denoted by $\mathrm{Subsets}_j(\mI)$. We can decompose this collection of subsets into $j$ subcollections that are all in bijection. Namely, consider for $1\leq k \leq j$ the subcollection
$$\mathrm{Subsets}_{j}^{\geq k}(\mI)=\Set{\mI_{k,t} \middle| k\leq j\leq t\leq r},$$
i.e. all connected subsets of $\mI$ with minimum $k$ that contain $j$.
For $k<j$, the bijection
\begin{align}\label{subcollection_bij}
\mathrm{Subsets}_{j}^{\geq k}(\mI)\isomorph \mathrm{Subsets}_{j}^{\geq k+1}(\mI)
\end{align}
is given by removing $k$ from the subset. The union 
\begin{align}\label{subcollection_union}
\mathrm{Subsets}_j(\mI)=\bigsqcup_{k=1}^j\mathrm{Subsets}_{j}^{\geq k}(\mI)
\end{align}
gives all connected subsets of $\mI$ that contain $j$. 

By \Cref{lem:gtop}, each subset $S$ in $\mathrm{Subsets}_j$ contributes a  factor $g_j^{\pm 1}$ to $g_\ell$, and the power $a_j$ of $g_j$ in $g_\ell$ is the product of these factors.
Thus, using the partition in \eqref{subcollection_union} we can decompose \emph{this power $a_j$} as 
$$g_j^{a_j}=\prod_{k=1}^j g_j^{a_j^{\geq k}}=g_j^{\sum_{k=1}^j a_j^{\geq k}},$$
where $ g_j^{a_j^{\geq k}}$ summarizes the factors contributed by subsets $S\in \mathrm{Subsets}_{j}^{\geq k}(\mI)$.


 Assume given a generalized Dynkin diagram of type Super A and rank $r$ such that $g_\ell=1$. The proof of the forward implication want to show that $s\in \mJ$, for all $1\leq s\leq r$. We enumerate the odd vertices by $\mJ=\{i_1 <i_2< \ldots< i_k\}$. 

\emph{Step~1:} Consider the power $a_1$ of $g_1$ in $g_\ell$. The subsets of $\mI$ containing 1 are the sets $\mI_{1,j}$ for $1\leq j \leq r$. Using \Cref{lem:gtop}, we thus have 
\begin{align*}
    a_1=\sum_{j=0}^{k}(-1)^{j}(i_{j+1}-i_{j}),
\end{align*}
where $i_0:=1$ and $i_{k+1}:=r+1$.

\emph{Step 2:} Consider the power $a_2$ of $g_2$ in $g_\ell$. In Step~0, we have seen that 
$$\mathrm{Subsets}_2(\mI) = \mathrm{Subsets}_{2}^{\geq 1}(\mI)\sqcup \mathrm{Subsets}_{2}^{\geq 2}(\mI).$$
Thus, $a_2=a_2^{\geq 1}+a_2^{\geq 2}\mod N$. Hence, $a_2=0\mod N$ if and only if $a_2^{\geq 1}=-a_2^{\geq 2}\mod N$. Assume for a contradiction that $1\notin \mJ$. Then for any $j$, we have
$$|\mI_{1,j}\cap \mJ|=|\mI_{2,j}\cap \mJ|$$
which implies that $a_2^{\geq 1}=a_2^{\geq 2}$. Thus, $2a_2^{\geq 1}=0\mod N=2n$. Now, observe the bijection
$$\mathrm{Subsets}_2^{\geq 1}(\mI) = \mathrm{Subsets}_1(\mI)\setminus \{ \mI_{1,1}\}.$$
This gives $a_1\pm 1=a_2^{\geq 1} \mod N$, thus $2a_1\pm 2=2a_2^{\geq 1} \mod N$. As both $2a_1$ and $2a_2^{\geq 1}$ vanish modulo $N$, we find that $2=0\mod N$, which contradicts $N>2$. Hence, $1\in \mJ$.

\emph{Step 3:}  Consider the power $a_3$ of $g_3$ in $g_\ell$. In Step~0, we have seen that 
$$\mathrm{Subsets}_3(\mI) = \mathrm{Subsets}_{3}^{\geq 1}(\mI)\sqcup \mathrm{Subsets}_{3}^{\geq 2}(\mI)\sqcup \mathrm{Subsets}_{3}^{\geq 3}(\mI).$$
Hence $a_3=a_3^{\geq 1}+a_3^{\geq 2}+a_3^{\geq 3}\mod N$. Since $1\in \mJ$, we know that $a_3^{\geq 1}+a_{3}^{\geq 2}=0\mod N$. 
Thus, $a_3=0\mod N$ implies that  $a_3^{\geq 3}=0\mod N$.

With the notation from Step~1, we concluded in Step~2 that $i_1=1$, and if $2\notin \mJ$, then $i_2\geq 3$. We can determine $a_3^{\geq 3}$, using the notation of Step~1 but restricting to the subdiagram $\mI_{3,r}$. Thus, we get
$$0=a_3^{\geq 3}=\sum_{j=2}^k(-1)^j (i_{j+1}-i_j)+(-1)(i_2-3) \mod N.$$
Together with the calculation in Step~1 this gives that 
$$\sum_{j=2}^k(-1)^j (i_{j+1}-i_j)+(-1)(i_2-3)=\sum_{j=0}^{k}(-1)^{j}(i_{j+1}-i_{j})\mod N.$$
This in turn implies
$$(-1)(i_2-3)=i_1-i_0 +(-1)(i_2-i_1)=1-1+(-1)(i_2-1)= (-1)(i_2-1) \mod N.$$
Thus, $i_2=i_2+2\mod N$ which contradicts $N>2$. Thus, $2\in \mJ$.

\medskip
\emph{Step~4:} 
We now claim that for any $2\leq s< r$ if $1,\ldots, s-1\in \mJ$, then $s\in \mJ$, and prove this claim by induction on $s$. The induction base was established in Steps~1--3.

For the induction step, consider the coefficient $a_{s+1}$ of $g_{s+1}$ in $g_\ell$. Using the partition of $\op{Subsets}_{s+1}(\mI)$ from Step~0, we see that 
$$a_{s+1}=a_{s+1}^{\geq 1}+\ldots +a_{s+1}^{\geq s+1}\mod N.$$
Since $1,\ldots, s-1\in \mJ$, we get that 
$$a_{s+1}^{\geq j}+a_{s+1}^{\geq j+1}=0 \mod N, \qquad \text{for all $j=1,\ldots, s-1$.}
$$

\emph{Step~4.1:} Assume that 
$s$ is even. Then we have that $a_{s+1}=0\mod N$ if and only if $a_{s+1}^{\geq s+1}=0\mod N$. As $1,\ldots, s-1\in \mJ$, we have that $i_1=1,\ldots, i_{s-1}=s-1$ using the notation of Step 1. Assume that $s\notin \mJ$. Then $i_s\geq s+1$ and as in Step 3, we find that 
$$\sum_{j=s}^k(-1)^j (i_{j+1}-i_j)+(-1)(i_s-(s+1))=\sum_{j=0}^{k}(-1)^{j}(i_{j+1}-i_{j})\mod N$$
Using $i_1=1,\ldots, i_{s-1}=s-1$, the right hand side simplifies to 
$$\sum_{j=s}^k(-1)^j (i_{j+1}-i_j)+(-1)(i_s-(s-1)).$$
Thus, $i_s=i_s+2 \mod N$ which contradicts that $N>2$. Hence, $s\in \mJ$.

\emph{Step~4.2:} Assume that 
 $s$ is odd. Then we have that  $a_{s+1}=0\mod N$ if and only if $a_{s+1}^{\geq s}+a_{s+1}^{\geq s+1}=0\mod N$. Arguing as in Step~2, we find that $2a_{s+1}^{\geq s}=0\mod N$. Assume for a contradiction that $s\notin \mJ$ so that $i_s>s$. We compute
 $$a_{s+1}^{\geq s}=\sum_{j=s}^k(-1)^j (i_{j+1}-i_j)+(i_s-(s+1))=0\mod N.$$
 This implies that 
 $$2\left(\sum_{j=s}^k(-1)^j (i_{j+1}-i_j)+(i_s-(s+1))\right)=2\sum_{j=0}^{k}(-1)^{j}(i_{j+1}-i_{j})\mod N.$$
 The right-hand side simplifies to 
 $$2\sum_{j=s}^{k}(-1)^{j}(i_{j+1}-i_{j}) +2(i_s-(s-1))-2.$$
 It follows that 
 $2i_s-2(s+1)=2i_s-2(s-1)-2\mod N$, which implies the contradiction that $2=0\mod N$. Hence, $s\in \mJ$. 
This completes the induction step.

\emph{Step~5:} It remains to show that if $1,\ldots, r-1\in \mJ$, then $r\in \mJ$ (since the case $s=r$ was excluded in Step~4). 
Since $r>1$, we use the same argument as in Step~2 but with reversed coefficient order, considering the coefficient of $g_{r-1}$ instead of $g_2$, to show that $r\in \mJ$. (Note that the case of rank $r=1$ gives a Nichols algebra of Cartan type $A_1$, which is not unimodular; thus $g_\ell=1$ is not possible in this situation.)

Hence, we have shown that  $g_\ell=1$ implies that $\mJ=\mI$. The conclusion that, in this case, $r$ is even and the converse implication both follow from  Lemma~\ref{lem:unimodularfamily}.
Indeed, if $\mJ=\mI$ and $r$ is even, then $\BB_\qs\rtimes \Bbbk[G]$ is unimodular for any $\ord(q)=2n>2$. 
\end{proof}

\subsection{Classification of non-semisimple spherical structures for the bosonization}\label[section]{sec:spherical}

In this section, we characterize when categories of modules over the bosonizations of Nichols algebras of type Super A admit a spherical structure in the sense of \cite{DSS}.
We again let $N=2n> 2$ be an even integer and consider the group $G=\langle g_1,\ldots, g_r | g_1^{N}=\ldots =g_r^N=1\rangle $. 
Fixing a primitive $N$-th root of unity $q$, we can construct an isomorphism of Hopf algebra 
\begin{equation}\label{eq:kappai}
    \kappa\colon \Bbbk G\longrightarrow \Bbbk[G], \quad g_i\longmapsto \kappa_i:=\sum_{{\bf j}=(j_1,\ldots, j_r)}q^{j_i}\delta_{\bf j}.
\end{equation}
For ${\bf i}\in \Lambda$ we denote $\kappa_{\bf i}:= \kappa_1^{i_1}\cdot \ldots \cdot \kappa_{r}^{i_r}$ and all grouplike elements in $\Bbbk[G]$ are of this form.

\begin{theorem}\label{theorem:spherical}
Let $\qs$ as in \eqref{eq:setup}.
The category $\lmod{H}\simeq \lmod{\BB_\qs}(\cA_\qs)$, for 
$H:=\BB_\qs\rtimes \Bbbk[G]$, admits a spherical structure if and only if $\lmod{H}$ is unimodular if and only if $\qs$ is of type $\superqa{r}{q}{\mI}$ for  $r\geq 2$ even. In this case, the spherical structure is determined by the pivotal element 
$$a=\sum_{\bf j} (-1)^{j_1+\ldots+j_r}\delta_{\bf j}=\kappa_1^n\ldots \kappa_r^n\in \Bbbk[G].$$
\end{theorem}
\begin{proof}
By \cite{LW2}*{Proposition~3.12}, $\lmod{H}$ being spherical implies that $\alpha_H=\varepsilon$ (i.e.~$\lmod{H}$ is unimodular) and in this case spherical structures on $\lmod{H}$ are in bijection with grouplike elements $a\in G(H)$ such that $a^2=g_H$ and
$$S^2(h)=aha^{-1}, \qquad \forall h\in H.$$
By \Cref{prop:unimodularity}, $\lmod{H}$ is unimodular if and only if $\mJ=\mI$ and $r$ is even. In this case, we further observe that $g_H=1$, cf. \Cref{lem:unimodularNichols}. Thus, it remains to classify spherical structures in the case that $\qs$ is of type $\superqa{r}{q}{\mI}$. We use \Cref{prop:Hspherical} and assume given an element $a=\kappa_{\bf a}\in \mathsf{SPiv}(H)$, i.e., a grouplike element $a$ such that $a^2=1$ and $S^2(h)=aha^{-1}$, for all $h\in H$. The latter condition is equivalent to $S^2(x_i)=ax_ia^{-1}$ for all $i\in \mI$. Now observe the following equivalences, using that $N=2n$:
\begin{align*}
    a^2&=1 &&\Longleftrightarrow&  \forall i\in \mI:& 2 a_i=0\mod 2n,\\
    \forall i\in \mI: S^2(x_i)&=ax_ia^{-1}, &&\Longleftrightarrow& \forall i\in \mI:& q_{ii}^{-1}x_i=-x_i=q^{a_i}x_i\\
    &&&\Longleftrightarrow& \forall i\in \mI:& -1=q^n=q^{a_i}\\
    &&&\Longleftrightarrow& \forall i\in \mI:& a_i=n \mod 2n.
\end{align*}
Here, we apply that $q_{ii}=-1$ for all $i$ and the equation 
$$\kappa_ix_j=\begin{cases}
x_j \kappa_i, & i\neq j,\\
q  x_j \kappa_i, & i=j.
\end{cases}
$$
Thus, the unique spherical structure for $H$ of type $\superqa{r}{q}{\mI}$ is given by $a=\kappa_{\bf a}$, with ${\bf a}=(n,\ldots, n)\in \Lambda=\mZ_{2n}^{r}=\mZ_{N}^{r}$.
\end{proof}
In the case $r=2$ we recover the spherical structure found in \cite{LW2}*{Example~5.18}.

\begin{corollary}\label[corollary]{cor:modularity}
    If $\mJ=\mI$ and $r$ is even, the categories $\cZ(\lmod{\BB_\qs}(\cA_\qs))$ and $\cZ_{\cA_\qs}(\lmod{\BB_\qs}(\cA_\qs))$ are modular categories.
\end{corollary}
\begin{proof}
    These results are consequences of \Cref{theorem:spherical} by \cite{Shi2}*{Theorem~5.11} for the Drinfeld center, and by \cite{LW2}*{Corollary~4.16} and \Cref{prop:base-I=J-non-degeneracy} for the relative center.
\end{proof}

\subsection{Classification of ribbon structures for the Drinfeld double}
\label[section]{sec:ribbon}

Again assuming  $N=2n> 2$, $q$ a primitive $N$-th root of unity, and $G=\langle g_1,\ldots, g_r | g_i^{N}=1\rangle $, we now classify ribbon structures for the braided category $\cZ(\lmod{\BB_\qs \rtimes \Bbbk G})$, where $\BB_\qs$ is a Nichols algebra of type $\superqa{r}{q}{\mJ}$, with $\mJ\subseteq \mI$.
Recall the distinguished grouplike elements $g_H$ and $\alpha_H$ of $H\coloneq\BB_\qs \rtimes \Bbbk G$ from \Cref{lem:unimodularNichols}. 

\begin{proposition}\label[proposition]{prop:ribboncond} Regard the matrix ${\bf u}^\qs=(u_{ij})$ defined in \eqref{eq:supera-matrix-power2} as having coefficients in $\mZ/N\mZ$ and denote ${\bf u}_\Delta=(u_{11},\ldots, u_{r r})^t$. Then the Hopf algebra $\Drin(H)$ admits a ribbon structure if and only if the top degree ${\bf i}_\ell$  of $\BB_\qs$ has even entries and satisfies
\begin{equation}
     ({\bf Id}+ {\bf u}^\qs ){\bf i}_\ell=-2{\bf u}_\Delta \mod N.\label{ribboncondition}
\end{equation}
\end{proposition}
\begin{proof}
    By \cite{KR93}, ribbon structures of $H$ are classified by the set of pairs $(a,\zeta)$ of group-like elements in $G(H), G(H^*)$ satisfying that 
\begin{gather}
    a^2=g_H, \qquad \zeta^2=\alpha_H, \label{eq:squares}\\
   S^2(h)=\zeta^{-1}(h_{(1)})ah_{(2)}a^{-1}\zeta(h_{(3)}), \qquad \text{ for all }h\in H,    \label{eq:S2}
\end{gather}
cf. also \cite{LW2}*{Theorem~3.6}. A general group-like element is of the form 
$$a=\kappa_{\bf a}, \qquad \text{ with }{\bf a}=(a_1,\ldots, a_r)\in \Lambda,$$ 
using the elements $\kappa_{\bf i}$ defined in \eqref{eq:kappai}. The element $\zeta\in H^*$ is of the form 
$$\zeta(x\otimes \delta_{\bf i})=\varepsilon(x)\delta_{{\bf i},{\bf j}}, \qquad \text{ with }{\bf j}=(j_1,\ldots, j_r)\in \Lambda.$$
Thus, \eqref{eq:squares} is equivalent to 
\begin{equation}\label{eq:ajtimes2}
2{\bf a}={\bf i}_\ell \qquad \text{ and } \qquad 2{\bf j}=-{\bf i}_\ell.
\end{equation}
Similarly to \cite{LW2}*{Proposition~5.15(ii)} one computes that \eqref{eq:S2} is equivalent to 
\begin{align}
    q_{ii}^{-1}=r_{\qs}^{-1}(g_{\bf j},g_i)q^{a_i}, \qquad \text{for all }i=1,\ldots, r.
\end{align}
The latter equations are equivalent to 
\begin{align}\label{eq-q-conditions}
    q_{ii}^{-1}=q_{i,1}^{-j_1}\ldots q_{i,r}^{-j_r}q^{a_i}&
    =\begin{cases} 
    q_{i,i}^{-j_i}q_{i,i+1}^{-j_{i+1}}q^{a_i}, \qquad &\text{if } i=1,\ldots, r-1,\\
    q_{r,r}^{-j_r}q^{a_r},& \text{if $i=r$},
    \end{cases}
\end{align}
following the conventions of \eqref{eq:setup}. Now, regarding ${\bf u}^\qs=(u_{ij})$ as a matrix with coefficients in $\mZ/N\mZ$, the conditions of \eqref{eq-q-conditions} are equivalent to the following system of linear equations modulo $N$
\begin{align}
    u_{ii}+a_i-\sum_{k=1}^r u_{ik}j_k=0, \qquad \forall i=1,\ldots, r.
\end{align}
Denoting ${\bf a}=(a_1,\ldots, a_r)^t$ and ${\bf j}=(j_1,\ldots, j_r)^t$, this system of linear equations can be rewritten as the matrix equation
\begin{align}\label{eq-q-conditions-mat}
    {\bf a}=-{\bf u}_\Delta+{\bf u}^\qs {\bf j} \, \mod N.
\end{align}
Thus, ${\bf a}$ is uniquely determined by ${\bf j}$.

Equation \eqref{eq-q-conditions-mat} implies the condition given in \Cref{ribboncondition} for the existence of ribbon structures. Assume that  $2{\bf j}= -{\bf i}_\ell \mod N$. Then $2{\bf a}={\bf i}_{\ell} \mod N$ if and only if 
\begin{gather*}
 {\bf i}_{\ell}=2{\bf a}=2(-{\bf u}_\Delta+{\bf u}^\qs {\bf j})=-2{\bf u}_\Delta+2{\bf u}^\qs {\bf j} =-2{\bf u}_\Delta - {\bf u}^\qs {\bf \bf i}_\ell\nonumber\\
\Longleftrightarrow \qquad 
 ({\bf Id}+ {\bf u}^\qs ){\bf i}_\ell=-2{\bf u}_\Delta \mod N.
\end{gather*}
Conversely, if \Cref{ribboncondition} holds and all entries of ${\bf i}_\ell$ are even, then  \Cref{eq-q-conditions-mat} implies that any choice of ${\bf a}\in \Lambda$ which solves $2{\bf a}={\bf i}_\ell$ gives an element ${\bf j}$ satisfying $2{\bf j}={\bf i}_\ell$, and thus a ribbon element for the Drinfeld double.
\end{proof}

The above proposition shows that for a given $\qs$ of type Super A, to find ribbon structures on $ \cZ(\lmod{\BB_\qs}(\cA_\qs))\simeq \lmod{\Drin(H)}$ we need to check when ${\bf i}_\ell$ has even entries and \Cref{ribboncondition} holds. We denote ${\bf i}_\ell=(\ell_1, \dots, \ell_r)$ and observe the following technical lemma.

\begin{lemma}\label[lemma]{lem:lr-recursion}Let $\qs$ be of type $\superqa{r}{q}{\mJ}$, and fix an index $0\leq l <r$.
\begin{enumerate}
    \item[(i)] Assume that $r-i\notin\mJ$ for all $0\leq i\leq l$. Then
\begin{align*}
    \ell_{r-i}=(i+1)(\ell_r+i)\mod N, &&\forall 0\leq i\leq l+1.
\end{align*}
\item [(ii)] Assume that $r-i\in\mJ$ for all $0\leq i\leq l$. Then 
$$
\ell_{r-i}=\begin{cases}
\ell_r \mod N& \text{if $i$ is even,}\\
0\mod N & \text{if $i$ is odd.}
\end{cases}$$
\end{enumerate}
\end{lemma}
\begin{proof}
Recall the notation $\mathrm{Subset}_j(\mI)$ of all connected subsets of $\mI$ containing $j$ from the proof of \Cref{prop:unimodularity} and denote, for $j\leq i$, 
$$\mathrm{Subset}_{j}^{\leq i}=\Set{\mI_{a,i}\,|\, a\leq j\leq i},$$
i.e., the connected subsets that contain $j$ with $i$ as maximum.
For any $i\leq l+1$, we have a partition 
$$\mathrm{Subset}_{r-i}(\mI)=\mathrm{Subset}_{r-i}^{\leq r-i}\sqcup \ldots \sqcup \mathrm{Subset}_{r-i}^{\leq r}.$$
Recall the proof of \Cref{lem:gtop} which explains how $g_{r-1}^{\ell_{r-i}}$, and hence $\ell_{r-i}$, is computed by contributions from the subsets contained in $\mathrm{Subset}_{r-i}(\mI)$. 
The above partition gives that $\ell_{r-i}=\sum_{j=0}^i\ell_{r-i}^{\leq r-j}$, where the term $\ell_{r-i}^{\leq r-j}$ accounts for the contribution of subsets in $\mathrm{Subset}_{r-i}^{\leq r-j}$. For later use, note that we have bijections 
$$\mathrm{Subset}_{r-i}^{\leq r-j}\cong \mathrm{Subset}_{r-i}^{\leq r-j+1}, \qquad \mI_{a,r-j}\mapsto \mI_{a,r-j+1},$$
given by adding the element $r-j+1$ to a subset of the form $\mI_{a,r-j}$.

Under the assumptions of Part (i), 
for $0\leq j\leq i$, all of the $\ell_{r-i}^{\leq r-j}$ are equal. This follows from the bijections $\mathrm{Subset}_{r-i}^{\leq r-j}\cong \mathrm{Subset}_{r-i}^{\leq r-j+1}$ and the observation $|\mI_{a,r-j}\cap \mJ|=|\mI_{a,r-j+1}\cap \mJ|$ since $r-j+1\notin\mJ$. Thus,
$$ \ell_{r-i}=(i+1)\ell_{r-i}^{\leq r}=(i+1)(\ell_r+i)\mod N,$$
where the last equality uses that
\begin{align}\label{subsetreduce}
\mathrm{Subset}_{r-i}^{\leq r}=\mathrm{Subset}_{r}\setminus\Set{\mI_{r,r}, \ldots, \mI_{r-i+1,r}},
\end{align}
and that $|\mI_{r-j+1,r}\cap \mJ|=0$ for $j\leq l+1$ by assumption.

Under the assumptions of Part (ii), we still have the same bijections 
$\mathrm{Subset}_{r-i}^{\leq r-j}\cong \mathrm{Subset}_{r-i}^{\leq r-j+1}$ but now 
$\ell_{r-i}^{\leq r-j}=-\ell_{r-i}^{\leq r-j+1}$ since $|\mI_{a,r-j}\cap \mJ|+1=|\mI_{a,r-j+1}\cap \mJ|$. Thus, if $i$ is odd, we have that 
$$\ell_{r-i}=\sum_{j=0}^i \ell_{r-i}^{\leq r-j}= (\ell_{r-i}^{\leq r-i}+\ell_{r-i}^{\leq r-i+1})+\ldots+ (\ell_{r-i}^{\leq r-1}+\ell_{r-i}^{\leq r})=0\mod N.$$
If $i$ is even, we have that 
$$\ell_{r-i}=\ell_{r-i}^{\leq r}=\ell_r.$$
The last equality uses \Cref{subsetreduce} and the fact that there is an even number of sets
$$\mI_{r,r}, \ldots, \mI_{r-i+1,r}$$
as $i$ is even whose contributions to $\ell_r$ cancel out.
\end{proof}

For the main theorem of this section, we use \Cref{lem:lr-recursion} to characterize which Nichols algebras of type Super A satisfy the conditions for ribbonality of the Drinfeld double found in \Cref{prop:ribboncond}.
\begin{theorem}\label{theorem:ribbon}
Let $\qs$ as in \eqref{eq:setup}.
The category $\cZ(\lmod{H})\simeq \cZ(\lmod{\BB_\qs}(\cA_\qs))$, for 
$H:=\BB_\qs\rtimes \Bbbk[G]$, admits a ribbon structure if and only if $\lmod{H}$ in unimodular if and only if $\qs$ is of type $\superqa{r}{q}{\mI}$ for $r\geq 2$ even. 
In this case, there exist precisely $2^r$ ribbon structures only one of which corresponds to the spherical structure from \Cref{theorem:spherical}.
\end{theorem}
\begin{proof}
We first prove the forward inclusion, assuming that the Drinfeld double of $H=\BB_\qs\rtimes\Bbbk[G]$ associated with $\qs$ admits a ribbon structure and conclude that this is only possible in the unimodular case. Denoting ${\bf i}_\ell=(\ell_1, \dots, \ell_r)$, \Cref{prop:ribboncond} implies that for the existence of a ribbon structure, $\ell_i$ is even for all $i=1,\ldots, r.$
Further, 
\Cref{ribboncondition} implies the necessary conditions
\begin{align}\label{eq:ellr}
(u_{rr}+1)\ell_r&=-2u_{rr}\mod N, \\ \label{ribboncondition2}
    (u_{i i}+1)\ell_{i}+ u_{i,i+1}\ell_{i+1}&=-2u_{ii}\mod N
\end{align}
 for the existence of a ribbon structure. Recall from \Cref{rem:super-A-diagram}  and \eqref{eq:supera-matrix-power2} that $u_{ii}=n$ if $r\in \mJ$, and $u_{ii}=\pm 1$ otherwise.

Next, we use these conditions \eqref{eq:ellr}, \eqref{ribboncondition2}  to show that the existence of ribbon structures implies that $\mI=\mJ$.
We will start by considering the $r$-th vertex of $\mI$ in a generalized Dynkin diagram
of type $\superqa{r}{q}{\mJ}$ and proceed with vertices $r-1,r-2, \ldots$ in order to rule out cases in which ribbon structures cannot exist.
\begin{enumerate} 
    \item[(1)] \emph{The case $r\notin \mJ$:} If $r\notin \mJ$, then we have $u_{rr}=\pm 1$, thus \eqref{eq:ellr} becomes
\begin{align*}
    (\pm 1 + 1)\ell_r =\mp 2 \mod N.
\end{align*}
If $u_{rr}=-1$, this equation has no solution as the left-hand side is zero and $N>2$. Thus, we may assume that $u_{rr}=1$ and \Cref{eq:ellr} implies that 
$$
2\ell_r=-2 \mod N.
$$
In this case, there are two subcases, using that $r>1$.
    \item[(1.1)] \emph{The case $r,r-1\notin \mJ$:}
    In this case, we derive from \Cref{ribboncondition} that
$$2\ell_{r-1}-\ell_r=-2\mod N,$$
and hence 
    \begin{equation}
    \ell_r=2\ell_{r-1}+2\mod N.\label{eq:ellr-1}.
    \end{equation}
    However, we can also compare $\ell_r$ and $\ell_{r-1}$ more directly. For this, denote by $$\mathrm{Subset}_{j}^{\leq k}(\mI)=\Set{\mI_{l,k}\,|\, l\leq j\leq k},$$
    i.e. the connected subsets of $\mI$ containing $l$ with maximum $k$. With this notation we have the following stratification of the connected subsets of $\mI$ containing $r-1$:
    $$\mathrm{Subset}_{r-1}(\mI)=\mathrm{Subset}_{r-1}^{\leq r}(\mI)\sqcup \mathrm{Subset}_{r-1}^{\leq r-1}(\mI).$$
    Further, appending $\Set{r}$ to a subset with maximum $r-1$ shows that $\mathrm{Subset}_{r-1}^{\leq r-1}(\mI)$ and $\mathrm{Subset}_{r-1}^{\leq r}(\mI)$ are in bijection.
    This implies that 
    $$\ell_{r-1}=\ell_{r-1}^{\leq r}+\ell_{r-1}^{\leq r-1}=2\ell_{r-1}^{\leq r}\mod N,$$
    where $\ell_{r-1}^{\leq r-i}$ denotes the contribution in $\ell_r$ coming from subsets with maximum $r-i$, for $i=0,1$.  Moreover, there is an evident equality of subsets of $\mI$
    $$\mathrm{Subset}_{r-1}^{\leq r}(\mI)=\mathrm{Subset}_{r}(\mI)\setminus \Set{\mI_{r,r}},$$
    from which we obtain that 
    $$\ell_{r-1}^{\leq r}=\ell_r+1\mod N.$$
    Combining these equations we derive that 
    $$\ell_{r-1}=2\ell_r+2\mod N.$$
    Substituting into \Cref{eq:ellr-1} this yields that 
    $$ \ell_r=4\ell_r+6\mod N.$$
    Thus, $3\ell_r=-6\mod N$. We also know that $2\ell_r=-2\mod N$. Hence, 
    $$\ell_r=-4\mod N.$$ On the other hand, $\ell_r=-1\mod n$ which can only be solved with even $\ell_r$ if $n$ is odd. But then $\ell_r=n-1 \mod N$ since $\ell_r=N-1\mod N$ is impossible for even $\ell_r$. In particular, we obtain that $n-1=-4\mod N$ whence $n=N-3\mod N.$ Thus, 
    $6=0\mod N$ and $N=6$.

    Thus, we distinguish further subcases:
    \item[(1.1.1)] \emph{The case $r-i\notin \mJ$ for $i=0,1,2$:}
We know from Case (1.1) that $N=6$ and $\ell_r=n-1=2\mod 6$. 
Thus, we derive  from \Cref{lem:lr-recursion}(i) that
$$\ell_{r-1}=\ell_{r-2}=0\mod 6.$$
Since $r-2\notin\mJ$ in this subcase, we also have, from \Cref{ribboncondition}, that
$$0=2\ell_{r-2}-\ell_{r-1}=-2\mod 6,$$
a contradiction.
\item[(1.1.2)] \emph{The case $r,r-1\notin\mJ$ and $r-2\in \mJ$:}
We claim first that in this case all further vertices $r-i$, with $i\geq 2$, will be in $\mJ$. 
Assume for a contradiction that $r-l\notin \mJ$ but all $r-i$ with $2\leq i <l$ are in $\mJ$. Then \Cref{ribboncondition} gives that 
$$(1+n)\ell_{r-i}\pm \ell_{r-i+1}=\ell_{r-i}\pm \ell_{r-i+1}=0\mod N,$$
using in the second equality that $\ell_{r-i}$ is even. Thus, as we have seen that $\ell_{r-1}=0\mod 6$ we see inductively that $\ell_{r-i}=0\mod N$. Now, we find that 
$$(1\pm 1)\ell_{r-l}\mp \ell_{r-l+1}=\mp 2\mod 6,$$
since $r-l\notin \mJ$. This contradicts $0\neq 2\mod 6$. 

Thus, $r-i\in\mJ$ for all $i\geq2$, and the only diagrams left to consider are of the  form
$$
\vcenter{\hbox{\begin{tikzpicture}
\node [circle,draw,label=above:$-1$] (1){};
\node [circle,draw,label=above:$-1$] (2)[right of=1,node distance=1.5cm]{};
\node [circle,draw,label=above:$-1$] (3)[right of=2,node distance=1.5cm]{};
\node [circle,draw,label=above:$-1$] (4)[right of=3,node distance=1.5cm]{};
\node [circle,draw,label=above:$q$] (5)[right of=4,node distance=1.5cm]{};
\node [circle,draw,label=above:$q$] (6)[right of=5,node distance=1.5cm]{};
\draw  (1.east) -- (2.west) node [above,text centered,midway]
{$q^{\pm}$};
\draw  (2.east) -- (3.west) node [above,text centered,midway]
{$q^\mp$};
\draw[dotted] (3.east) -- (4.west) node []{};
\draw  (4.east) -- (5.west) node [above,text centered,midway]
{$q^{-1}$};
\draw  (5.east) -- (6.west) node [above,text centered,midway]
{$q^{-1}$};
\end{tikzpicture}}}
$$
For such a diagram we can compute $\ell_r$ explicitly as
\begin{align*}
    \ell_r=-2 + \sum_{i=3}^{r}(-1)^{r-1}= \begin{cases}
    -1 & \text{if $r$ is odd,}\\
    -2 & \text{if $r$ is even.}
    \end{cases}
\end{align*}
This contradicts $\ell_r=2\mod 6$.
    \item[(1.2)] \emph{The case $r\notin \mJ$ and $r-1\in \mJ$:}  
    Similarly as before, \Cref{ribboncondition} implies that 
    \begin{align*}
        2\ell_r&=-2\mod N,&
        (1+n)\ell_{r-1}-\ell_r&=0\mod N,
    \end{align*}
    and hence $\ell_{r-1}=\ell_r=n-1\mod N$.
    As $r\notin \mJ$ we find by \Cref{lem:lr-recursion}(i) that 
    $$\ell_{r-1}=2(\ell_r+1)\mod N.$$
    Combining these equations we see that 
    $n-1=2n=0\mod N,$ contradicting $N>2$.
    
Thus, if $r\notin \mJ$, there are no ribbon structures. 

\item[(2)] \emph{The case $r\in \mJ$:} If $r\in \mJ$, then $u_{rr}=n$ and \eqref{eq:ellr} simplifies to $\ell_r=0\mod N$. By induction on $l$, we show that if $r-i\in \mJ$ for all $0\leq i\leq j$, then $u_{r-i,r-i}=0\mod N$ for all $0\leq i \leq j$. This follows as \Cref{ribboncondition2} gives that
$$(n+1)\ell_{r-i}\pm \ell_{r-i+1}=\ell_{r-i}\pm \ell_{r-i+1}=0\mod N.$$
Assume for a contradiction that $r-l-1\notin\mJ$. Then 
$$(1+ \epsilon)\ell_{r-l-1}-\epsilon\ell_{r-l}=(1+\epsilon)\ell_{r-l-1} =-2\epsilon \mod N,$$
where $\epsilon=u_{r-l-1,r-l-1}=\pm 1$. If $\epsilon=-1$ we derive that $0=2\mod N$, contradicting $N>2$.

On the other hand, if $\epsilon=1$ we find that $2\ell_{r-l-1}=-2\mod N$. This implies that $\ell_{r-l-1}$, as it is even, has to equal $n-1\mod N$. However, the case $u_{r-l-1,r-l-1}=\epsilon=1$ requires that there is an even number of vertices $r,r-1,\ldots, r-l$ in $\mJ$. Thus, $l+1$ is even. In this case, \Cref{lem:lr-recursion}(ii) implies that 
$$\ell_{r-l-1}=\ell_{r-l-1}^{\leq r}=l_r=0\mod N,$$
a contradiction to $N>2$.

 In particular, we have now seen that the only  $\qs$ of type Super A for which ribbon structures could exist are of type $\superqa{r}{q}{\mI}$. In this case, we saw that $\ell_i=0\mod N$ for all $i$ which shows that ${\bf i}_\ell=0\mod N$ and, hence,the bosonization $H$ is unimodular. This implies by \Cref{prop:unimodularity} that $r$ is even.
\end{enumerate}

At this point, we have shown that if $\mJ\neq \mI$ or $r$ odd, then ribbon structures can not exist.
It remains to classify the ribbon structures for $\qs$ of type $\superqa{r}{q}{\mI}$, $r$ even. These are found by solving Equation \eqref{eq-q-conditions-mat} for ${\bf a}$ and ${\bf j}$ satisfying $2{\bf a}={\bf 0}\mod N$ and $2{\bf j}={\bf 0}\mod N$. Hence, there exist tuples 
$${\bf \varepsilon}=(\varepsilon_1,\ldots, \varepsilon_r), \qquad {\bf \eta}=(\eta_1,\ldots, \eta_r)\in (\mZ/2\mZ)^r$$
satisfying 
$$a_i=n\varepsilon_i, \qquad j_i=n\eta_i, \qquad \forall i=1,\ldots, r.$$
But since ${\bf u}_\Delta=(n,\ldots, n)$, it follows that if ${\bf j}$ is given, then ${\bf a}$ is uniquely determined and satisfies $2{\bf a}={\bf 0}$ as all of its entries are divisible by $n$ (modulo $N$). Hence, any choice of ${\bf \eta}$ uniquely determines ${\bf \varepsilon}$ and hence a solution ${\bf a}, {\bf j}$ to Equation \ref{eq-q-conditions-mat} and all solutions are of this form.

Explicitly, these ribbon structures are determined by a tuple $\eta\in (\mZ/2\mZ)^r$, from which $\varepsilon$ is computed by
\begin{align}
   {\varepsilon_i}=\begin{cases} 1+n\eta_i+\eta_{i+1} \, \mod 2, & \text{for $i=1,\ldots, r-1$},\\ 
     1+n\eta_i \, \mod 2, & \text{for $i=r$},
    \end{cases}
\end{align}

As verified in the proof of \Cref{theorem:spherical}, the only ribbon structure induced from a spherical structure for $H=\BB_\qs\rtimes \Bbbk [G]$ corresponds to ${\bf a}=(n,\ldots, n)$. This is obtained from ${\bf j}=(0,\ldots,0)$.
\end{proof}

\Cref{theorem:ribbon} implies that the categories  $\cZ(\lmod{\BB_\qs}(\cA_\qs))$ are modular \emph{if and only if} $\mJ=\mI$. The same conclusion holds for the relative centers as discussed in \Cref{thm:modularunique} below, where we will show that the relative center has a unique ribbon structure.

\section{Quantum groups of type Super A}
\label[section]{sec:superquantum}

In this section, we construct quasitriangular Hopf algebras $\ru_q(\fr{sl}_{r,\mJ})$, which we regard as quantum groups associated to Cartan data of type Super A. We note that these are Hopf algebras over $\Bbbk$ rather that Hopf superalgebras. 

\subsection{Definition and presentations}

For the rest of this section, we fix $q \in \Bbbk$ an $N$-th root of unity, for $N=2n>2$ an even order. We assume the setup from \eqref{eq:setup} and denote the associated Nichols algebras of type $\superqa{r}{q}{\mJ}$ by $\BB_\qs$.
Similarly to \cite{LW2}*{Section~5}, a presentation of the braided Drinfeld double of $\BB_\qs$ can be given. These braided Drinfeld doubles are analogues of quantum groups of super-type $\superqa{r}{q}{\mJ}$.

As in \Cref{subsec:NicholstypeA}, the Nichols algebra $\BB_\qs$ is a quotient of the tensor algebra in generators $e_1, \dots, e_r$ by a $\mZ^r$-homogeneuos ideal, where the  $\mZ^r$-grading of the generator $e_i$ is the $i$-th element of the canonical basis $(\alpha_i)_{1\leq i\leq r}$ of $\mZ^r$. 

The following definition is based on the braided Drinfeld double $\Drin_{\Bbbk [G]}(\BB_\qs,\BB_\qs^*)$ form \Cref{sec:general-Drin}, using the isomorphism of Hopf algebras \eqref{eq:kappai}. We use the braided commutators defined in \eqref{eq:q-comm}.

\begin{definition}[$\ru_q(\fr{sl}_{r,\mJ})$]\label[definition]{def:uqsl-pres}
 We define $\ru_q(\fr{sl}_{r,\mJ})$ to be the Hopf algebra generated as a $\Bbbk$-algebra by $x_i$, $y_i$, and $\kappa_i$, for $i=1,\ldots, r$, subject to the relations
\begin{gather}
   \kappa_i x_i= qx_i\kappa_i, \qquad \kappa_i x_j= x_j\kappa_i \quad (i\neq j),\qquad    \kappa_i y_i= q^{-1}y_i\kappa_i, \qquad \kappa_i y_j= y_j\kappa_i \quad (i\neq j), \label{eq:uqsl1}\\
        \kappa_i\kappa_j=\kappa_j\kappa_i, \qquad \kappa_i^{N}=1, \qquad\label{eq:uqsl2}\\
        y_ix_j-q^{u_{ji}}x_jy_i=\delta_{ij}(1-\ov{\gamma}_i\gamma_i),\label{eq:uqsl3}\\
            x_{ij}=0 \quad (i < j-1), \qquad x_{i i i\pm 1}=0  \quad (i \notin \mJ), \qquad x_i^2=0 \quad (i \in \mJ), \label{uq-rel-Nichols1}\\
    [x_{(i-1 i+1)},x_{i}]_\qs =0 \quad (i \in \mJ), \qquad x_{(ij)}^N=0 \quad (\alpha_{ij} \text{ even root}), \label{uq-rel-Nichols2}\\
        y_{ij}=0 \quad (i < j-1), \qquad y_{i i i\pm 1}=0  \quad (i \notin \mJ), \qquad y_i^2=0 \quad (i \in \mJ), \label{uq-rel-Nichols3}\\
        [y_{(i-1 i+1)},y_{i}]_\qs =0 \quad (i \in \mJ), \qquad y_{(ij)}^N=0 \quad (\alpha_{ij} \text{ even root}).\label{uq-rel-Nichols4}
\end{gather}
The coproduct, counit, and antipode of $\ru_q(\fr{sl}_{r,\mJ})$ are given on generators by
\begin{gather}
    \Delta(\kappa_i)=\kappa_i\otimes \kappa_i, \qquad \Delta(x_i)=x_i\otimes 1+\gamma_i\otimes x_i, \qquad \Delta(y_i)=y_i\otimes 1+\ov{\gamma}_i\otimes y_i,\label{drinrel3special}\\
    \varepsilon(\delta_{\bf i})=\delta_{{\bf i},0}, \qquad \varepsilon(x_i)=\varepsilon(y_i)=0,\label{drinrel4special}\\
S(\kappa_i)= \kappa_i^{-1},\qquad    S(x_i)=-\gamma_i^{-1}x_i, \qquad  S(y_i)=-\ov{\gamma}_i^{-1}y_i, \label{drinrel5special}
\end{gather} 
where we use the matrix $(u_{ij})=(u_{ij}^{\qs})$ from \eqref{eq:supera-matrix-power2} to define
\begin{align}\label{eq:gammas}
    \gamma_i=\kappa_{i}^{u_{i,i}}\kappa_{i+1}^{u_{i,i+1}}, \qquad \ov{\gamma}_i=\kappa_{i-1}^{u_{i-1,i}}\kappa_{i}^{u_{i,i}}.
\end{align}
\end{definition}
In particular, we have that
\begin{align}
        \gamma_i x_j&=q^{u_{ij}}x_j\gamma_i, & \gamma_i y_j &= q^{-u_{ij}}y_j \gamma_i,&
    \ov{\gamma}_i x_j&=q^{u_{ji}}x_j\ov{\gamma}_i, & \ov{\gamma}_i y_j &= q^{-u_{ji}}y_j \ov{\gamma_i}.
\end{align}
The relations  $x_{i i i\pm 1}=0$ in \eqref{uq-rel-Nichols1} also hold for $i \in \mJ$; this follows directly from \eqref{eq:general-qserre}. One can also show that $x_{(ij)}^2=0$ if $\alpha_{ij}$ is an odd root.

The following example describes the case $\ru_q(\fr{sl}_{r,\mI})$, where $\mI=\mJ$, which is of most interest in this paper due to the existence of ribbon structures, see \Cref{theorem:ribbon}.
\begin{example}
If $\mI=\mJ$ and $r$ is even, relation \eqref{eq:uqsl3} specializes to 
\begin{align*}
    y_ix_i+x_iy_i&=1-\ov{\gamma}_i\gamma_i, 
    &y_{i+1}x_{i}&=q^{(-1)^{i+1}} x_{i}y_{i+1}, &y_ix_{j}&=x_{j}y_i \quad (j\neq i, i-1). 
\end{align*}
Since $\alpha_{ij}$ is an even root if and only if $j-i$ is odd, the Nichols relations \eqref{uq-rel-Nichols1}--\eqref{uq-rel-Nichols4} specify to
\begin{gather*}
    x_ix_j=x_jx_i \quad (|i -j|>1),\qquad  x_i^2=0 \quad (1\leq i\leq r), \qquad x_{(ij)}^N=0 \quad (\text{$j-i$ odd}),\\
   x_{i-1} x_i x_{i+1} x_i -(1+q^{(-1)^{i}}) x_i x_{i-1} x_{i+1} x_i+ q^{(-1)^i} x_{i} x_{i-1} x_{i}x_{i+1}+ x_{i+1} x_{i} x_{i-1}x_{i} +q^{(-1)^{i}} x_{i} x_{i+1} x_{i}x_{i-1}=0,\\
    y_iy_j=y_jy_i \quad (|i -j|>1),  \qquad y_i^2=0 \quad (1\leq i\leq r),\qquad y_{(ij)}^N=0 \quad (\text{$j-i$ odd}),\\
    y_{i-1}y_iy_{i+1}-q^{(-1)^i}y_iy_{i+1}y_{i-1}-q^{(-1)^{i+1}}y_{i-1}y_{i+1}y_{i}+y_{i+1}y_iy_{i-1}=0 \quad  (2\leq i \leq r-1),\\
    y_{i-1} y_i y_{i+1} y_i -(1+q^{(-1)^{i}}) y_i y_{i-1} y_{i+1} y_i+ q^{(-1)^i} y_{i} y_{i-1} y_{i}y_{i+1}+ y_{i+1} y_{i}y_{i-1}y_{i} +q^{(-1)^{i}} y_{i} y_{i+1} y_{i}y_{i-1}=0.
\end{gather*}
In particular, we have $0=x_{(ij)}^N=(x_i x_{(i+1,j)})^N+(x_{(i+1,j)}x_i)^N$ for any $i<j$, and one can show that $x_{(ij)}^2=0$ if $j-i$ is even.
\end{example}

\begin{example}[$\ru_q(\fr{sl}_{2,\mI})$]
The smallest case is given when $r=2$. Then, the Nichols relations are  given by 
\begin{gather*}
        x_1^2=x_2^2=0, \qquad (x_1x_2)^N=-(x_2x_1)^N, \qquad y_1^2=y_2^2=0 \qquad (y_1y_2)^N=-(y_2y_1)^N,
\end{gather*}
Note that in this case $\ov{\gamma}_1\gamma_1=\kappa_2$ and $\ov{\gamma}_2\gamma_2=\kappa_1$, thus relation \eqref{eq:uqsl3} specializes to 
\begin{gather*}
    y_1x_1+x_1y_1=1-\kappa_2, \quad y_2x_2+x_2y_2=1-\kappa_1, \quad y_2x_1=qx_{1}y_2, \quad y_1x_{2}=x_{2}y_1.
\end{gather*}
The coproduct is determined by
\begin{align*}
    \Delta(x_1)=x_1\otimes 1+ \kappa_1^n \kappa_2\otimes x_1, \qquad \Delta(x_2)=x_2\otimes 1+ \kappa_2^n \otimes x_2, \\
    \Delta(y_1)=y_1\otimes 1+ \kappa_1^n\otimes y_1, \qquad \Delta(y_2)=y_2\otimes 1+ \kappa_1\kappa_2^n \otimes y_2.
\end{align*}
\end{example}

\begin{corollary}[PBW theorem]\label[corollary]{cor:pbw}
The Hopf algebra $\ru_q (\mathfrak{sl}_{r,\mJ})$ has a triangular decomposition as a $\Bbbk$-vector space as $\BB_\qs\otimes \Bbbk G\otimes \BB_\qs^*$. It has a PBW basis given by the set
\begin{align*}
\Set{\prod_{i\le j} x_{(ij)}^{n_{ij}} \prod_{l=1}^r \kappa_l^{a_l} \prod_{s\le t} y_{(ij)}^{m_{ij}} \middle| \begin{array}{cc}0\le n_{ij}, m_{ij}< N &\text{ if  $\alpha_{ij}$ is even},\\ 0\le n_{ij}, m_{ij}< 2 &\text{ if $\alpha_{ij}$  is odd},
\end{array}
 0\le a_l<N }.
\end{align*}
Here, we use the convention that $x_{(ii)}=x_i$, $y_{(ii)}=y_i$.
\end{corollary}
\begin{proof}
The triangular decomposition of $\ru_q(\mathfrak{sl}_{r,\mJ})$ follows from its definition as a braided Drinfeld double, see \Cref{sec:general-Drin} and the isomorphism  $\Bbbk G\cong \Bbbk [G]$ in Equation \eqref{eq:kappai}. The PBW bases for $\BB_\qs$ and its dual are obtained from \eqref{eq:supera-PBW} and assemble into the stated basis for $\ru_q (\mathfrak{sl}_{r,\mJ})$.
\end{proof}
Note that the Hopf algebra $\ru_q (\mathfrak{sl}_{r,\mJ})$ is $\mZ$-graded, where
\begin{align}\label{Z-grading}
    \deg x_i = 1, \qquad \deg \kappa_i=0, \qquad \deg y_i=-1,
\end{align}
for all $i=1,\ldots, r$.  

\begin{remark}
If $r$ is even and $\mJ=\mI$ we have that the category $\cA_\qs$ is non-degenerate. This is equivalent to the subgroup 
$\langle \gamma_1\ov{\gamma}_1, \ldots, \gamma_r\ov{\gamma}_r \rangle$
being equal to $G$. An alternative presentation of $\ru_q(\mathfrak{sl}_{r,\mI})$ can be obtained using the grouplike elements
\begin{align} \label{eq:ki-def}
k_i=\gamma_i\ov{\gamma}_i=\begin{cases}
\kappa_2, &\text{if $i=1$},\\
\kappa_{i-1}\kappa_{i+1}^{-1}, &\text{if $1<i<r$ and $i$ is even},\\
\kappa_{i-1}^{-1}\kappa_{i+1}, &\text{if $1<i<r$ and $i$ is odd},\\
\kappa_{r-1}, &\text{if $i=r$}.
\end{cases}.
\end{align}
It follows that 
\begin{align}
    \kappa_{2i}=\prod_{j=1}^{i}k_{2j-1}, \qquad     \kappa_{r-2i-1}=\prod_{j=0}^{i}k_{r-2j},
\end{align}
and hence for $j=1,\ldots, r/2$
\begin{gather}\label{eq:gammaind}
    \gamma_i=\begin{cases}
    k_r^n k_{r-2}^n \ldots k_{2j}^n k_1k_3\ldots, k_{2j-1}, & \text{if $i=2j-1$},\\
    k_1^nk_3^n\ldots k_{2j-1}^nk_{2j+2}^{-1}k_{2j+4}^{-1}\ldots k_r^{-1}, & \text{if $i=2j$},
    \end{cases}\\
    \ov{\gamma}_i=\begin{cases}
    k_r k_{r-2}\ldots k_{2j}k_1^nk_3^n\ldots k_{2j-1}^n, &\text{if $i=2j$},\\
    k_1^{-1}k_3^{-1}\ldots k_{2j-3}^{-1}k_r^nk_{r-2}^n\ldots k_{2j}^n, &\text{if $i=2j-1$}.
    \end{cases}\label{eq:ovgammaind}
\end{gather}
\end{remark}

With this alternative set of grouplike elements, we derive the following alternative presentation.

\begin{proposition}\label[proposition]{prop:alternativepres}
The Hopf algebra $\ru_q(\mathfrak{sl}_{r,\mI})$ is isomorphic to the Hopf algebra generated by $x_i, y_i,k_i$ subject to the algebra relations
\begin{gather}\label{eq:gammaind2}
    k_i x_i=x_ik_i, \qquad 
    k_i x_j=x_jk_i \quad(|i-j|>1),
\\ k_ix_{i+1}=q^{(-1)^{i+1}}x_{i+1}k_i
    \qquad   k_ix_{i-1}=q^{(-1)^i}x_{i-1}k_i
    \\
       k_i y_i=y_ik_i, \qquad
    k_i y_j=y_jk_i\quad (|i-j|>1),
\\ 
k_iy_{i+1}=q^{(-1)^{i}}y_{i+1}k_i
    \qquad   k_iy_{i-1}=q^{(-1)^{i+1}}y_{i+1}k_i
\\ \kappa_i\kappa_j=\kappa_j\kappa_i, \qquad \kappa_i^{N}=1,\\
    y_ix_i+x_iy_i=1-k_i,\qquad y_{i+1}x_{i}=q^{(-1)^{i+1}}x_{i}y_{i+1}, \qquad y_ix_j=x_jy_i \quad (i\neq j,j+1).\label{rel:comm}
\end{gather}
The coproduct, counit, and antipode are determined on the generators by
\begin{gather}
   \Delta(k_i)=k_i\otimes k_i,\qquad  \Delta(x_i)=x_i\otimes  1+\gamma_i \otimes x_i, \qquad \Delta(y_i)=y_i\otimes 1+\ov{\gamma}_i\otimes y_i,\\
   \varepsilon(k_i)=1, \qquad \varepsilon(x_i)=0, \qquad \varepsilon(y_i)=0,\\
   S(k_i)=k_i^{-1}, \qquad S(x_i)=- \gamma_i^{-1}x_i, \qquad S(y_i)=-\ov{\gamma}_i^{-1}y_i,
\end{gather}
where $\gamma_i$ and $\ov{\gamma}_i$ are defined in \eqref{eq:gammaind} and \eqref{eq:ovgammaind}.
\end{proposition}

\subsection{Non-semisimple modular categories from quantum groups of type Super A}
\label[section]{sec:modularresults}

Recall the classification of ribbon structures for $\cZ(\lmod{\BB_\qs\rtimes \Bbbk G})$ from \Cref{theorem:ribbon}. There, we found $2^r$ distinct ribbon structures parametrized by pairs of elements
\begin{align}\label{eq:ribbonpair-double}
    a=\kappa_1^{n\varepsilon_1}\ldots \kappa_r^{n\varepsilon_r}\in \Bbbk[G], \qquad \zeta=g_1^{-n\eta_1}\ldots g_r^{-n\eta_r}\in \Bbbk G.
\end{align}
Here, $(\eta_1,\ldots, \eta_r)\in\Set{\pm 1}^{\times r}$ determines $(\varepsilon_1,\ldots, \varepsilon_r)$ by the equations 
\begin{align}\label{eq:eta-varep-comp}
   {\varepsilon_i}=\begin{cases} 1+n\eta_i+\eta_{i+1} \, \mod 2, & \text{for $i=1,\ldots, r-1$},\\ 
     1+n\eta_i \, \mod 2, & \text{for $i=r$},
    \end{cases}
\end{align}
This follows from \Cref{eq-q-conditions-mat} using that $2n=0\mod N$ whence we can ignore the signs.

\begin{proposition}\label[proposition]{prop:ribbonelement-uqslI}
All ribbon structures from \Cref{theorem:ribbon} induce the same ribbon structure on $\ru_q(\fr{sl}_{r,\mI})$. The ribbon element element making $\ru_q(\fr{sl}_{r,\mI})$ a ribbon Hopf algebra is given by the element $u \nu$, where
\begin{align}\label{eq:ribbon-elt}
    \nu = \kappa_1^n\ldots \kappa_r^n,
\end{align}
and $u$ is the Drinfeld element, see \Cref{sec:ribbonback}.
\end{proposition}

\begin{proof}
    By \Cref{thm:KR}, the ribbon element of $\Drin(H)$ is given by $u(\zeta^{-1}\otimes a^{-1})$ for $a$ and $\zeta$ as in \Cref{eq:ribbonpair-double}. To compute the ribbon element of the braided Drinfeld double $\Drin_{\Bbbk[G]}(\BB_\qs,\BB_\qs^*)=\ru_q(\fr{sl}_{r,\mI})$ we use the surjective Hopf algebra homomorphism from \Cref{prop:Drin-quotient}. Under this homomorphism, $g_{\bf i}\mapsto \ov{\gamma}_{\bf i}$. Thus, $\zeta$ acts by
    $$\ov{\gamma}_1^{n\eta_1}\ldots \ov{\gamma}_r^{n\eta_r}=\kappa_1^{n^2\eta_1+n\eta_2}\ldots \kappa_{r-1}^{n^2\eta_{r-1}+n\eta_r}\kappa_r^{n^2\eta_r}$$
    on each module over $\ru_q(\fr{sl}_{r,\mI})$.
    Hence, $\zeta \otimes a$ acts by
    $$\kappa_1^{n\varepsilon_1+n^2\eta_1+n\eta_2}\ldots \kappa_{r-1}^{n\varepsilon_{r-1}+n^2\eta_{r-1}+n\eta_r}\kappa_r^{n\varepsilon_r+n^2\eta_r}=\kappa_1^n\ldots \kappa_r^n=\nu.$$
    by \Cref{eq:eta-varep-comp}.
    Thus, as recalled in \Cref{sec:ribbonback}, the ribbon twist is given by the action of the inverse of the ribbon element $u\nu^{-1}=u\nu$. 
\end{proof}

The results of this paper amount to the following cumulative statement. 

\begin{theorem}\label{thm:modularunique}
Let $q$ be a primitive root of unity of order $N=2n$. The following statements are equivalent for the braided tensor category $\cC=\lmod{\ru_q(\fr{sl}_{r,\mJ})}$:
\begin{enumerate}
\item[(i)] $\cC$ is a ribbon category.
\item[(ii)] $\cC$ is a  modular category.
\item[(iii)] $r$ is even and $\mJ=\mI$.
\end{enumerate}
In case the equivalent statements hold, the ribbon category structure on $\cC$ is uniquely determined by the element $\nu$ from \Cref{eq:ribbon-elt} via 
$$\theta_W\colon  W\to W, \quad w\mapsto \nu u^{-1}\cdot w.$$
\end{theorem}
\begin{proof}
We know that (iii) implies (ii) by \Cref{theorem:ribbon} and \Cref{thm:LW-center-relcenter} as in this case $\cA_\qs$ is non-degenerate by \Cref{cor:Aq-nondeg}. Further, (ii) implies (i) by definition. 

By \cite{LW2}*{Example~4.18} we have an equivalence of braided monoidal categories
\begin{equation}\lmod{\Drin(\BB_\qs\rtimes \Bbbk G)}\simeq \cA_\qs^\rev \boxtimes \cC.\label{factorization-Del}\end{equation}
Since $N$ is even, the category $\cA_\qs^\rev$ has $2^r$ ribbon structures, corresponding to the quadratic forms on $\mZ_N^{r}$, which are determined by choosing $\vartheta(g_i)\in \Set{\pm 1}$, see \Cref{prop:ribbon-Aq} and \Cref{rem:group-ribbon}. Hence, if $\cC$ is a ribbon category, then $\cA_\qs^\rev \boxtimes \cC$ is a ribbon category by \Cref{lem:ribHopftensor}. Then, by \Cref{theorem:ribbon}, (iii) follows, which, in turn, implies (ii). This shows the equivalence of (i), (ii), and (iii). 

It remains to prove uniqueness of the ribbon structure in case that $\mJ=\mI$.
For this, we again use the decomposition from \Cref{factorization-Del}, where 
all categories are finite $\Bbbk$-tensor categories so a ribbon structure on $\lmod{\Drin(\BB_\qs\rtimes \Bbbk G)}$ is uniquely determined by a pair of ribbon structures on $\cA_\qs^\rev$ and $\cC$ by \Cref{lem:ribHopftensor}.
There exists a ribbon structure on $\cC$, see \Cref{prop:ribbonelement-uqslI}. Then, as there are $2^r$ ribbon structures on $\cA_\qs^{\rev}$, for each ribbon structure on $\cC$, we obtain $2^r$ ribbon structures on $\lmod{\Drin(\BB_\qs\rtimes \Bbbk G)}$. Note that these ribbon structures are all distinct by \Cref{lem:ribHopftensor}. However, according to \Cref{theorem:ribbon}, there are only $2^r$ distinct ribbon structures on $\lmod{\Drin(\BB_\qs\rtimes \Bbbk G)}$, hence, there can be at most one ribbon structure on $\cC$. This completes the proof.
\end{proof}

Note that by \eqref{eq:Z}, since it is non-empty, the set of ribbon elements for $H=\ru_q(\fr{sl}_{r,\mI})$ is in bijection with the set 
$Z=\Set{z\in Z(H)\cap G(H) \;|\;z^2=1}$. Using the coradical filtration one sees that $G(H)=\Bbbk G=\langle \kappa_1,\ldots, \kappa_r\rangle$ and since $\kappa_i x_i=qx_i\kappa_i$, $\kappa_ix_j=x_j\kappa_i$, if $i\neq j$, by Equation \eqref{eq:uqsl1} and $q$ is a primitive $N$-th root of unity, it follows that $Z=\Set{1}$. This gives an alternative argument that the ribbon structure on $\lmod{\ru_q(\fr{sl}_{r,\mI})}$ is unique.

The formula for $\nu$ from \Cref{prop:ribbonelement-uqslI} makes it easy to compute quantum dimensions of modules over $\ru_q(\fr{sl}_{r,\mI})$. Given such a module $V$, the $\mZ_N^r$-graded structure induces a $\mZ_2$-grading $V=V_{\un{0}}\oplus V_{\un{1}}$, where $V_{\un{0}}$ is spanned as a vector space by elements that are $\mZ^r_N$-homogeneous of degree $(i_1,\ldots, i_r)$ such that $\sum_{j=1}^r i_j$ is even, and $V_{\un{1}}$ is the complement consisting of elements with odd sum of $\mZ^r_N$-degrees. 

\begin{corollary}\label[corollary]{cor:dimq}
The quantum dimension of a $\ru_q(\fr{sl}_{r,\mI})$-module $V$ is given by 
\begin{align}
    \dim_q V= \dim_\Bbbk V_{\un{0}} - \dim_\Bbbk V_{\un{1}}.
\end{align}
\end{corollary}
\begin{proof}
The quantum dimension is given by 
$$\dim_q(V)= \ev_{V^*}(j_V\otimes \ide_{V^*})\coev_{V},$$
where $j_V=\phi_V\theta_V$ is the pivotal structure associated to the ribbon twist $\theta_V$, see \Cref{eq:ribbonpivotal}. As explained in \Cref{sec:ribbonback} and \Cref{prop:ribbonelement-uqslI}, 
$\theta_V(v)= \nu u^{-1}\cdot v,$
and one checks that the Drinfeld isomorphism $\phi_V$ is given by the action of the Drinfeld element $u=S(R^{(2)})R^{(1)}$ followed by the canonical pivotal structure $\tau_V\colon V\to V^{**}$ of $\Vect$. Thus, if $\Set{v_i}_{i}$ is a $\mZ^r$-homogeneous basis for $V$ and  with dual basis $\Set{f_i}_i$, we get
\begin{align*}
\dim_q(V)&= \ev_{V^*}(j_V\otimes \ide_{V^*})\coev_{V}=\ev_{V^*}(\tau_V(u\nu u^{-1}\cdot (-))\otimes \ide_{V^*})\coev_{V}\\&=\sum_{i}\langle f_i, u\nu u^{-1}\cdot v_i\rangle =\sum_{i}\langle S^{-1}(u)\cdot f_i, \nu u^{-1}\cdot v_i\rangle \\&= \sum_{i}\langle f_i,\nu u^{-1}u\cdot v_i\rangle 
=\sum_{i}\langle f_i, \nu \cdot v_i\rangle.
\end{align*}
Here, we use the identities 
$$\langle f,h\cdot v\rangle = \langle S^{-1}(h)\cdot f,v\rangle, \quad \sum_{i}v_i\otimes h\cdot f_i= \sum_{i}S(h)\cdot v_i\otimes f_i,$$
for all $h\in H, v\in V, f\in V^*$. 
We note that $\nu=\kappa_1^n\ldots \kappa_r^n$ acts by the eigenvalue $1$ on $V_{\un{0}}$ and by $-1$ on $V_{\un{1}}$. This implies the claim.
\end{proof}

\subsection{A basis of divided powers and the universal R-matrix in the rank-two case}
\label[section]{sec:divpowers}

In this subsection, we provide a favourable basis for $\ru_q(\mathfrak{sl}_{2,\mI})$ that will ease computations, especially in view of studying the representation theory of these Hopf algebras in Section \ref{sec:reps}, e.g., by providing a formula for the universal $R$-matrix. 
We define divided powers for the negative part of $\ru_q(\mathfrak{sl}_{2,\mI})$ by setting
\begin{equation}
    y_{12}^{(k)}:=\frac{y_{12}^{k}}{(1-q)^k[k]_q!}, \qquad  
    0\leq k\leq N-1,
\end{equation}
recalling that $y_{12}=y_1y_2-qy_2y_1$.
Then it follows from \Cref{cor:pbw} (PBW theorem) that the set
\[
\Set{x_1^{a_1}x_{21}^{a_{21}}x_2^{a_2}\kappa_1^{n_1}\kappa_2^{n_2}y_2^{b_2}y_{12}^{(b_{12})}y_1^{b_1}~\middle|~0\leq a_1,a_2,b_1,b_2\leq 1, ~0\leq a_{21},n_1,n_2,b_{12}\leq N-1 } 
\]
forms a basis for  $\ru_q(\mathfrak{sl}_{2,\mI})$. 

We introduce the modified Sweedler's notation 
$$\un{\Delta}(h)=h_{\un{(1)}}\otimes h_{\un{(2)}}, \qquad \text{for }h\in \BB_\qs \text{ or }h\in \BB_\qs^*,$$
where the summation is implicit, to distinguish the braided coproduct from the coproduct of $h$ as an element of the $\Bbbk$-Hopf algebra $\ru_q(\fr{sl}_{2,\mI})$. On this basis, the braided coproduct $\un{\Delta}$ of the Nichols algebras $\BB_\qs^*$ is given by the following formulas with integer coefficients. 

\begin{lemma}\label[lemma]{lem:rank2-negative-coprod}
For any $k=0,\ldots, N-1$, we have
\begin{align*}
\un{\Delta}(y_{12}^{(k)})=&\sum_{i=0}^k y_{12}^{(i)}\otimes y_{12}^{(k-i)}+\sum_{i=0}^{k-1} y_{12}^{(i)}y_1\otimes y_2y_{12}^{(k-i-1)},\\
\un{\Delta}(y_{12}^{(k)}y_1)=&\sum_{i=0}^k\left(y_{12}^{(i)}\otimes y_{12}^{(k-i)}y_1+(-1)^{k-i}y_{12}^{(i)}y_1\otimes y_{12}^{(k-i)}\right)+
\sum_{i=0}^{k-1}y_{12}^{(i)}y_1\otimes y_2y_{12}^{(k-i-1)}y_1,\\
\un{\Delta}(y_2y_{12}^{(k)})=&\sum_{i=0}^k\left( y_2y_{12}^{(i)}\otimes y_{12}^{(k-i)} + (-1)^i y_{12}^{(i)}\otimes y_2y_{12}^{(k-i)} \right)+
\sum_{i=0}^{k-1}y_2y_{12}^{(i)}y_1\otimes y_2y_{12}^{(k-i-1)},\\
\un{\Delta}(y_2y_{12}^{(k)}y_1)
=&\sum_{i=0}^k\Big(
(-1)^{(k-i)}y_2y_{12}^{(i)}y_1\otimes y_{12}^{(k-i)} + y_2y_{12}^{(i)}\otimes y_{12}^{(k-i)}y_1+(-1)^k y_{12}^{(i)}y_1\otimes y_2y_{12}^{(k-i)}\\
& +  (-1)^i y_{12}^{(i)}\otimes y_2y_{12}^{(k-i)}y_1 \Big)+
\sum_{i=0}^{k-1}y_2y_{12}^{(i)}y_1\otimes y_2y_{12}^{(k-i-1)}y_1.
\end{align*}
\end{lemma}

\begin{proof}
We verify the first formula inductively on $k$. The case $k=0$ is clear. Next we use that $y_{12}^{(1)}y_{12}^{(j)}=[j+1]_q y_{12}^{(j+1)}$ for all $j$, and since the coproduct $\un{\Delta}$ is a map of braided algebras we get 
\begin{align*}
\un{\Delta}\left(y_{12}^{(1)}y_{12}^{(k)}\right)=&\left(y_{12}^{(1)}\otimes1 + 1\otimes y_{12}^{(1)} +y_1\otimes y_2\right) \left(\sum_{i=0}^k y_{12}^{(i)}\otimes y_{12}^{(k-i)}+\sum_{i=0}^{k-1} y_{12}^{(i)}y_1\otimes y_2y_{12}^{(k-i-1)}\right)\\
=& \sum_{i=0}^k\left( y_{12}^{(1)}y_{12}^{(i)}\otimes y_{12}^{(k-i)} +q^i y_{12}^{(i)}\otimes y_{12}^{(1)}y_{12}^{(k-i)} + (-1)^i y_1 y_{12}^{(i)}\otimes y_2y_{12}^{(k-i)} \right) \\
&+ \sum_{i=0}^{k-1}\left( y_{12}^{(1)}y_{12}^{(i)}y_1\otimes y_2 y_{12}^{(k-i-1)} -q^i y_{12}^{(i)}y_1\otimes y_{12}^{(1)}y_2y_{12}^{(k-i-1)}\right)\\
=& \sum_{i=0}^k\left( [i+1]_q y_{12}^{(i+1)}\otimes y_{12}^{(k-i)} +q^i [k+1-i]_q y_{12}^{(i)}\otimes y_{12}^{(k+1-i)} \right) + \sum_{i=0}^k q^i y_{12}^{(i)}y_1\otimes y_2y_{12}^{(k-i)}\\
&+ \sum_{i=0}^{k-1}\left( [i+1]_qy_{12}^{(i+1)}y_1\otimes y_2 y_{12}^{(k-i-1)} +q^{i+1} [k-i]_q y_{12}^{(i)}y_1\otimes y_2y_{12}^{(k-i)}\right)\\
=& [k+1]_q \left(y_{12}^{(k+1)}\otimes 1 + 1\otimes y_{12}^{(k+1)}\right)+\sum_{i=1}^k\left([i]_q+q^i [k+1-i]_q\right) y_{12}^{(i)}\otimes y_{12}^{(k+1-i)}\\
&+ [k+1]_q\left( y_1\otimes y_2y_{12}^{(k)}+ y_{12}^{(k)}y_1\otimes y_2\right)
+\sum_{i=1}^{k-1}\left( q^i[i]_q+q^{i+1}[k-i]_q\right)y_{12}^{(i)}y_1\otimes y_2 y_{12}^{(k-i)}\\
=& [k+1]_q \left( \sum_{i=0}^{k+1} y_{12}^{(i)}\otimes y_{12}^{(k+1-i)}+\sum_{i=0}^{k} y_{12}^{(i)}y_1\otimes y_2y_{12}^{(k-i)}\right),
\end{align*}
where the third equality uses that $y_1y_{12}^{(j)}=(-q)^jy_{12}^{(j)}y_1$ and $y_{12}^{(j)}y_2=(-q)^jy_2y_{12}^{(j)}$ for any $j$.

The other formulas now follow directly using that $\un{\Delta}$ is a map of braided algebras.
\end{proof}

Next we show that the basis $\Set{y_2^{b_2}y_{12}^{(b_{12})}y_1^{b_1}}$ of $\BB_\qs^*$ is the dual basis to the PBW basis $\Set{x_1^{a_1}x_{21}^{a_{21}}x_2^{a_2}}$ with respect to the pairing $\langle\;,\, \rangle \colon \BB_\qs^* \otimes \BB_\qs \to\Bbbk$ extending the evaluation, see \Cref{sec:general-Drin}.

\begin{lemma}\label[lemma]{lem:dual-orthogonal}
We have $\langle y_2^{b_2}y_{12}^{(b_{12})}y_1^{b_1}, x_1^{a_1}x_{21}^{a_{21}}x_2^{a_2}\rangle = \delta_{a_1, b_1} \delta_{a_{21}, b_{12}} \delta_{a_2, b_2}$.
\end{lemma}

\begin{proof}
Since  $\langle y, x \rangle=0$ whenever $x$ and $-y$ have different $\mZ^2$-degrees, the statement holds trivially for any pair of basis elements such that $b_1 + b_{12} \ne a_1+a_{21}$ or $b_2 + b_{12}  \ne a_2+a_{21}$. We proceed by induction on $k\coloneq b_1+b_{12}+b_2$. The case $k=1$ contains only two non-trivial cases, both of them coming from $b_{12}=1$. Since $\un{\Delta}\left(y_{12}^{(1)}\right)=y_{12}^{(1)}\otimes 1 + 1 \otimes y_{12}^{(1)} + y_1 \otimes y_2$, we have
\begin{align*}
\langle y_{12}^{(1)}, x_{1}x_{2} \rangle = 0, &&
\langle y_{12}^{(1)}, x_{21} \rangle = \langle y_2, x_2\rangle \langle y_1, x_1\rangle = 1.
\end{align*}
Assume $b_1+b_{12}+b_2>0$. In case $a_1=1$, using \Cref{lem:rank2-negative-coprod} we get
\begin{align*}
\langle y_2^{b_2}y_{12}^{(b_{12})}y_1^{b_1}, x_1x_{21}^{a_{21}}x_2^{a_2}\rangle &= 
\langle (y_2^{b_2}y_{12}^{(b_{12})}y_1^{b_1})_{\un{(2)}}, x_1\rangle 
\langle (y_2^{b_2}y_{12}^{(b_{12})}y_1^{b_1})_{\un{(1)}}, x_{21}^{a_{21}}x_2^{a_2}\rangle 
\\
&=\delta_{1, b_1} \langle (y_2^{b_2}y_{12}^{(b_{12})}y_1^{b_1})_{\un{(1)}}, x_{21}^{a_{21}}x_2^{a_2}\rangle=
\delta_{1, b_1} \langle y_2^{b_2}y_{12}^{(b_{12})}, x_{21}^{a_{21}}x_2^{a_2}\rangle,
\end{align*}
and the claim follows by induction. Similarly, for the case $a_1=0$, using \Cref{lem:rank2-negative-coprod} we see that
\begin{align*}
\langle y_2^{b_2}y_{12}^{(b_{12})}y_1^{b_1}, x_{21}^{a_{21}}x_2^{a_2}\rangle &= 
\langle (y_2^{b_2}y_{12}^{(b_{12})}y_1^{b_1})_{\un{(2)}}, x_{21}^{a_{21}}\rangle 
\langle (y_2^{b_2}y_{12}^{(b_{12})}y_1^{b_1})_{\un{(1)}}, x_2^{a_2}\rangle 
\\
&=\langle (y_2^{b_2}y_{12}^{(b_{12})}y_1^{b_1})_{\un{(2)}}, x_{21}^{a_{21}}\rangle \delta_{0, b_1} \delta_{a_2, b_2}=
\langle (y_2^{b_2}y_{12}^{(b_{12})})_{\un{(2)}}, x_{21}^{a_{21}}\rangle \delta_{0, b_1} \delta_{a_2, b_2}
\\&= \left\langle \sum_{i=0}^{b_{12}}y_{12}^{(i)}, x_{21}^{a_{21}}\right\rangle \delta_{0, b_1} \delta_{a_2, b_2}, 
\end{align*}
and the claim follows by induction.
\end{proof} 

Using the above lemma, we can now compute the universal $R$-matrix in the rank-two case from \Cref{eq:r-matrixG} and \Cref{cor:R-matrix-Drin}.

\begin{corollary}
The universal $R$-matrix for $\ru_q(\mathfrak{sl}_{2,\mI})$ is given by 
$$
R=\sum_{(b_1,b_{12},b_2)\in \mZ_2\times \mZ_N\times \mZ_2} \sum_{(i_1,i_2)\in \mZ_N^2}  \kappa_1^{ni_1+i_2}\kappa_2^{ni_2} y_2^{b_2}y_{12}^{(b_{12})}y_1^{b_1} \otimes x_1^{b_1}x_{21}^{a_{21}}x_2^{b_2} \delta_{(i_1,i_2)}.
$$
\end{corollary}

Next we give formulas for the braided coproduct of the Nichols algebra $\BB_\qs$.

\begin{lemma}\label[lemma]{lem:rank2-positive-coprod}
For any $k=0,\ldots, N-1$, we have
\begin{align*}
\un{\Delta}(x_{21}^{k})=&\sum_{i=0}^k \binom{k}{i}_q x_{21}^{i}\otimes x_{21}^{k-i}+\sum_{i=0}^{k-1} \binom{k}{i}_q (1-q^{k-i}) x_{21}^{i}x_1\otimes x_2x_{21}^{k-i-1},
\\
\un{\Delta}(x_1x_{21}^{k})=&\sum_{i=0}^k \binom{k}{i}_q \left(x_1x_{21}^{i}\otimes x_{21}^{k-i} + (-q)^i x_{21}^{i}\otimes x_1x_{21}^{k-i} \right)\\
& +\sum_{i=0}^{k-1}\binom{k}{i}_q (-1)^i(1-q)^{k-i}q^{i+1}x_1x_{21}^{i}x_1\otimes x_1x_{21}^{k-i-1},
\\
\un{\Delta}(x_{21}^{k}x_2)=&\sum_{i=0}^k \binom{k}{i}_q \left( (-q)^{k-i}x_{21}^{i}x_2\otimes x_{21}^{k-i}+x_{21}^{i}\otimes x_{21}^{k-i}x_2\right)\\
&+
\sum_{i=0}^{k-1} \binom{k}{i}_q (1-q^{k-1}) x_{21}^{i}x_1\otimes x_2x_{21}^{k-i-1}x_1,
\\
\un{\Delta}(x_1x_{21}^{k}x_2)
=&\sum_{i=0}^k \binom{k}{i}_q \Big(
(-q)^{k-i}x_1x_{21}^{i}x_2\otimes x_{21}^{k-i} + (-1)^k q^{k+1} x_{21}^{i}x_2\otimes x_1x_{21}^{k-i}+ x_1x_{21}^{i}\otimes x_{21}^{k-i}x_2\\
& +  (-q)^i x_{21}^{i}\otimes x_1x_{21}^{k-i}x_2 \Big)+
\sum_{i=0}^{k-1} \binom{k}{i}_q (1-q^{k-i})x_1x_{21}^{i}x_2\otimes x_1x_{21}^{k-i-1}x_2.
\end{align*} \qedhere
\end{lemma}

The above formulas help to compute the braided commutator of the elements of the divided power basis by means of the cross relation
\begin{equation}
\begin{split}
  q( |x_{\un{(2)}}|\otimes|y_{\un{(1)}}|)^{-1}  x_{\un{(2)}}y_{\un{(1)}}\langle y_{\un{(2)}}, x_{\un{(1)}} \rangle&=\ov{\gamma}_{|x_{\un{(2)}}|} y_{\un{(2)}}x_{\un{(1)}}\gamma_{|x_{\un{(2)}}|}\langle y_{\un{(1)}}, x_{\un{(2)}} \rangle\\
  &=q(|y_{\un{(2)}}||x_{\un{(1)}}|,|x_{\un{(2)}}|)y_{\un{(2)}}x_{\un{(1)}}k_{|x_{\un{(2)}}|}\langle y_{\un{(1)}}, x_{\un{(2)}} \rangle.
\end{split}
\end{equation}
Here, we assume that the coproduct sum is given in terms of $G$-homogeneous elements and denote by $|z|$ the $G$-degree of a homogeneous element using the $G$-grading introduced in \Cref{prop:BBqs-graded}.
Thus, we obtain the following formulas commuting the elements $x_1$ and $x_2$ past divided powers of the $y_i$ which will be used in representation theoretic computations in the next section. 

\begin{lemma}\label[lemma]{lem:commrel-rank2}
For any $k=0, \ldots,N-1$, we have, with the notation $k_i:=\gamma_i\ov{\gamma_i}$, that
\begin{gather*}
y_{12}^{(k)}x_1-(-q)^k x_1y_{12}^{(k)}= (-1)^k y_2y_{12}^{(k-1)}k_1,\qquad 
y_{12}^{(k)}x_2-(-1)^kx_2y_{12}^{(k)}=y_{12}^{(k-1)}y_1,
\\
y_{12}^{(k)}y_1x_1+(-q)^kx_1y_{12}^{(k)}y_1=y_{12}^{(k)}(1- k_1)-(-1)^ky_2y_{12}^{(k-1)}y_1k_1, \qquad 
y_{12}^{(k)}y_1x_2=(-1)^kx_2y_{12}^{(k)}y_1,\\
y_2y_{12}^{(k)}x_1=(-1)^kq^{k+1} x_1y_2y_{12}^{(k)}, \qquad
y_2y_{12}^{(k)}x_2+(-1)^{k}x_2y_2y_{12}^{(k)}=(-1)^ky_{12}^{(k)}(1- q^{-k} k_2)+y_2y_{12}^{(k-1)}y_1,\\
y_2y_{12}^{(k)}y_1x_1+(-1)^kq^{k+1}x_1 y_2y_{12}^{(k)}y_1= y_2y_{12}^{(k)}(1-k_1), \\ 
y_2y_{12}^{(k)}y_1x_2+(-1)^{k}x_2 y_2y_{12}^{(k)}y_1=
(-1)^ky_{12}^{(k)}y_1(1-q^{-k-1}k_2),
\end{gather*}
where terms involving $(k-1)$ are omitted if $k=0$.
\end{lemma}

\section{Representation Theory}
\label[section]{sec:reps}

In this section, we study the representation theory of the Hopf algebras $\ru_q(\fr{sl}_{r,\mJ})$. We are able to employ the general framework of graded Hopf algebras with triangular decomposition provided in \cites{HN, BT, Vay}, which produces and classifies simple modules as simple heads of standard modules. In fact, the categories of graded finite-dimensional modules are highest weight categories. In the case $r=2$ and $\mJ=\mI$ we computed the simple modules explicitly in \Cref{sec:rank2}. At the end of the section, we discuss the semisimplification of this category of modules.

\subsection{Classification of simple modules}

In this section, we classify the  simple modules over $\ru_q(\fr{sl}_{r,\mJ})$, where $r \geq 2$ and $q$ is a root of unity of order $N=2n$. We consider $\ru_q(\fr{sl}_{r,\mJ})$ as a graded algebra with the $\mZ$-grading given in \Cref{Z-grading}.
We first observe that since $\ru_q(\fr{sl}_{r,\mJ})$ is finite-dimensional, both simple and projective modules admit a grading \cite{GG}*{Section~3}.

We start by establishing some terminology. Let $V$ be a $\ru_q(\fr{sl}_{r,\mJ})$-module with action $x\cdot v=xv$, for $x\in \ru_q(\fr{sl}_{r,\mJ})$ on $v\in V$. Then $V$ has a direct sum decomposition as a $\Bbbk K$-module, for the abelian group $K=\langle \kappa_1,\ldots,\kappa_r\rangle$ from \Cref{eq:kappai},
\begin{align}\label{eq:Lambda-decomp}
    V=\bigoplus_{{\bf i}\in \Lambda}V_{\bf i},\qquad V_{\bf i}=\Set{v\in V | \kappa_jv=q^{i_j}v, \forall j=1,\ldots, r}.
\end{align}
This decomposition gives a grading over $\Lambda=(\mZ/N\mZ)^{\times r}$.
If $v\in V_{\bf i}$, we say that $v$ is \emph{$\Lambda$-homogeneous} of $\Lambda$-degree ${\bf i}$. Note that the set $\Lambda$ also classifies all simple $\Bbbk K$-modules. An element ${\bf i}\in \Lambda$ corresponds to the one-dimensional simple $\Bbbk K$-module $\Bbbk_{\bf i}$ with the action of $\kappa_j$ given by multiplication with $q^{i_j}$.

\smallskip

A \emph{highest weight vector} is $v\in V\setminus \Set{0}$ such that $v$ is $\Lambda$-homogeneous and $x_i\cdot v=0$ for all $i=1,\ldots, r$. In this case, we say that the $\Lambda$-degree of $v$ is a \emph{highest weight} of $V$. If $V$ is generated by a highest weight vector of degree ${\bf i}\in \Lambda$, then we say that $V$ is a \emph{highest weight module} of weight ${\bf i}$.
It can be shown that every finite-dimensional module contains a highest weight vector.
\smallskip

Recall the triangular decomposition 
$$ \ru_q(\fr{sl}_{r,\mJ})= \BB_\qs\otimes \Bbbk K\otimes \BB_\qs^*$$
from \Cref{cor:pbw}, where $\BB_\qs$ is the subalgebra generated by $x_1,\ldots, x_r$ and $\BB_\qs^*$ is the subalgebra generated by $y_1,\ldots, y_r$. We define the \emph{positive} and \emph{negative Borel parts} $\ru_q(\fr{sl}_{r,\mJ})^+$ and $\ru_q(\fr{sl}_{r,\mJ})^-$ to be the subalgebras of $\ru_q(\fr{sl}_{r,\mJ})$ generated by the $x_i$ and $\kappa_i$, respectively, the $y_i$ and $\kappa_i$. We note that
\begin{equation}
    \ru_q(\fr{sl}_{r,\mJ})^+=\BB_\qs\rtimes K, \qquad \ru_q(\fr{sl}_{r,\mJ})^-=\BB_\qs^*\rtimes K,
\end{equation}
are Hopf subalgebras of $\ru_q(\fr{sl}_{r,\mJ})$ which are bosonizations (Radford--Majid biproducts) of the corresponding Nichols algebras.

\smallskip 

The simple $\Bbbk K$-module $\Bbbk_{\bf i}$ corresponding to ${\bf i}\in \Lambda$ extends to a $\ru_q(\fr{sl}_{r,\mJ})^+$-module via inflation, i.e. by letting all $x_i$ act by zero. The \emph{standard module} of $\ru_q(\fr{sl}_{r,\mJ})$ associated to ${\bf i}\in \Lambda$ is defined by
\begin{equation}
M({\bf i}):=\Ind_{\ru_q(\fr{sl}_{r,\mJ})^+}^{\ru_q(\fr{sl}_{r,\mJ})}(\Bbbk_{\bf i}).
\end{equation}
By construction, $M({\bf i})$ admits a $\mZ$-grading and is a highest weight module of weight ${\bf i}$. The module $M({\bf i})$ contains a unique maximal submodule and thus a simple head, denoted by $L({\bf i})$. The following theorem is due to \cites{HN,BT} and can be found in \cite{Vay}*{Theorem~2.1(i)--(ii)} in this form.

\begin{theorem}
The set $\{ L({\bf i})|{\bf i}\in \Lambda\}$ uniquely represents all isomorphism classes of simple $\ru_q(\fr{sl}_{r,\mJ})$-modules.
\end{theorem}

Further, the projective covers of $L({\bf i})$ have filtrations by standard modules, and Brauer and BGG reciprocity hold for the algebras $\ru_q(\fr{sl}_{r,\mJ})$ \cite{Vay}*{Theorem~2.1, Theorem~5.11}.

Next, we include a general result that will be used in \Cref{sec:rank2} to study decomposition of some tensor products for the special case $r=2$ and $\mJ=\mI$.

\begin{lemma}\label[lemma]{lem:V-otimes-standard}
Every $V \in \lmod{\ru_q(\fr{sl}_{r,\mJ})}$ has a $\Lambda$-homogeneous basis $v_1,\dots, v_d$, such that for any ${\bf i} \in \Lambda $, the module 
$V\otimes M({\bf i})$ admits a filtration with factors $M \big(\deg v_k+{\bf i}\big)$ for $1\leq k \leq d$, where $\deg v_k$ denotes the $\Lambda$-degree of $v_k$.
\end{lemma}

\begin{proof}
Denote $\ru=\ru_q(\fr{sl}_{r,\mJ})$ and $\ru^+=\ru_q(\fr{sl}_{r,\mJ})^+$. 
As the $x_i\in \rad \ru^+$, every simple $\ru^+$-module is one-dimensional and has a $\Lambda$-homogeneous basis. Thus, there exists a basis $v_1,\ldots, v_d$ of $V$ such that, as a $\ru^+$-module, $V\otimes\Bbbk_{\bf i}$ has a composition series $V_1 \subset \dots\subset V_d$, where $V_k$ is linearly spanned by $\{v_1 \otimes 1, \dots, v_k \otimes 1\}$, and the quotient $V_k/V_{k-1}$ is one-dimensional of weight $\deg v_k + {\bf i}$. 

Now, since $V$ is finite-dimensional, we have natural isomorphisms
\begin{align*}
\Hom_{\ru} (V\otimes M({\bf i}), -)&\simeq\Hom_{\ru} ( \Ind_{\ru^+}^{\ru}(\Bbbk_{\bf i}), V^*\otimes-) \simeq \Hom_{\ru^+} (\Bbbk_{\bf i}, V^*\otimes-) \simeq \Hom_{\ru^+} (V\otimes\Bbbk_{\bf i}, -)\\
&\simeq \Hom_{\ru} (\Ind_{\ru^+}^{\ru}(V\otimes\Bbbk_{\bf i}), -).
\end{align*}
By the PBW decomposition \Cref{cor:pbw}, $\ru$ is free over $\ru^+$, so the functor $\Ind_{\ru^+}^{\ru}$ is exact; applying it to the filtration built for the $\ru^+$-module $V\otimes\Bbbk_{\bf i}$, gives the desired filtration for $V\otimes M({\bf i})\in\lmod{\ru}$. 
\end{proof}

\subsection{The category of graded modules over
\texorpdfstring{$\ru_q(\fr{sl}_{r,\mJ})$}{uq(sl(theta,J)}}

Denote by $\glmod{\ru_q(\fr{sl}_{r,\mJ})}$ the category of finite-dimensional $\mZ$-graded left $\ru_q(\fr{sl}_{r,\mJ})$-modules. For $d\in \mZ$ and $V=\bigoplus_{i\in \mZ} V_i$ a graded $\ru_q(\fr{sl}_{r,\mJ})$-module, we write $V[d]$ for the shift by $d$ of $V$, i.e.~$(V[d])_i=V_{i-d}$. We use the convention that the simple modules $L({\bf i})$ are generated by a highest weight vector of degree zero. The following theorem is a consequence of \cite{BT}*{Theorem~1.1} (see also \cite{Vay}*{Theorem~5.1}).

\begin{theorem}[Bellamy--Thiel]
The category $\glmod{\ru_q(\fr{sl}_{r,\mJ})}$ is a highest weight category. Its set of weights is given by $\Lambda\times \mZ$, with standard modules $\{ M({\bf i})[d] ~|~ {\bf i}\in \Lambda, d\in \mZ\}$ and simple modules $\{ L({\bf i})[d] ~|~ {\bf i}\in \Lambda, d\in \mZ\}$.
\end{theorem}

Given a graded $\ru_q(\fr{sl}_{r,\mJ})$-module $V=\bigoplus_{i\in \mZ}V_i$, we follow \cite{Vay}*{Section 2.1.1} to define the \emph{graded character} of $V$ as
\begin{align*}
    \ch^\bullet V = \sum_{i\in \mZ} \ch V_i t^i \;\in \; \mN \Lambda [t,t^{-1}], 
\end{align*}
where $\ch V_i =\sum_{{\bf j}\in \Lambda} \dim_\Bbbk(V_{i,{\bf j}}){\bf j}$, for a $\Bbbk K$-module decomposition $V_i=\bigoplus_{{\bf j}\in \Lambda} V_{i,{\bf j}}$ as in \eqref{eq:Lambda-decomp}.

\begin{proposition}[{\cite{BT}*{Proposition~3.19}, \cite{Vay}*{Proposition 2.6}}]\label[proposition]{prop:K0comp}
The assigment $V\mapsto \ch^\bullet V$ defines an injective morphism of $\mZ[t,t^{-1}]$-algebras from $K_0(\glmod{\ru_q(\fr{sl}_{r,\mJ})})$ to $\mZ \Lambda [t,t^{-1}]$.
\end{proposition}
Here, $K_0(\cC)$ is the Grothendieck ring of an abelian category $\cC$, which is generated by the symbols $[V]$ for $V$ an object in $\cC$, modulo to the relation that $[B]=[A]+[C]$ if $B$ is an extension of $C$ by $A$. If $\cC$ is a tensor category, then $K_0(\cC)$ is a ring, see e.g. \cite{EGNO}*{Section~4.5}.  Note that if $V$ is a graded $\ru_q(\fr{sl}_{r,\mJ})$-module with filtration 
$$0\leq W_1\leq W_2\leq \ldots W_m=V,$$
with simple subquotients $W_i/W_{i-1}\cong L({\bf j}^i)[d_i]$, then 
$$\ch^\bullet(V)=\sum_{i=1}^{m}\ch^\bullet L({\bf j}^i)t^{d_i}.$$
The graded characters $\ch^\bullet L({\bf j})$ generate the image of $K_0(\glmod{\ru_q(\fr{sl}_{r,\mJ})})$ in $\mZ \Lambda [t,t^{-1}]$ as a $\mZ[t,t^{-1}]$-module. Thus, we can recover the decomposition numbers of $V$ (i.e., how many copies of $L({\bf j})[d]$ appear as filtration factors) from $\ch^\bullet V$.  These observations about graded characters of $\ru_q(\fr{sl}_{r,\mJ})$-modules help us, for instance, to compute tensor product decompositions in \Cref{sec:rank2}.

\smallskip

In \cite{Vay}*{Section 5.1}, several questions are posed for a class of Hopf algebras similar to $\ru_q(\fr{sl}_{r,\mJ})$, including the question to describe $\ch^\bullet L({\bf i})$ for all ${\bf i}\in \Lambda$. In \Cref{sec:rank2}, we answer this and other questions posted by C.~Vay in the case $r=2$ and $\mJ=\mI$.

\subsection{The rank-two case}
\label[section]{sec:rank2}
In this section, we include detailed results on the representations of $\ru_q=\ru_q(\fr{sl}_{r,\mJ})$, for $r=2$ and $\mJ=\mI=\{1,2\}$. In particular, we give an explicit description of the simple modules, compute composition series of the standard modules in \Cref{sec:tensordec}, and include some computations in the Grothendieck ring in \Cref{sec:Koexmpl}.

\subsubsection{Simple modules}

We now describe the simple modules of $\ru_q=\ru_q(\fr{sl}_{2,\mJ})$ as quotients of standard modules.
In this section, $\Lambda= \mZ_{N}\times \mZ_N$.

Let $(i,j)\in \Lambda$ be a weight. We can use the divided power basis from \Cref{sec:divpowers} to obtain that 
$$
\Set{v_{ij}^{(b_2,b_{12},b_1)}:=y_2^{b_2}y_{12}^{(b_{12})}y_1^{b_1}v_{ij}~\middle|~ 0\leq b_1,b_2\leq 1, ~ 0\leq b_{12} \leq N-1}
$$
is a $\Lambda$-homogeneous basis for the standard module $M(i,j)$, where $v_{ij}$ is a highest weight vector. The $\Lambda$-grading is given by 
\begin{align*}
M(i,j)_{(i-k,j-k)}&=\begin{cases}
\Bbbk\inner{ v_{i,j}}, & \text{if }k=0\\
\Bbbk\inner{ v_{i,j}^{(1,k-1,1)}, v_{ij}^{(0,k,0)}}, & \text{if }k=1, \ldots, N-1\\
\Bbbk\inner{v_{i,j}^{(1,N-1,1)}}, & \text{if }k=N
\end{cases}\\
M(i,j)_{(i-k,j-k+1)}&=\Bbbk \inner{ v_{i,j}^{(0,k-1,1)}}, \qquad 1\leq k\leq N-1\\
M(i,j)_{(i-k+1,j-k)}&=\Bbbk \inner{ v_{i,j}^{(1,k-1,0)}},\qquad 1\leq k\leq N-1.
\end{align*}

Next, from a direct application of the commutation relations obtained in \Cref{lem:commrel-rank2}, we compute the action of the generators $x_i, y_i$ on these basis elements. Note that, in the rank $2$ case, the group-like elements $k_i$ defined in \Cref{eq:ki-def} specialize to $k_1= \gamma_1\ov{\gamma}_1=\kappa_2$ and $k_2=\gamma_2\ov{\gamma}_2=\kappa_1$.

\begin{lemma}\label[lemma]{lem:standard-action} The $\ru_q$-action on the above basis for $M(i,j)$ is determined by the following formulae.
\begin{align*}
    y_1v_{ij}^{(0,k,0)}&=(-1)^kq^kv_{ij}^{(0,k,1)},
    &y_2v_{ij}^{(0,k,0)}&=v_{ij}^{(1,k,0)}, 
    \\  
    y_1v_{ij}^{(0,k,1)}&=0,
    &y_2v_{ij}^{(0,k,1)}&=v_{ij}^{(1,k,1)},
    \\
    y_1v_{ij}^{(1,k,0)}&=(1-q)[k+1]_qv_{ij}^{(0,k+1,0)}+(-1)^kq^{k+1}v_{ij}^{(1,k,1)},
    &y_2v_{ij}^{(1,k,0)}&=0, 
    \\
    y_1v_{ij}^{(1,k,1)}&=(1-q)[k+1]_qv_{ij}^{(0,k+1,1)},
    &y_2v_{ij}^{(1,k,1)}&=0,
    \\
    x_1 v_{ij}^{(0,k,0)}&=-q^{j-k}v_{ij}^{(1, k-1, 0)},&x_2 v_{ij}^{(0,k,0)}&=(-1)^{k-1}v_{ij}^{(0,k-1,1)},
    \\
    x_1v_{ij}^{(0,k,1)}&=(-q)^{-k}(1- q^j)v_{ij}^{(0,k,0)}-q^{j-k}v_{ij}^{(1,k-1,1)}, &x_2v_{ij}^{(0,k,1)}&=0,
    \\
    x_1v_{ij}^{(1,k,0)}&=0,&x_2v_{ij}^{(1,k,0)}&=(1-q^{i-k})v_{ij}^{(0,k,0)}+(-1)^kv_{ij}^{(1,k-1,1)},
    \\
    x_1 v_{ij}^{(1,k,1)}&= (-1)^kq^{-k-1}(1-q^j)v_{ij}^{(1,k,0)},&x_2 v_{ij}^{(1,k,1)}&=
(1-q^{i-k-1})v_{ij}^{(0,k,1)},
\end{align*}
where, for $k=0$, the terms of the form $v_{ij}^{(b_2,k-1,b_1)}$ are omitted. 
\end{lemma}

\begin{lemma}\label[lemma]{lem:combo-hwv}
Let $0\leq i,j <N$ and $1\leq k <N$; consider $\lambda, \mu \in \Bbbk$ not both zero. Then the vector 
$$w=\lambda v_{ij}^{(0,k,0)}+ \mu v_{ij}^{(1,k-1,1)} \in M(i,j)$$
is highest weight if and only if $k=i+j \mod N$ and $\lambda = (-1)^k(1-q^{-j}) \mu$.
\end{lemma}

\begin{proof}
Using \Cref{lem:standard-action}, we compute 
\begin{align*}
x_1w&=\left(\lambda (-q^{j-k})+ \mu (-1)^{k-1}q^{-k}(1-q^j)\right)v_{ij}^{(1,k-1,0)}, &
x_2w&=\left(\lambda (-1)^{k-1}+ \mu (1-q^{i-k})\right)v_{ij}^{(0,k-1,1)}.
\end{align*}
Thus $x_1 w=0$ if and only if $\lambda=\mu (-1)^k(1-q^{-j})$, and $x_2w=0$ if and only if $\lambda = \mu (-1)^k(1-q^{i-k})$.
Now, since $\lambda, \mu$ are not both zero, it is evident that $w$ is a highest weight vector if and only if  $q^{k}=q^{i+j}$ and $\lambda=\mu (-1)^k(1-q^{-j})$.
\end{proof}

The following Lemma computes generators for the maximal submodule of a standard module $M(i,j)$. This, in turn, directly provides an explicit description of the simple modules.

\begin{lemma}\label[lemma]{lem:maximal-submod}
We distinguish four cases for $0\leq i,j<N$.
\begin{enumerate}
    \item \label{item:maximal-00} The maximal submodule of $M(0,0)$ is generated by $v_{00}^{(1,0,0)}, v_{00}^{(0,0,1)}$.
    \item \label{item:maximal-i0} If $i\neq 0\mod N$, the maximal submodule of $M(i,0)$ is generated by $v_{i0}^{(0,0,1)}, v_{i0}^{(1,i,0)}$.
    \item \label{item:maximal-0j} If $j\neq 0\mod N$, the maximal submodule of $M(0,j)$ is generated by $v_{0j}^{(1,0,0)}, v_{0j}^{(0,j,1)}$.
    \item \label{item:maximal-i+jnotN} Assume that $i,j\neq 0\mod N$, and $i+j\neq N$. Let $0< k < N$ such that  $k=i+j\mod N$. Then 
    $$w=(-1)^k(1-q^{-j})v_{ij}^{(0,k,0)}+v_{ij}^{(1,k-1,1)}$$
    generates the maximal submodule of $M(i,j)$.
     \item \label{item:maximal-i+j=N} Assume that $i,j\neq 0\mod N$ and $i+j=N$. Then $M(i,j)$ is simple.
\end{enumerate}
\end{lemma}

\begin{proof}
For \eqref{item:maximal-00}, it is clear from \Cref{lem:standard-action} that $v_{00}^{(1,0,0)}$ and $v_{00}^{(0,0,1)}$ are highest weight vectors. Also, the quotient of $M(0,0)$ by the ideal generated by these vectors is $1$-dimensional spanned by $v_{00}$.

\eqref{item:maximal-i0} It follows directly from \Cref{lem:standard-action} that $v_{i0}^{(0,0,1)}$ is a highest weight vector in $M(i,0)$. Next, we show that (the image of) $v_{i0}^{(1,i,0)}$ is a highest weight vector in $M'(i,0)=M(i,0)\Big/\left\langle v_{i0}^{(0,0,1)}\right\rangle$. Indeed, $x_1$ vanishes $v_{i0}^{(1,i,0)}$ already in $M(i,0)$, and $x_2 v_{i0}^{(1,i,0)}= (-1)^iv_{i0}^{(1,i-1,1)}$, which vanishes in $M'(i,0)$. Finally, in the quotient $M(i,0)\Big/\left \langle v_{i0}^{(0,0,1)}, v_{i0}^{(1,i,0)} \right \rangle$ the unique highest weight vector is $v_{i0}$.

\eqref{item:maximal-0j} In this case $v_{0j}^{(1,0,0)}$ is automatically a highest weight vector in $M(0,j)$. The next step is to show that the vector $v_{0j}^{(0,j,1)}$ is highest weight in $M(0,j)'=M(0,j)\Big/ \left \langle v_{0j}^{(1,0,0)} \right \rangle$. We have $x_2v_{0j}^{(0,j,1)}=0$ already in $M(0,j)$, and computing in $M'(0,j)$ we get
\begin{align*}
x_1v_{0j}^{(0,j,1)}&=(-q)^{-j}(1- q^j)v_{0j}^{(0,j,0)}-v_{0j}^{(1,j-1,1)}=(-q)^{-j}(1- q^j)y_{12}^{(j)}v_{0j}-y_2y_{12}^{(j-1)}y_1v_{0j}\\
&=\frac{(-q)^{-j}(1-q^j)(-q)^j}{(1-q)^j[j]_q!}(y_{2}y_{1})^{j}v_{0j}-\frac{1}{(1-q)^{j-1}[j-1]_q!}(y_{2}y_{1})^{j}v_{0j}\\
&=\frac{1}{(1-q)^{j-1}[j-1]_q!}\left( \frac{1-q^j}{(1-q)[j]_q} -1 \right) (y_{2}y_{1})^{j}v_{0j} = 0.
\end{align*}
Now the claim follows because $v_{0j}$ is the unique highest weight vector in $M(0,j)\Big /\left \langle v_{0j}^{(1,0,0)}, v_{0j}^{(0,j,1)} \right \rangle$.

\eqref{item:maximal-i+jnotN} We know that $w$ is a highest weight vector by \Cref{lem:combo-hwv}. 
Again, one can show that $v_{ij}$ is the unique highest weight vector in $M(i,j)/\langle w \rangle$. 

\eqref{item:maximal-i+j=N} In this case \Cref{lem:combo-hwv} implies that $v_{ij}$ is the unique highest weight vector in $M(i,j)$.
\end{proof}

As a direct consequence of \Cref{lem:maximal-submod} and \cite{Vay}*{Corollary~5.12}, we get.

\begin{corollary}\label[corollary]{cor:standardproj}
For $0\leq i,j < N$, the module $M(i,j)$ is projective if and only if $i+j=0\mod N$. 
\end{corollary}

Hence, there are exactly $N-1$ projective standard modules, namely $M(i,N-i)$ for $i=1, \ldots, N-1$.

The above Lemmas allow us to fully describe the simple module of $\ru_q$. For this, we use labeled directed graphs with vertices corresponding to vectors forming a $\Lambda$-homogeneous basis of the simple module, and edges of the form 
\begin{gather*}
\vcenter{\hbox{
\begin{tikzpicture}
\tikzstyle{ann} = [circle,draw,fill,scale=0.75]
\node [ann,label=above:\scriptsize${(a,b)}$] (1){};
\node [ann,label=above:\scriptsize${(a-1,b)}$] (2)[right of=1,node distance=2.5cm]{};
\begin{scope}[very thick,decoration={markings,mark=at position 0.5 with {\arrow{>}}}]
\draw[postaction={decorate}] (1.east) to node[midway,above] {$1$} (2.west);
\end{scope}
\end{tikzpicture}}}
\qquad \text{or} \qquad 
\vcenter{\hbox{\begin{tikzpicture}
\tikzstyle{ann} = [circle,draw,fill,scale=0.75]
\node [ann,label=above:\scriptsize${(a,b)}$] (1){};
\node [ann,label=above:\scriptsize${(a,b-1)}$] (2)[right of=1,node distance=2.5cm]{};
\begin{scope}[very thick,decoration={markings,mark=at position 0.5 with {\arrow{>}}}]
\draw[postaction={decorate}] (1.east) to node[midway,above] {$2$} (2.west);
\end{scope}
\end{tikzpicture},}}
\end{gather*}
where the label $i=1,2$ in the edge indicate that the vector corresponding to the left vertex, of weight $(a,b)\in \Lambda$, is mapped by $y_i$ to a nonzero vector of
$\Lambda$-degree given by the label of the right vertex. 
We will use horizontal concatenation of these pictures, where the left-most vertex corresponds to a highest-weight generator and the right-most vector is annihilated by $y_1$ and $y_2$.
Such diagrams can be used to display the graded dimensions of the standard modules $M(i,j)$:
\begin{align*}
\vcenter{\hbox{
\begin{tikzpicture}
\tikzstyle{ann} = [circle,draw,fill,scale=0.75]
\node [ann,label=left:\scriptsize${(i,j)}$] (1){};
\node [above of=1,node distance=1cm,scale=0.75] (10){};
\node [below of=1,node distance=1cm,scale=0.75] (20){};
\node [ann,label=above:\scriptsize${(i,j-1)}$] (2)[right of=10,node distance=2.5cm]{};
\node [ann,label=below:\scriptsize${(i-1,j)}$] (12)[right of=20,node distance=2.5cm]{};
\node [ann,label=above:\scriptsize${(i-1,j-1)}$] (3)[right of=2,node distance=2.5cm]{};
\node [ann,label=below:\scriptsize${(i-1,j-1)}$] (13)[right of=12,node distance=2.5cm]{};
\node [ann,label=above:\scriptsize${}$] (4)[right of=3,node distance=2.5cm]{};
\node [ann,label=below:${}$] (14)[right of=13,node distance=2.5cm]{};
\node [ann,label=above:\scriptsize${}$] (5)[right of=4,node distance=1.5cm]{};
\node [ann,label=below:${}$] (15)[right of=14,node distance=1.5cm]{};
\node [ann,label=above:\scriptsize${(i\!-\!N\!+\!1,j\!-\!N\!+\!1)}$] (6)[right of=5,node distance=2.8cm]{};
\node [ann,label=below:\scriptsize${(i\!-\!N\!+\!1,j\!-\!N\!+\!1)}$] (16)[right of=15,node distance=2.8cm]{};
\node [ann,label=above:\scriptsize${(i\!-\!N\!+\!1,j\!-\!N)}$] (7)[right of=6,node distance=3.5cm]{};
\node [ann,label=below:\scriptsize${(i\!-\!N,j\!-\!N\!+\!1)}$\;\;\;] (17)[right of=16,node distance=3.5cm]{};
\node [right of=7,node distance=1.5cm,scale=0.75] (8){};
\node [ann,label=right:\scriptsize${\genfrac{}{}{0pt}{}{(i-N,j-N)}{=(-j,-i)}}$] (9)[below of=8,node distance=1.25cm]{};
\begin{scope}[very thick,decoration={markings,mark=at position 0.5 with {\arrow{>}}}]
\draw[postaction={decorate}] (1.north east) to node[midway,above] {$2$} (2.west);
\draw[postaction={decorate}] (2.east) to node[midway,below] {$1$} (3.west);
\draw[postaction={decorate}] (3.east) to node[midway,below] {$2$} (4.west);
\draw [dotted] (4) to (5);
\draw[postaction={decorate}] (5.east) to node[midway,below] {$1$} (6.west);
\draw[postaction={decorate}] (6.east) to node[midway,below] {$2$} (7.west);
\draw[postaction={decorate}] (7.east) to node[midway,above] {$\;1$} (9.north west);
\draw[postaction={decorate}] (1.south east) to node[midway,above] {$\;1$} (12.west);
\draw[postaction={decorate}] (12.east) to node[midway,above] {$2$} (13.west);
\draw[postaction={decorate}] (13.east) to node[midway,above] {$1$} (14.west);
\draw [dotted] (14) to (15);
\draw[postaction={decorate}] (15.east) to node[midway,above] {$2$} (16.west);
\draw[postaction={decorate}] (16.east) to node[midway,above] {$1$} (17.west);
\draw[postaction={decorate}] (17.east) to node[midway,above] {$2$} node[midway,right] {} (9.south west);
\end{scope}
\end{tikzpicture}
}}
\end{align*}
Below, we will use such diagrams to represent the simple modules $L(i,j)$. The vertices of these diagrams do not necessarily correspond to the basis $v_{ij}^{(b_2,b_{12},b_1)}$ above.

\begin{theorem}\label{thm:rank2}
The following is a complete list of non-isomorphic simple $\ru_q$-modules.
\begin{enumerate}
    \item $L(0,0)=\one$ is the tensor unit, the unique simple $1$-dimensional $\ru_q$-module given by $$\vcenter{\hbox{\begin{tikzpicture}
    \tikzstyle{ann} = [circle,draw,fill,scale=0.75]
\node [ann,label=above:\scriptsize${(0,0)}$] (1){};
    \end{tikzpicture}}}$$
    \item For any $0< j<N$, $L(0,j)$ is given by 
    \begin{align*}
\vcenter{\hbox{
\begin{tikzpicture}
\tikzstyle{ann} = [circle,draw,fill,scale=0.75]
\node [ann,label=above:\scriptsize${(0,j)}$] (1){};
\node [ann,label=above:\scriptsize${(-1,j)}$] (2)[right of=1,node distance=2.5cm]{};
\node [ann,label=above:\scriptsize${(-1,j-1)}$] (3)[right of=2,node distance=2.5cm]{};
\node [ann,label=above:\scriptsize${(-j+1,1)}$] (4)[right of=3,node distance=2.5cm]{};
\node [ann,label=above:\scriptsize${(-j,1)}$] (5)[right of=4,node distance=2.5cm]{};
\node [ann,label=above:\scriptsize${(-j,0)}$] (6)[right of=5,node distance=2.5cm]{};
\begin{scope}[very thick,decoration={markings,mark=at position 0.5 with {\arrow{>}}}]
\draw[postaction={decorate}] (1.east) to node[midway,above] {$1$} (2.west);
\draw[postaction={decorate}] (2.east) to node[midway,above] {$2$} (3.west);
\draw [dotted] (3) to (4);
\draw[postaction={decorate}] (4.east) to node[midway,above] {$1$} (5.west);
\draw[postaction={decorate}] (5.east) to node[midway,above] {$2$} (6.west);
\end{scope}
\end{tikzpicture}
}}
    \end{align*}
        \item For any $0< i<N$, $L(i,0)$ is given by 
    \begin{align*}
\vcenter{\hbox{
\begin{tikzpicture}
\tikzstyle{ann} = [circle,draw,fill,scale=0.75]
\node [ann,label=above:\scriptsize${(i,0)}$] (1){};
\node [ann,label=above:\scriptsize${(i,-1)}$] (2)[right of=1,node distance=2.5cm]{};
\node [ann,label=above:\scriptsize${(i-1,-1)}$] (3)[right of=2,node distance=2.5cm]{};
\node [ann,label=above:\scriptsize${(1,-i+1)}$] (4)[right of=3,node distance=2.5cm]{};
\node [ann,label=above:\scriptsize${(1,-i)}$] (5)[right of=4,node distance=2.5cm]{};
\node [ann,label=above:\scriptsize${(0,-i)}$] (6)[right of=5,node distance=2.5cm]{};
\begin{scope}[very thick,decoration={markings,mark=at position 0.5 with {\arrow{>}}}]
\draw[postaction={decorate}] (1.east) to node[midway,above] {$2$} (2.west);
\draw[postaction={decorate}] (2.east) to node[midway,above] {$1$} (3.west);
\draw [dotted] (3) to (4);
\draw[postaction={decorate}] (4.east) to node[midway,above] {$2$} (5.west);
\draw[postaction={decorate}] (5.east) to node[midway,above] {$1$} (6.west);
\end{scope}
\end{tikzpicture}
}}
    \end{align*}
\item  For $0<i,j<N$, $L(i,j)$ is given by     
\begin{align*}
\vcenter{\hbox{
\begin{tikzpicture}
\tikzstyle{ann} = [circle,draw,fill,scale=0.75]
\node [ann,label=left:\scriptsize${(i,j)}$] (1){};
\node [above of=1,node distance=1cm,scale=0.75] (10){};
\node [below of=1,node distance=1cm,scale=0.75] (20){};
\node [ann,label=above:\scriptsize${(i,j-1)}$] (2)[right of=10,node distance=2.85cm]{};
\node [ann,label=below:\scriptsize${(i-1,j)}$] (12)[right of=20,node distance=2.85cm]{};
\node [ann,label=above:\scriptsize${(i-1,j-1)}$] (3)[right of=2,node distance=2.5cm]{};
\node [ann,label=below:\scriptsize${(i-1,j-1)}$] (13)[right of=12,node distance=2.5cm]{};
\node [ann,label=above:\scriptsize${}$] (4)[right of=3,node distance=2.5cm]{};
\node [ann,label=below:${}$] (14)[right of=13,node distance=2.5cm]{};
\node [ann,label=above:\scriptsize${}$] (5)[right of=4,node distance=1.5cm]{};
\node [ann,label=below:${}$] (15)[right of=14,node distance=1.5cm]{};
\node [ann,label=above:\scriptsize${(i\!-\!k\!+\!1,j\!-\!k\!+\!1)}$] (6)[right of=5,node distance=2.5cm]{};
\node [ann,label=below:\scriptsize${(i\!-\!k\!+\!1,j\!-\!k\!+\!1)}$] (16)[right of=15,node distance=2.5cm]{};
\node [ann,label=above:\scriptsize${(i\!-\!k\!+\!1,j\!-\!k)}$] (7)[right of=6,node distance=3cm]{};
\node [ann,label=below:\scriptsize${(i\!-\!k,j\!-\!k\!+\!1)}$\;\;\;] (17)[right of=16,node distance=3cm]{};
\node [right of=7,node distance=1.5cm,scale=0.75] (8){};
\node [ann,label=right:\scriptsize${\genfrac{}{}{0pt}{}{(i-k,j-k)}{=(-j,-i)}}$] (9)[below of=8,node distance=1.25cm]{};
\begin{scope}[very thick,decoration={markings,mark=at position 0.5 with {\arrow{>}}}]
\draw[postaction={decorate}] (1.north east) to node[midway,above] {$2$} (2.west);
\draw[postaction={decorate}] (2.east) to node[midway,below] {$1$} (3.west);
\draw[postaction={decorate}] (3.east) to node[midway,below] {$2$} (4.west);
\draw [dotted] (4) to (5);
\draw[postaction={decorate}] (5.east) to node[midway,below] {$1$} (6.west);
\draw[postaction={decorate}] (6.east) to node[midway,below] {$2$} (7.west);
\draw[postaction={decorate}] (7.east) to node[midway,above] {$\;1$} (9.north west);
\draw[postaction={decorate}] (1.south east) to node[midway,above] {$\;1$} (12.west);
\draw[postaction={decorate}] (12.east) to node[midway,above] {$2$} (13.west);
\draw[postaction={decorate}] (13.east) to node[midway,above] {$1$} (14.west);
\draw [dotted] (14) to (15);
\draw[postaction={decorate}] (15.east) to node[midway,above] {$2$} (16.west);
\draw[postaction={decorate}] (16.east) to node[midway,above] {$1$} (17.west);
\draw[postaction={decorate}] (17.east) to node[midway,above] {$2$} node[midway,right] {} (9.south west);
\end{scope}
\end{tikzpicture}
}}
\end{align*}
Here, we choose the unique representative $1\leq k \leq N$ of $i+j$ modulo $N$. In particular, if $i+j=N$, then $L(i,j)=M(i,j)$.
\end{enumerate}
\end{theorem}

We observe that \Cref{thm:rank2} answers \cite{Vay}*{Question~5.4} since we have determined the graded characters $\ch^\bullet L(i,j)$ for all simple $\ru_q$-modules.
Note that the Hilbert series of all of these simple modules are symmetric answering \cite{Vay}*{Question~5.5} in the affirmative for this Hopf algebra.

As a direct consequence of the above computations, and \Cref{cor:dimq}, we can find the dimensions and quantum dimensions of all simple $\ru_q$-modules.

\begin{corollary}\label[corollary]{cor:simples-dimension}
The dimensions of the simple $\ru_q$-modules are given by 
\begin{equation}
    \dim L(i,j)=\begin{cases}
    2i+1, & \text{if } j=0,\\
    2j+1, & \text{if } i=0,\\
    4(i+j), &\text{if } i,j\neq 0, \, 1< i+j\leq N,\\
    4(i+j-N), &\text{if } i,j\neq 0, \, N< i+j< 2N.\\
    \end{cases}
\end{equation}
The quantum dimensions are given by 
\begin{equation}
    \dim_q L(i,j)=\begin{cases}
    (-1)^i, & \text{if } j=0,\\
    (-1)^j, & \text{if } i=0,\\
    0, &\text{if } i,j\neq 0. 
    \end{cases}
\end{equation} 
\end{corollary}

The dimensions observed match those of certain typical modules, which are $4j$-dimensional, and atypical modules, which are $(2j+1)$-dimensional, for $j$ a natural number, of the Lie superalgebra $\fr{sl}(1|2)$ \cite{FScS}*{Section~2.53}. 

\smallskip

\Cref{cor:standardproj} implies that the projective standard modules  $M(i,N-i)$, for $i=1,\ldots, N-1$, are precisely the simple projective modules of $\ru_q$. We note the other simple modules are not projective.
\begin{corollary}
The following are equivalent for $i,j=0,\ldots, N-1$.
\begin{enumerate}
    \item[(i)] The simple module $L(i,j)$ is projective.
    \item[(ii)] $L(i,j)=M(i,j)$.
    \item[(iii)] The standard module $M(i,j)$ is projective.
    \item[(iv)] $i+j=N$.
\end{enumerate} 
\end{corollary}
\begin{proof}
With the above observations, it remains to show that $L(i,j)$ is not projective if $i+j\neq N$. This follows from \cite{Vay}*{Theorem~2.1} which states that the projective covers of the $L(i,j)$ are filtered by standard modules. All standard modules have dimension $4N$, but the simple module $L(i,j)$ with $i+j \neq N$ has strictly smaller dimension and hence cannot have a standard filtration.
\end{proof}

\begin{corollary}
We obtain the following duality relation for simple $\ru_q$-modules
\begin{gather*}
    \qquad L(i,j)^*\cong L(j,i).
\end{gather*}
\end{corollary}

For Hopf algebras similar to $\ru_q$, \cite{Vay}*{Question~5.6} asks whether the bijection between highest weights and lowest weights (i.e., highest weights of the dual) of simple modules corresponds to an autoequivalence of the category of $\Bbbk[G]$-modules. In the case of $\ru_q$ considered in this section, the question can be answered in the affirmative through the equivalence of $\Bbbk \Lambda$-modules induced by the Hopf algebra involution 
$$\iota\colon \Bbbk[G]\to \Bbbk[G],\quad \delta_{(i,j)} \mapsto \delta_{(j,i)},$$
which is induced by the group involution $g_1^ig_2^j\mapsto g_1^{j}g_2^{i}$.

\begin{example}
The smallest-dimensional non-trivial $\ru_q(\mathfrak{sl}_{2,\mI})$-module is $V=L(1,0)$, which is $3$-dimensional. To examine its braiding, let $v$ be a highest weight vector for $V$ and consider the basis 
$$v_0=v, \qquad v_1=y_2\cdot v,\qquad  v_2=y_1\cdot v_1.$$ Then, it follows that  
$$x_2\cdot v_1=(1-q)v_0, \qquad x_1\cdot v_2=(1-q^{-1})v_1.$$
Using the lexicographic order on the basis $\{v_i\otimes v_j\}$ of $V\otimes V$, the braiding is given by the  $9\times 9$-matrix
$$
\Psi=\begin{pmatrix}
-1 & 0 & 0 & 0 & 0 & 0 & 0 & 0 & 0 
\\
 0 & q^{-1}-1 & 0 & -1 & 0 & 0 & 0 & 0 & 0 
\\
 0 & 0 & q^{-1}-1 & 0 & 0 & 0 & 1 & 0 & 0 
\\
 0 & -q^{-1} & 0 & 0 & 0 & 0 & 0 & 0 & 0 
\\
 0 & 0 & 0 & 0 & q^{-1} & 0 & 0 & 0 & 0 
\\
 0 & 0 & 0 & 0 & 0 & q^{-1}-1 & 0 & -1 & 0 
\\
 0 & 0 & q^{-1} & 0 & 0 & 0 & 0 & 0 & 0 
\\
 0 & 0 & 0 & 0 & 0 & -q^{-1} & 0 & 0 & 0 
\\
 0 & 0 & 0 & 0 & 0 & 0 & 0 & 0 & -1 
\end{pmatrix}.
$$
The module $V$ is \emph{not} self-dual and $V^*\cong L(0,1)$. Thus, we also need to consider the braiding 
$$\Psi'\colon V\otimes V^*\to V^*\otimes V.$$
This braiding is given by the $9\times 9$-matrix
$$\Psi'=\begin{pmatrix}
q & 0 & 0 & 0 & 0 & 0 & 0 & 0 & 0 
\\
 0 & 0 & 0 & -q & 0 & 0 & 0 & 0 & 0 
\\
 0 & 0 & q^{-1}-1 & 0 & -1+q & 0 & -1 & 0 & 0 
\\
 0 & -q & 0 & 0 & 0 & 0 & 0 & 0 & 0 
\\
 0 & 0 & q^{-1}-1 & 0 & q & 0 & 0 & 0 & 0 
\\
 0 & 0 & 0 & 0 & 0 & 0 & 0 & -1 & 0 
\\
 0 & 0 & -1 & 0 & 0 & 0 & 0 & 0 & 0 
\\
 0 & 0 & 0 & 0 & 0 & -1 & 0 & 0 & 0 
\\
 0 & 0 & 0 & 0 & 0 & 0 & 0 & 0 & 1 
\end{pmatrix}.
$$
One computes that the twist  $\theta_{V}$ is given by the identity.
\end{example}

\begin{lemma}
The braiding $\Psi$ and $\Psi'$ on $V\otimes V$, respectively, $V\otimes V^*$, satisfy the Skein relations 
\begin{gather}
    q \Psi-\Psi^{-1}=(1-q)\ide_{V\otimes V},\label{skein-L10}\\
    (\Psi')^* \Psi' + q((\Psi')^* \Psi')^{-1} = (1+q)\ide_{V\otimes V^*},\\
    \Psi' (\Psi')^* + q(\Psi' (\Psi')^*)^{-1} = (1+q)\ide_{V^*\otimes V}.
\end{gather}
\end{lemma}

We conclude this section with an open question, answered in the case $i+j=N$ by \Cref{cor:standardproj}.
\begin{question}
What are the projective covers of the simple modules $L(i,j)$?
\end{question}

\subsubsection{Composition series of standard modules}\label[section]{sec:tensordec}

\begin{proposition}\label[proposition]{prop:standard-composition}
Consider the standard modules $M(i,j)$, where $0\leq i,j<N$.
\begin{enumerate}
    \item \label{item:standard00-composition} The module  $M(0,0)$ has series and factors
\begin{gather*}
0\leq M_1=\underset{\dim 1}{\left \langle v^{(1,N-1,1)}_{00} \right \rangle}\leq M_2=\underset{\dim 2N}{\left \langle v^{(0,0,1)}_{00} \right \rangle} \leq M_3=\underset{\dim 4N-1}{\left \langle v^{(0,0,1)}_{00}, v^{(1,0,0)}_{00} \right\rangle}\leq M_4=\underset{\dim 4N}{M(0,0)},\\
M_1\cong\underset{\dim 1}{L(0,0)}, \quad M_2/M_1\cong \underset{\dim 2N-1}{L(N-1,0)} ,\quad M_3/M_2\cong  \underset{\dim 2N-1}{ L(0,N-1)},\quad M_4/M_3\cong \underset{\dim 1}{L(0,0)}.
\end{gather*}

\item \label{item:standardi0-composition} For any $0<i<N$, $M(i,0)$ has series and factors
\begin{gather*}
0\leq M_1=\underset{\dim 2(N-i)+1}{\left \langle v^{(1,i-1,1)}_{i0}\right\rangle} \leq M_2=\underset{\dim 2N}{\left\langle  v^{(0,0,1)}_{i0} \right \rangle} \leq M_3=\underset{\dim 4N-2i-1}{\left \langle v^{(0,0,1)}_{i0}, v^{(1,i,0)}_{i0} \right \rangle}\leq M_4=\underset{\dim 4N}{M(i,0)},\\
M_1\cong\underset{\dim 2(N-i)+1}{L(0,N-i)}, \quad M_2/M_1\cong \underset{\dim 2i-1}{L(i-1,0)} ,\quad M_3/M_2\cong  \underset{\dim 2(N-i)-1}{ L(0,N-i-1)},\quad M_4/M_3\cong \underset{\dim 2i+1}{L(i,0)}.
\end{gather*}

\item \label{item:standard0j-composition} For any $0<j<N$, $M(0,j)$ has series and factors
\begin{gather*}
0\leq M_1=\underset{\dim 2(N-j)+1}{\langle(y_1y_2)^j v_{0j}}\rangle \leq M_2=\underset{\dim 2N}{\langle  v_{0j}^{(1,0,0)} \rangle} \leq M_3=\underset{\dim 4N-2j-1}{\langle v_{0j}^{(1,0,0)}, v_{0j}^{(1
0,j,1)} \rangle}\leq M_4=\underset{\dim 4N}{M(0,j)},\\
M_1\cong\underset{\dim 2(N-j)+1}{L(N-j,0)}, \quad M_2/M_1\cong \underset{\dim 2j-1}{L(0,j-1)} ,\quad M_3/M_2\cong  \underset{\dim 2(N-j)-1}{ L(N-j-1,0)},\quad M_4/M_3\cong \underset{\dim 2j+1}{L(j,0)}.
\end{gather*}

\item \label{item:standard-i+jne0-composition} If $0<i+j \neq N$, let $w$ as in \Cref{lem:maximal-submod}. Then $M(i,j)$ has series and factors
\begin{gather*}
0\leq M_1=\underset{\dim 4(N-j-i)}{\langle w \rangle} \leq M_2=\underset{\dim 4N}{M(i,j)},\qquad \qquad
M_1\cong\underset{\dim 4(N-j-i)+1}{L(N-j,N-i)}, \quad M_2/M_1\cong \underset{\dim 4(i+j)}{L(i,j)}.
\end{gather*}

\item \label{item:standard-i+j=0-composition}If $0<i+j=N$, then $M(i,j)=L(i,j)$ is simple.
\end{enumerate}
\end{proposition}

\begin{proof}
\eqref{item:standard00-composition} The submodule $M_2$ generated by the highest weight vector $v^{(0,0,1)}_{00}$ of $\Lambda$-degree $(N-1,0)$ is clearly $2N$-dimensional and contains $M_1=\left \langle v^{(1,N-1,1)}_{00} \right \rangle$, which is $1$-dimensional of $\Lambda$-degree $(0,0)$. Hence $M_2/M_1$ is a highest 
weight module generated in degree $(N-1,0)$, and by $\Lambda$-degree limitations the quotient must be isomorphic to $L(N-1,0)$. Similarly,  $M_3/M_2$ is generated by the highest weight vector $v_{00}^{(1,0,0)}$ of $\Lambda$-degree $(0,N-1)$, which by dimension restrictions must be isomorphic to $L(0,N-1)$. Finally, it is clear that $M_4/M_3$ is isomorphic to $L(0,0)$.

\eqref{item:standardi0-composition}  Note first that $v^{(1,i-1,1)}_{i0}$ is a highest weight vector:
\begin{align*}
x_1 v^{(1,i-1,1)}_{i0}&=(-1)^k(1-q^0) v^{(1,i-1,0)}_{i0}=0, & x_2 v^{(1,i-1,1)}_{i0}&=(q^i-q^i) v^{(0,i-1,1)}_{i0}=0.
\end{align*} 
Hence $M_1$ is a highest weight module generated in $\Lambda$-degree $(0,N-i)$, and it clearly has dimension $2(N-i)+1$. Thus $M_1\simeq L(0,N-i)$. Next, the submodule $M_2$ is clearly $2N$-dimensional, generated by the highest weight vector $v^{(0,0,1)}_{i0}$ of $\Lambda$-degree $(i-1,0)$. Thus $M_2/M_1$ has dimension $2N-(2(N-i)+1)=2i-1$, and it must be isomorphic to $L(i-1,0)$. On the other hand, by \Cref{lem:maximal-submod}, $M_3$ is the maximal submodule of $M_4=M(i,0)$ and $M_4/M_3\simeq L(i,0)$ by definition. Thus the dimension of $M_3$ is $4N-2i-i$. Finally, $M_3/M_2$ is generated by the vector $v_{i0}^{(1,i,0)}$ of highest weight $(0,N-i-1)$, and since the dimension of $M_3/M_2$ is $2(N-i)-1$, we have $M_3/M_2\simeq L(0,N-i-1)$. The proof of \eqref{item:standard0j-composition} is analogous. 

\eqref{item:standard-i+jne0-composition} Assume $i+j<N$. By \Cref{cor:simples-dimension} we have $\dim M=4(i+j)=\dim L(N-j, N-i)$. Since $M$ is generated by a heighest weight vector of $\Lambda$-degree $(N-j,N-i)$, we have $M\simeq L(N-j,N-i)$. The case $i+j>N$ follows similarly. Finally, \eqref{item:standard-i+j=0-composition} is a direct consequence of \Cref{lem:maximal-submod}.
\end{proof}

\subsubsection{Some tensor product decompositions}\label[section]{subsec:tensordec}

For tensor products of standard modules, we obtain:

\begin{proposition}\label[proposition]{prop:tensordec}
For any $0\leq a,b,c,d <N$, we have 
\begin{align*}
M(a,b) \otimes M(c,d) \simeq& M(a+c, b+d) \oplus \bigoplus_{l=1}^N  M(a+c-l, b+d-l+1) \oplus M(a+c-l+1, b+d-l)\\
&\oplus \bigoplus_{l=1}^{N-1} M(a+c-l, b+d-l)^{\oplus2} \oplus M(a+c-N, b+d-N).
\end{align*}
\end{proposition}

\begin{proof}
We denote $M=M(a,b)$, $M'=M(c,d)$.
We first claim that the different $\mZ$-homogeneous components of $M\otimes M'$ contain the following highest weight vectors:
\begin{itemize}
\item $(M\otimes M')_{0}$ contains a highest weight in $\Lambda$-degree $(a+c, b+d)$.

\item For $0<2l<2N$, $(M\otimes M')_{-2l}$ contains two linearly independent highest weights in $\Lambda$-degree $(a+c-l, b+d-l)$.

\item  For $0<2l-1<2N$, $(M\otimes M')_{-2l+1}$ contains two linearly independent highest weights in $\Lambda$-degrees $(a+c-l, b+d-l+1)$ and $(a+c-l+1, b+d-l)$.
\end{itemize}
Next, we claim that in $M \otimes M'$, linearly independent highest weight vectors generate linearly independent standard submodules. 
Indeed, given a highest weight vector $v$, one can use the commutation relations in $\ru_q$ to show that $\ru_q^- v$ is actually an $\ru_q$-submodule. Thus $\ru_q v=\ru_q^- v$, which is isomorphic to $M(|v|)$, where $|v|\in \Lambda^2$ is the degree of $v$. Now, by \Cref{lem:V-otimes-standard}, we know that $M\otimes M'$ has a composition series where the factors correspond to the standard modules alluded to above. Since standard modules are projective as $\ru_q^-$-modules, this composition series splits, and we have a direct sum decomposition of $M\otimes M'$ as above, only as an $\ru_q^-$-module. However, this implies that standard modules generated by independent highest weight vectors are pairwise disjoint. 
\end{proof}

For tensor products of simple modules, it is harder to compute composition series. We include here an example. For other examples, see \Cref{sec:semisimplification}.

\begin{example}Let $N=10$. We find the following composition series for the tensor product:
\begin{gather*}
0\leq M_1\leq M_2\leq M_3\leq M_4=\underset{\dim 7}{L(0,3)}\otimes \underset{\dim 9}{L(0,4)}\\
M_1\cong\underset{\dim 15}{L(0,7)}, \quad M_2/M_1\cong \underset{\dim 24}{L(9,7)} ,\quad M_3/M_2\cong  \underset{\dim 16}{ L(8,6)},\quad M_4/M_3\cong \underset{\dim 8}{L(7,5)}.
\end{gather*}
The simple modules are graded. Without loss of generality, their highest weight vectors are in $\mZ$-degree 0.
Their tensor product filtration is computed as follows: 
We start with the highest weight vector of top degree, namely $v_{03}\otimes v_{04}$, which has $\Lambda$-degree $(0,7)$ and consider the submodule this vector generates. The graded dimension limitations imply that this highest weight vector generates a submodule isomorphic to $L(0,7)$. Now, the quotient $\big(L(0,3)\otimes L(0,4)\big)/L(0,7)$  has a $1$-dimensional highest $\mZ$-graded dimension of $-1$ which has $\Lambda$-degree $(-1,7)$. Thus, this quotient contains a highest weight vector of degree $(-1,7)=(9,7)$ which, again, using limitations on the $\Lambda$-graded dimensions, generates a submodule of dimension $24=4\cdot 6=4(9+7-10)$ which must be isomorphic to $L(9,7)$. We continue identifying the highest $\mZ$-graded dimension of the subsquent quotient by $L(9,7)$ to find another unique highest weight vector (up to scalar) in this quotient. This procedure can be illustrated by the following diagram.
\begin{footnotesize}
    \begin{align*}
\vcenter{\hbox{
\begin{tikzpicture}
\tikzstyle{ann} = [circle,draw,fill,scale=0.75]
\node [ann,label=above:\scriptsize${(0,4)}$] (1){};
\node [ann,label=above:\scriptsize${(-1,4)}$] (2)[right of=1,node distance=2cm]{};
\node [ann,label=above:\scriptsize${(-1,3)}$] (3)[right of=2,node distance=2cm]{};
\node [ann,label=above:\scriptsize${(-2,3)}$] (4)[right of=3,node distance=2cm]{};
\node [ann,label=above:\scriptsize${(-2,2)}$] (5)[right of=4,node distance=2cm]{};
\node [ann,label=above:\scriptsize${(-3,2)}$] (6)[right of=5,node distance=2cm]{};
\node [ann,label=above:\scriptsize${(-3,1)}$] (7)[right of=6,node distance=2cm]{};
\node [ann,label=above:\scriptsize${(-4,1)}$] (8)[right of=7,node distance=2cm]{};
\node [ann,label=above:\scriptsize${(-4,0)}$] (9)[right of=8,node distance=2cm]{};
\begin{scope}[very thick,decoration={markings,mark=at position 0.5 with {\arrow{>}}}]
\draw[postaction={decorate}] (1.east) to node[midway,above] {$1$} (2.west);
\draw[postaction={decorate}] (2.east) to node[midway,above] {$2$} (3.west);
\draw[postaction={decorate}] (3.east) to node[midway,above] {$1$} (4.west);
\draw[postaction={decorate}] (4.east) to node[midway,above] {$2$} (5.west);
\draw[postaction={decorate}] (5.east) to node[midway,above] {$1$} (6.west);
\draw[postaction={decorate}] (6.east) to node[midway,above] {$2$} (7.west);
\draw[postaction={decorate}] (7.east) to node[midway,above] {$1$} (8.west);
\draw[postaction={decorate}] (8.east) to node[midway,above] {$2$} (9.west);
\end{scope}
\node [ann,label=above:\scriptsize${(0,7)}$] (11)[below of=1,node distance=1.5cm]{};
\node [ann,label=left:\scriptsize${(0,3)}$] (10)[left of=11,node distance=2cm]{};
\node [ann,label=left:\scriptsize${(-1,3)}$] (20)[below of=10,node distance=1.5cm]{};
\node [ann,label=left:\scriptsize${(-1,2)}$] (30)[below of=20,node distance=1.5cm]{};
\node [ann,label=left:\scriptsize${(-2,2)}$] (40)[below of=30,node distance=1.5cm]{};
\node [ann,label=left:\scriptsize${(-2,1)}$] (50)[below of=40,node distance=1.5cm]{};
\node [ann,label=left:\scriptsize${(-3,1)}$] (60)[below of=50,node distance=1.5cm]{};
\node [ann,label=left:\scriptsize${(-3,0)}$] (70)[below of=60,node distance=1.5cm]{};
\begin{scope}[very thick,decoration={markings,mark=at position 0.5 with {\arrow{>}}}]
\draw[postaction={decorate}] (10.south) to node[midway,left] {$1$} (20.north);
\draw[postaction={decorate}] (20.south) to node[midway,left] {$2$} (30.north);
\draw[postaction={decorate}] (30.south) to node[midway,left] {$1$} (40.north);
\draw[postaction={decorate}] (40.south) to node[midway,left] {$2$} (50.north);
\draw[postaction={decorate}] (50.south) to node[midway,left] {$1$} (60.north);
\draw[postaction={decorate}] (60.south) to node[midway,left] {$2$} (70.north);
\end{scope}
\node [ann,label=above:\scriptsize${}$] (12)[right of=11,node distance=2cm]{};
\node [ann,label=above:\scriptsize${}$] (13)[right of=12,node distance=2cm]{};
\node [ann,label=above:\scriptsize${}$] (14)[right of=13,node distance=2cm]{};
\node [ann,label=above:\scriptsize${}$] (15)[right of=14,node distance=2cm]{};
\node [ann,label=above:\scriptsize${}$] (16)[right of=15,node distance=2cm]{};
\node [ann,label=above:\scriptsize${}$] (17)[right of=16,node distance=2cm]{};
\node [ann,label=above:\scriptsize${}$] (18)[right of=17,node distance=2cm]{};
\node [ann,label=above:\scriptsize${}$] (19)[right of=18,node distance=2cm]{};
\node [ann,label=above:\scriptsize${(-1,7)}$] (21)[below of=11,node distance=1.5cm]{};
\node [ann,label=above:\scriptsize${}$] (22)[right of=21,node distance=2cm]{};
\node [ann,label=above:\scriptsize${}$] (23)[right of=22,node distance=2cm]{};
\node [ann,label=above:\scriptsize${}$] (24)[right of=23,node distance=2cm]{};
\node [ann,label=above:\scriptsize${}$] (25)[right of=24,node distance=2cm]{};
\node [ann,label=above:\scriptsize${}$] (26)[right of=25,node distance=2cm]{};
\node [ann,label=above:\scriptsize${}$] (27)[right of=26,node distance=2cm]{};
\node [ann,label=above:\scriptsize${}$] (28)[right of=27,node distance=2cm]{};
\node [ann,label=above:\scriptsize${}$] (29)[right of=28,node distance=2cm]{};
\node [ann,label=above:\scriptsize${}$] (31)[below of=21,node distance=1.5cm]{};
\node [ann,label=above:\scriptsize${(-2,6)}$] (32)[right of=31,node distance=2cm]{};
\node [ann,label=above:\scriptsize${}$] (33)[right of=32,node distance=2cm]{};
\node [ann,label=above:\scriptsize${}$] (34)[right of=33,node distance=2cm]{};
\node [ann,label=above:\scriptsize${}$] (35)[right of=34,node distance=2cm]{};
\node [ann,label=above:\scriptsize${}$] (36)[right of=35,node distance=2cm]{};
\node [ann,label=above:\scriptsize${}$] (37)[right of=36,node distance=2cm]{};
\node [ann,label=above:\scriptsize${}$] (38)[right of=37,node distance=2cm]{};
\node [ann,label=above:\scriptsize${}$] (39)[right of=38,node distance=2cm]{};
\node [ann,label=above:\scriptsize${}$] (41)[below of=31,node distance=1.5cm]{};
\node [ann,label=above:\scriptsize${}$] (42)[right of=41,node distance=2cm]{};
\node [ann,label=above:\scriptsize${(-3.5)}$] (43)[right of=42,node distance=2cm]{};
\node [ann,label=above:\scriptsize${}$] (44)[right of=43,node distance=2cm]{};
\node [ann,label=above:\scriptsize${}$] (45)[right of=44,node distance=2cm]{};
\node [ann,label=above:\scriptsize${}$] (46)[right of=45,node distance=2cm]{};
\node [ann,label=above:\scriptsize${}$] (47)[right of=46,node distance=2cm]{};
\node [ann,label=above:\scriptsize${}$] (48)[right of=47,node distance=2cm]{};
\node [ann,label=above:\scriptsize${}$] (49)[right of=48,node distance=2cm]{};
\node [ann,label=above:\scriptsize${}$] (51)[below of=41,node distance=1.5cm]{};
\node [ann,label=above:\scriptsize${}$] (52)[right of=51,node distance=2cm]{};
\node [ann,label=above:\scriptsize${}$] (53)[right of=52,node distance=2cm]{};
\node [ann,label=above:\scriptsize${}$] (54)[right of=53,node distance=2cm]{};
\node [ann,label=above:\scriptsize${}$] (55)[right of=54,node distance=2cm]{};
\node [ann,label=above:\scriptsize${}$] (56)[right of=55,node distance=2cm]{};
\node [ann,label=above:\scriptsize${}$] (57)[right of=56,node distance=2cm]{};
\node [ann,label=above:\scriptsize${}$] (58)[right of=57,node distance=2cm]{};
\node [ann,label=above:\scriptsize${}$] (59)[right of=58,node distance=2cm]{};
\node [ann,label=above:\scriptsize${}$] (61)[below of=51,node distance=1.5cm]{};
\node [ann,label=above:\scriptsize${}$] (62)[right of=61,node distance=2cm]{};
\node [ann,label=above:\scriptsize${}$] (63)[right of=62,node distance=2cm]{};
\node [ann,label=above:\scriptsize${}$] (64)[right of=63,node distance=2cm]{};
\node [ann,label=above:\scriptsize${}$] (65)[right of=64,node distance=2cm]{};
\node [ann,label=above:\scriptsize${}$] (66)[right of=65,node distance=2cm]{};
\node [ann,label=above:\scriptsize${}$] (67)[right of=66,node distance=2cm]{};
\node [ann,label=above:\scriptsize${}$] (68)[right of=67,node distance=2cm]{};
\node [ann,label=above:\scriptsize${}$] (69)[right of=68,node distance=2cm]{};
\node [ann,label=above:\scriptsize${}$] (71)[below of=61,node distance=1.5cm]{};
\node [ann,label=above:\scriptsize${}$] (72)[right of=71,node distance=2cm]{};
\node [ann,label=above:\scriptsize${}$] (73)[right of=72,node distance=2cm]{};
\node [ann,label=above:\scriptsize${}$] (74)[right of=73,node distance=2cm]{};
\node [ann,label=above:\scriptsize${}$] (75)[right of=74,node distance=2cm]{};
\node [ann,label=above:\scriptsize${}$] (76)[right of=75,node distance=2cm]{};
\node [ann,label=above:\scriptsize${}$] (77)[right of=76,node distance=2cm]{};
\node [ann,label=above:\scriptsize${}$] (78)[right of=77,node distance=2cm]{};
\node [ann,label=above:\scriptsize${}$] (79)[right of=78,node distance=2cm]{};
\begin{scope}[very thick,decoration={markings,mark=at position 0.5 with {\arrow{>}}}]
\draw[postaction={decorate}] (11.east) to node[midway,above] {$1$} (12.west);
\draw[postaction={decorate}] (12.east) to node[midway,above] {$2$} (13.west);
\draw[postaction={decorate}] (13.east) to node[midway,above] {$1$} (14.west);
\draw[postaction={decorate}] (14.east) to node[midway,above] {$2$} (15.west);
\draw[postaction={decorate}] (15.east) to node[midway,above] {$1$} (16.west);
\draw[postaction={decorate}] (16.east) to node[midway,above] {$2$} (17.west);
\draw[postaction={decorate}] (17.east) to node[midway,above] {$1$} (18.west);
\draw[postaction={decorate}] (18.east) to node[midway,above] {$2$} (19.west);
\draw[postaction={decorate}] (19.south) to node[midway,left] {$1$} (29.north);
\draw[postaction={decorate}] (29.south) to node[midway,left] {$2$} (39.north);
\draw[postaction={decorate}] (39.south) to node[midway,left] {$1$} (49.north);
\draw[postaction={decorate}] (49.south) to node[midway,left] {$2$} (59.north);
\draw[postaction={decorate}] (59.south) to node[midway,left] {$1$} (69.north);
\draw[postaction={decorate}] (69.south) to node[midway,left] {$2$} (79.north);
\draw[postaction={decorate}] (21.east) to node[midway,above] {$1$} (22.west);
\draw[postaction={decorate}] (22.east) to node[midway,above] {$2$} (23.west);
\draw[postaction={decorate}] (23.east) to node[midway,above] {$1$} (24.west);
\draw[postaction={decorate}] (24.east) to node[midway,above] {$2$} (25.west);
\draw[postaction={decorate}] (25.east) to node[midway,above] {$1$} (26.west);
\draw[postaction={decorate}] (26.east) to node[midway,above] {$2$} (27.west);
\draw[postaction={decorate}] (27.east) to node[midway,above] {$1$} (28.west);
\draw[postaction={decorate}] (28.south) to node[midway,left] {$2$} (38.north);
\draw[postaction={decorate}] (38.south) to node[midway,left] {$1$} (48.north);
\draw[postaction={decorate}] (48.south) to node[midway,left] {$2$} (58.north);
\draw[postaction={decorate}] (58.south) to node[midway,left] {$1$} (68.north);
\draw[postaction={decorate}] (68.south) to node[midway,left] {$2$} (78.north);
\draw[postaction={decorate}] (21.south) to node[midway,left] {$2$} (31.north);
\draw[postaction={decorate}] (31.south) to node[midway,left] {$1$} (41.north);
\draw[postaction={decorate}] (41.south) to node[midway,left] {$2$} (51.north);
\draw[postaction={decorate}] (51.south) to node[midway,left] {$1$} (61.north);
\draw[postaction={decorate}] (61.south) to node[midway,left] {$2$} (71.north);
\draw[postaction={decorate}] (71.east) to node[midway,above] {$1$} (72.west);
\draw[postaction={decorate}] (72.east) to node[midway,above] {$2$} (73.west);
\draw[postaction={decorate}] (73.east) to node[midway,above] {$1$} (74.west);
\draw[postaction={decorate}] (74.east) to node[midway,above] {$2$} (75.west);
\draw[postaction={decorate}] (75.east) to node[midway,above] {$1$} (76.west);
\draw[postaction={decorate}] (76.east) to node[midway,above] {$2$} (77.west);
\draw[postaction={decorate}] (77.east) to node[midway,above] {$1$} (78.west);
\draw[postaction={decorate}] (32.east) to node[midway,above] {$2$} (33.west);
\draw[postaction={decorate}] (33.east) to node[midway,above] {$1$} (34.west);
\draw[postaction={decorate}] (34.east) to node[midway,above] {$2$} (35.west);
\draw[postaction={decorate}] (35.east) to node[midway,above] {$1$} (36.west);
\draw[postaction={decorate}] (36.east) to node[midway,above] {$2$} (37.west);
\draw[postaction={decorate}] (37.south) to node[midway,left] {$1$} (47.north);
\draw[postaction={decorate}] (47.south) to node[midway,left] {$2$} (57.north);
\draw[postaction={decorate}] (57.south) to node[midway,left] {$1$} (67.north);
\draw[postaction={decorate}] (32.south) to node[midway,left] {$1$} (42.north);
\draw[postaction={decorate}] (42.south) to node[midway,left] {$2$} (52.north);
\draw[postaction={decorate}] (52.south) to node[midway,left] {$1$} (62.north);
\draw[postaction={decorate}] (62.east) to node[midway,above] {$2$} (63.west);
\draw[postaction={decorate}] (63.east) to node[midway,above] {$1$} (64.west);
\draw[postaction={decorate}] (64.east) to node[midway,above] {$2$} (65.west);
\draw[postaction={decorate}] (65.east) to node[midway,above] {$1$} (66.west);
\draw[postaction={decorate}] (66.east) to node[midway,above] {$2$} (67.west);
\draw[postaction={decorate}] (43.east) to node[midway,above] {$1$} (44.west);
\draw[postaction={decorate}] (44.east) to node[midway,above] {$2$} (45.west);
\draw[postaction={decorate}] (45.east) to node[midway,above] {$1$} (46.west);
\draw[postaction={decorate}] (46.south) to node[midway,left] {$2$} (56.north);
\draw[postaction={decorate}] (43.south) to node[midway,left] {$2$} (53.north);
\draw[postaction={decorate}] (53.east) to node[midway,above] {$1$} (54.west);
\draw[postaction={decorate}] (54.east) to node[midway,above] {$2$} (55.west);
\draw[postaction={decorate}] (55.east) to node[midway,above] {$1$} (56.west);
\end{scope}
\end{tikzpicture}
}}
    \end{align*}\end{footnotesize}
Note that, in general, the method used in this example does not reveal if the tensor product decomposes as a direct sum. In this example, we obtain a direct sum decomposition.
$$L(0,3)\otimes L(0,4)\cong L(7,5)\oplus L(8,6)\oplus L(9,7)\oplus L(0,7).$$
\end{example}

\subsubsection{Some Grothendieck ring calculations}\label[section]{sec:Koexmpl}

\Cref{prop:K0comp} enables us to compute the graded Grothendieck ring 
$$R_q^\mZ:=K_0(\lmod{\ru_q(\fr{sl}_{2,\mI})}^\mZ)$$
of $\ru_q(\fr{sl}_{2,\mI})$ in terms of a subring of the ring $\mZ \Lambda [t,t^{-1}]$ of Laurent polynomials with coefficients in the group ring $\mZ\Lambda$. In this section, $\Lambda=\mZ_N\times \mZ_N$ and we will denote its elements multiplicatively  by $g_1^ig_2^j$, $i,j=0,\ldots, N$. This ring is the Grothendieck ring of the category of graded $\Lambda$-comodules. 

In order to carry out such computations, we require the graded dimensions of the simple modules. These are derived from \Cref{thm:rank2}. We normalize the simple modules $L(i,j)$ so that their highest weight vector is concentrated in $\mZ$-degree zero. Recall that for any graded modules $M=(M_i)_{i\in \mZ}$ the grading shift is defined by 
$M[i]_j=M_{i-j}$ and hence for the symbol in $R_q$ we have 
$$[M[i]]=[M]t^{i}.$$

\begin{corollary}
The simple $\ru_q$-modules $L(i,j)$ have the following symbols $l_{i,j}:=[L(i,j)]$ in the graded Grothendieck ring $R_q^\mZ$, for  $i,j=0,\ldots, N-1$ such that $i\neq j$:
\begin{gather*}
    l_{i,0}=\sum_{a=0}^i g_1^{i-a}g_2^{-a}t^{2a} + \sum_{a=0}^{i-1}g_1^{i-a}g_2^{-a-1} t^{2a+1},\\
    l_{0,i}=\sum_{a=0}^i g_1^{-a}g_2^{i-a}t^{2a} + \sum_{a=0}^{i-1}g_1^{-a-1}g_2^{i-a} t^{2a+1},\\
    l_{i,j}=g_1^{i}g_2^{j}+\sum_{a=1}^{k-1} \big(2g_1^{i-a}g_2^{j-a}t^{2a}+ (g_1^{i-a+1}g_2^{j-a} + g_1^{i-a}g_2^{j-a+1})t^{2a-1}\big) + g_1^{-j}g_2^{-i}t^{2k}.
\end{gather*}
Here, $k$ is the unique representative $1\leq k \leq N$ of $i+j$ modulo $N$.
\end{corollary}

\Cref{prop:K0comp} implies that any product of the $l_{i,j}$ is uniquely a sum of the shifts of the same polynomials. This reduces finding the composition factors of tensor products to decomposing products of polynomials in $\mZ\Lambda[t,t^{-1}]$ as $\mZ_+$-linear combinations. Further, such computations in $R_q^\mZ$ help us to compute the fusion rules of the Grothendieck ring $R_q=K_0(\lmod{\ru_q(\fr{sl}_{2,\mI})})$ by simply setting $t\mapsto 1$. This uses that every simple module has a unique $\mZ$-grading up to shift \cite{Vay}*{Theorem 5.1}.

The algebra $R_q$ is an $N^2$-dimensional commutative fusion ring in the sense of \cite{EGNO}*{Section~3.1} with basis $\{l_{i,j}\}_{i,j=1,\ldots, N}$. By \cite{EGNO}*{Example 6.1.9.}, its Frobenius--Perron dimension equals $\dim_\Bbbk \ru_q(\fr{sl}_{2,\mI})=4N^4$. Its involution is given by the map
$l_{i,j}\mapsto l_{j,i}.$

\begin{example}
The easiest example is the the $16$-dimensional fusion ring $R_{\mathtt i}$ (or, the corresponding filtered ring $R_{\mathtt i}^\mZ)$, corresponding to the case $N=4$, with $q=\mathtt i$. The following relations hold in $R_{\mathtt i}^\mZ$:
\begin{gather*}
    l_{10}^2=l_{20}+l_{23}t, \qquad l_{10}l_{20}=l_{30}+l_{33}t,\\
    l_{10}l_{30}=l_{00}+2l_{03}t^{2}+l_{02}t^{2}+l_{00}t^{8}, \qquad
    l_{20}^2=l_{00}+2l_{03}t+l_{02}t^{2}+l_{32}t^{3}+l_{00}t^{8},\\
    l_{20}l_{30}=l_{10}+l_{13}t+2l_{02}t^{3}+l_{01}t^{4}+l_{10}t^{8},\\
    l_{30}^2=l_{20}-l_{23}t-2l_{12}t^{3}-2l_{01}t^{5}-l_{00}t^{6}-l_{20}t^{8}-l_{23}t^{9}\\
l_{10}l_{01}=l_{11}+l_{00}t^{2}, \qquad l_{10}l_{02}=l_{12}+l_{01}t^{2},\\
l_{10}l_{03}=l_{13}+l_{02}t^{2},\qquad l_{20}l_{02}=l_{22}+l_{11}t^{2}+l_{00}t^{4},\\
l_{20}l_{03}=l_{23}+2 l_{12}t^{2}+l_{01}t^{4}+l_{23}t^{8}, \qquad l_{30}l_{03}=l_{33}+l_{22}t^{2}+2 l_{11}t^{4}+l_{00}t^{6}+l_{33}t^{8}.
\end{gather*}
Setting $t\mapsto 1$ gives the fusion rules in $R_{\mathtt i}$. Further calculations reveal, e.g., that $R_{\mathtt i}$ is generated by the $l_{i,j}$ with either $i=0$ or $j=0$:
\begin{align*}
l_{11}&=l_{10}l_{01}-l_{00}, &
l_{12}&=l_{10}l_{02}-l_{01},\\
l_{13}&=l_{10}l_{03}-l_{02}, &
l_{22}&=l_{20}l_{02}-l_{11}-l_{00},\\
l_{23}&=l_{10}^2-l_{20},&
l_{31}&=l_{02}l_{03}-l_{01}-2l_{20}-l_{10}-l_{01},\\
l_{32}&=l_{20}l_{20}-l_{00}-2l_{03}-l_{02}-l_{00}, &
l_{33}&=l_{10}l_{20}-l_{30}.
\end{align*}
\end{example}

\subsection{Some remarks on the semisimplification}\label[section]{sec:semisimplification}

By \cite{EO}*{Theorem 2.6, Remark 2.9},  for a finite-dimensional pivotal Hopf algebra $H$, the quotient of $\cC=\lmod{H}$ by the ideal of negligible morphisms $\cN(\cC)$ yields a semisimple tensor category $\ov{\cC}:=\cC/\cN(\cC)$ called the \emph{semisimplification} of $\cC$. The semisimplification inherits several properties from $\cC$ such as being monoidal, braided and ribbon. The simple objects of $\ov{\cC}$ are given by those indecomposable objects of $\cC$ which have non-zero quantum dimension. 
This raises the following question.

\begin{question}\label[question]{ques:semisimple}
What are the simple objects in $\ov{\lmod{\ru_q(\fr{sl}_{r,\mI})}}$? What are the fusion rules for their tensor products?
\end{question}

We note that already in the case of $\ru_{\mathtt i}(\fr{sl}_{2,\mI})$, with $N=4$ and $q={\mathtt i}$, this is a non-trivial question as illustrated by the examples below. Recall that for any even $N$ and $q$ a primitive root of unity, a full list of simple $\ru_q(\fr{sl}_{2,\mI})$-modules with non-zero quantum dimension is given by $L(i,0)$, $i=0,1, \dots, N-1$, and their duals $L(0,i)$. 

\begin{lemma}
    For any even $N$, $q$ a primitive $N$-th root of unity, the object $L(1,0)$ becomes invertible in the semisimplification of $\lmod{\ru_q(\fr{sl}_{2,\mI})}$ with inverse $L(0,1)$.
\end{lemma}
\begin{proof}
The result follows from a decomposition of $\ru_q(\fr{sl}_{2,\mI})$-modules
    $$L(1,0)\otimes L(0,1)\cong L(1,1)\oplus L(0,0).$$
Indeed, one checks that in $L(1,0)\otimes L(0,1)$ the vector
$$v=y_1y_2v_{10}\otimes v_{01}-qv_{10}\otimes y_2y_1v_{01}-y_2v_{10}\otimes y_1v_{01}$$
is both a highest and lowest weight vector. Thus, as $\dim_\Bbbk L(1,1)=8$ and $v_{10}\otimes v_{01}$ defines an $8$-dimensional submodule which cannot contain $v$, the direct sum decomposition follows.
\end{proof}    
In particular, all modules $L_i:=L(1,0)^{\otimes n}$, for $n\in \mZ$, are non-zero simple objects in $\ov{\lmod{\ru_q(\fr{sl}_{2,\mI})}}$ of quantum dimension $(-1)^i$. These objects admit isomorphisms
$$L_i\otimes L_j\cong L_{i+j}.$$

\begin{lemma}
    There is an isomorphism $L_2\cong L(2,0)$ in the semisimplification.
\end{lemma}
\begin{proof}
    This follows from a direct sum decomposition
    $$L(1,0)\otimes L(1,0)\cong L(2,0)\oplus L(2,N-1).$$
    Such decomposition follows from the existence of a highest weight vector of weight $(2,N-1)$, which necessarily generates a $4$-dimensional direct summand $L(2,N-1)$ due to a dimension argument.
\end{proof}

Consider the braided category $\svect^{\Bbbk C_m}$ of finite-dimensional $\Bbbk C_m$-comodules, for $C_m$ a cyclic group ($m=\infty$ denotes the infinite case), but with braiding given by 
$\Psi_{V,W}(v,w)=(-1)^{\deg v\cdot \deg w}w\otimes v$ for $v\in V, w\in W$ homogeneous elements. 
This category is symmetric monoidal. 

From the above discussion, we know that the tensor subcategory generated by $L_1$ in $\ov{\lmod{\ru_q(\fr{sl}_{2,\mI})}}$ is  braided equivalent $\svect^{\Bbbk C_m}$ for some $m$, possibly infinite. 
Indeed, after passing to semisimplification, the braiding $\Psi_{L_1,L_1}$ is necessarily a scalar multiple of the identity on $L_2$. It follows from \Cref{skein-L10} that the braiding is given either by $(-1)\ide_{L_1\otimes L_1}$ or by $q^{-1}\ide_{L_1\otimes L_1}$. Computing the Hopf link invariant, one sees that the braiding is given by 
\begin{equation}\Psi_{L_1,L_1}=(-1)\ide_{L_1\otimes L_1}.\label{eq:L1braiding}
\end{equation}
This is the same as the braiding of the ($1$-dimensional) simple object in $\svect^{\Bbbk C_m}$ concentrated in degree $1$.
It remains unclear whether $m=\infty$ or a certain even natural number $m\geq 4$. Below, we demonstrate that $m\geq 6$ in the case $q=\mathtt i$ is a $4$-th root of unity.

\begin{example}
We decompose some tensor products into indecomposable direct summands of $\ru_{\mathtt i}(\fr{sl}_{2,\mI})$-modules. Here, for $L(i,0)$, we fix a homogeneous basis $\{v_j^i\}_{j=0,\ldots, 2i}$, where 
$$v_j^i=\begin{cases} (y_2y_1)^ay_2v_0^i,&\text{if $j=2a+1$},\\
(y_1y_2)^av_0^i,&\text{if $j=2a$.}
\end{cases}
$$
for $i=0,1,2,3$ and a highest weight vector $v_0^i$ of $L(i,0)$.
\begin{enumerate}
    \item The tensor product $L(1,0)\otimes L(1,0)$ decomposes as $$L(1,0)\otimes L(1,0)\cong L(2,0)\oplus L(2,3).$$ 
    The direct summand $L(2,0)$ is generated by the highest weight vector $v_0^1\otimes v_0^1$ and the direct summand $L(2,3)$ is generated by the highest weight vector $v_1^1\otimes v_0^1-v_0^1\otimes v_1^1$.
    \item The tensor product $L(1,0)\otimes L(2,0)$ decomposes as $$L(1,0)\otimes L(2,0)\cong L(3,3)\oplus L(3,0).$$
     The direct summand $L(3,0)$ is generated by the highest weight vector $v_0^1\otimes v_0^2$ and the direct summand $L(3,3)$ is generated by the highest weight vector $2v_1^1\otimes v_0^2+(i-1)v_0^1\otimes v_1^2$.
    \item The tensor product $L(1,0)\otimes L(3,0)$ is indecomposable. The highest weight vector $v_0^1\otimes v_0^3$ generates an $8$-dimensional submodule $V_1$ which is a non-split extension
    $$0\to L(0,3)\to V_1\to L(0,0)\to 0.$$
    Its socle $L(0,3)$ is generated by the highest weight vector $v_1^1\otimes v_0^3+i v_0^1\otimes v_1^3.$
    As $\dim_q(L(1,0)\otimes L(3,0))=\dim_q L(1,0)\cdot \dim_q L(3,0)=(-1)^{1+3}=1$, $V_{1}$ gives a simple object in the semisimplification.
    
    We note that $L(1,0)\otimes L(3,0)$ is an indecomposable module \emph{not} generated by highest weight vectors.
    \item The tensor product $L(2,0)\otimes L(2,0)$ decomposes as a direct sum
    $$L(2,0)\otimes L(2,0)\cong L(3,2)\oplus V_2,$$
    where $L(3,2)$ is generated by the highest weight vector
    $$2v_0^2\otimes v_3^2+(1-i) v_1^2\otimes v_2^2 -(1-i)v_2^2\otimes v_1^2 -2iv_2^2\otimes v_0^2.$$
    The indecomposable direct summand $V_2$ has the same composition series as $L(1,0)\otimes L(3,0)$.
    The vector $w=v_0^2\otimes v_1^2$ satisfies that $x_2w=2v_0^2\otimes v_0^2$.
    \item The tensor product $L(2,0)\otimes L(3,0)$ is again indecomposable. The two highest weight vectors 
    $$w_1=v_0^2\otimes v_0^3, \qquad w_2=2v_0^2\otimes v_2^3-(1+i)v_2^2\otimes v_0^3$$
    generate submodules $V_1$ and $V_2$. The submodule $V_1$ is a non-split extension
    $$0\to L(0,2)\to V_1\to L(1,0)$$
    while $V_2\cong L(1,3)=M(1,3)$ is a $16$-dimensional simple. 
    
    The quotient by $V_1\oplus V_2$ is generated by the vector 
    $$w_3=v_0^2\otimes v_3^3.$$
    Thus, the quotient $Q=(L(2,0)\otimes L(3,0))/(V_1\oplus V_2)$ is a non-split extension
    $$0\to L(1,0)\oplus L(0,1) \to Q\to L(0,2)\to 0.$$
    In $L(2,0)\otimes L(3,0)$ we have that 
    $$x_2\cdot w_3=2v_0^2\otimes v_2^3$$
    which is a sum of vectors from $V_2$ and $V_3$. Thus, neither $V_2$ nor $V_3$ are direct summands of the tensor product. Thus, $L_5$ is not isomorphic to $L_i, 0\leq i\leq 4$ and gives a new simple object in the tensor subcategory generated by $L_1$ in the semisimplification.
\end{enumerate}
\end{example}
The following questions about the semisimplified category remain unanswered.

\begin{question}
Does $\ov{\lmod{\ru_q(\fr{sl}_{2,\mI})}}$ contain infinitely many isomorphism classes of simple objects?
\end{question}

\begin{question}
Do $L(1,0)$ and $L(0,1)$ form a set of tensor generators for the semisimple tensor category $\ov{\lmod{\ru_q(\fr{sl}_{2,\mI})}}$?
\end{question}

\section{Application to link invariants}
\label[section]{sec:linkinvariants}

As an application of the non-semisimple ribbon categories constructed in this paper, we compute knot invariants obtained from low-dimensional simple modules over the rank-two super quantum groups $\ru_q(\fr{sl}_{2,\mI})$ associated to a primitive $N$-th root of unity $q$, where $N=2n$ is even. 

The discussion in \Cref{sec:semisimplification} implies that the link invariants associated to any simple  object with non-vanishing quantum dimension, i.e., any object of the form $L(i,0)$ or $L(0,i)$, $i=0, \ldots, N-1$, are not interesting. For example, the (framed) link invariants obtained from the $3$-dimensional simple module $L(1,0)$ and its dual $L(0,1)$ compute the number of connected components of the link modulo $2$.

To obtain interesting invariant of knots, we consider the $4$-dimensional simple module $W:=L(n,n+1)$. This module has vanishing quantum dimension $\dim_q W=0$. Thus, the associated Reshektikhin--Turaev link invariants are trivial. To produce non-trivial interesting link invariants,
we apply the theory of generalized traces, see \cite{GKP}.

\subsection{An ambidextrous trace for the four-dimensional simple module}

Consider the four-dimensional simple $\ru_q(\fr{sl}_{2,\mI})$-module $W=L(n,n+1)$ generated by a highest weight vector $w_0$ of weight $(n,n+1)$. We fix a $\mZ_N^2$-homogeneous ordered basis $\Set{w_0,w_1,w_1',w_2}$ with respect to which the action is given by 
\begin{align*}
    y_1w_0&=w_1, & y_1w_1&=0, & y_1 w_1'&=2(1-q^{-1})^{-1}w_2, & y_1w_2&=0,\\ 
    y_2w_0&=w_1', & y_2w_1&=w_2, & y_2 w_1'&=0, & y_2w_2&=0,\\
    x_1w_0&=0, & x_1w_1&=(1+q)w_0, & x_1 w_1'&=0, & x_1w_2&=(1+q^{-1})w_1',\\ 
    x_2w_0&=0, & x_2w_1&=0, & x_2 w_1'&=2w_0, & x_2w_2&=(1+q)w_1.
\end{align*}
Then, the braiding on $W\otimes W$ is given, in the lexicographically ordered basis $\Set{w_i\otimes w_j}$ by the following matrix:
\begin{align*}
\resizebox{\hsize}{!}{$\left(
\begin{smallmatrix}
\left(-1\right)^{n} & 0 & 0 & 0 & 0 & 0 & 0 & 0 & 0 & 0 & 0 & 0 & 0 & 0 & 0 & 0 
\\
 0 & \left(-1\right)^{n} \left(1+\frac{1}{q}\right) & 0 & 0 & -\frac{1}{q} & 0 & 0 & 0 & 0 & 0 & 0 & 0 & 0 & 0 & 0 & 0 
\\
 0 & 0 & 2 \left(-1\right)^{n} & 0 & 0 & 0 & 0 & 0 & -1 & 0 & 0 & 0 & 0 & 0 & 0 & 0 
\\
 0 & 0 & 0 & \left(-1\right)^{n} \left(3+\frac{1}{q}\right) & 0 & 0 & -2 & 0 & 0 & -2 & 0 & 0 & \frac{\left(-1\right)^{n}}{q} & 0 & 0 & 0 
\\
 0 & 1 & 0 & 0 & 0 & 0 & 0 & 0 & 0 & 0 & 0 & 0 & 0 & 0 & 0 & 0 
\\
 0 & 0 & 0 & 0 & 0 & \frac{\left(-1\right)^{n}}{q} & 0 & 0 & 0 & 0 & 0 & 0 & 0 & 0 & 0 & 0 
\\
 0 & 0 & 0 & 1+\frac{1}{q} & 0 & 0 & 0 & 0 & 0 & \left(-1\right)^{n+1} & 0 & 0 & 0 & 0 & 0 & 0 
\\
 0 & 0 & 0 & 0 & 0 & 0 & 0 & \left(-1\right)^{n} \left(1+\frac{1}{q}\right) & 0 & 0 & 0 & 0 & 0 & -\frac{1}{q} & 0 & 0 
\\
 0 & 0 & 1 & 0 & 0 & 0 & 0 & 0 & 0 & 0 & 0 & 0 & 0 & 0 & 0 & 0 
\\
 0 & 0 & 0 & 1+\frac{1}{q} & 0 & 0 & \left(-1\right)^{n+1} & 0 & 0 & 0 & 0 & 0 & 0 & 0 & 0 & 0 
\\
 0 & 0 & 0 & 0 & 0 & 0 & 0 & 0 & 0 & 0 & \left(-1\right)^{n} & 0 & 0 & 0 & 0 & 0 
\\
 0 & 0 & 0 & 0 & 0 & 0 & 0 & 0 & 0 & 0 & 0 & 2 \left(-1\right)^{n} & 0 & 0 & -1 & 0 
\\
 0 & 0 & 0 & \left(-1\right)^{n} & 0 & 0 & 0 & 0 & 0 & 0 & 0 & 0 & 0 & 0 & 0 & 0 
\\
 0 & 0 & 0 & 0 & 0 & 0 & 0 & 1 & 0 & 0 & 0 & 0 & 0 & 0 & 0 & 0 
\\
 0 & 0 & 0 & 0 & 0 & 0 & 0 & 0 & 0 & 0 & 0 & 1 & 0 & 0 & 0 & 0 
\\
 0 & 0 & 0 & 0 & 0 & 0 & 0 & 0 & 0 & 0 & 0 & 0 & 0 & 0 & 0 & \left(-1\right)^{n} 
\end{smallmatrix}
\right) $}
\end{align*}
Next, we compute the twist with respect to the ribbon structure from \Cref{thm:modularunique}.

\begin{lemma}\label[lemma]{lem:4dim-twist}
The twist on $W=L(n,n+1)$ is given by 
$\theta_W= (-1)^n\ide_W.
$\end{lemma}

To obtain link invariants (rather than framed link invariants), we assume, for the rest of this section that $n$ is even. The following results have been obtained computationally (using MAPLE\textsuperscript{TM}).\footnote{The MAPLE calculations have been made available in the Github repository \url{https://github.com/Robert-Laugwitz/q-group-super-type-A}.}

\begin{lemma}\label[lemma]{lem:endos}
The endomorphism ring $\End_{\ru_q(\fr{sl}_{2,\mI})}(W\otimes W)$ is three dimensional over $\Bbbk$. A basis is given by $\Set{\ide, \Psi, \Psi^{-1}}$, where $\Psi=\Psi_{W,W}$.
\end{lemma}

The results of \cite{GKP} now imply the following.

\begin{corollary}
The  isomorphism 
$$\deri_W\colon \End_{\ru_q(\fr{sl}_{2,\mI})}(W)\to \Bbbk,\qquad \lambda \ide\mapsto \lambda \in \Bbbk,$$ defines an ambidextrous trace on $W$. Thus, the tensor ideal $\cI_W$ of $\lmod{\ru_q(\fr{sl}_{2,\mI})}$ generated by $W$ has a unique ambidextrous trace.
\end{corollary}
\begin{proof}
It follows from \Cref{lem:endos} that any endomorphism of $W\otimes W$ commutes with the braiding $\Psi_{W,W}$. Thus, \cite{GKP}*{Lemma 3.3.4, Corollary 3.3.3} imply the claimed statements.
\end{proof}

The following gives a skein relation for the braiding $\Psi=\Psi_{W,W}$.
\begin{lemma} \label[lemma]{lem:4dim-skein}
The minimal polynomial of $\Psi$ is given by 
\begin{equation}\label{eq:4dim-skein}\Psi^3+(2+q^{-1})\Psi^2 +(1+2q^{-1})\Psi + q^{-1}\ide = 0.
\end{equation}
\end{lemma}

Next, we can provide a composition series for the module $W\otimes W$. This sample computation shows how non-simple indecomposables can occur in tensor products of simple objects. 
\begin{lemma}
The highest weight vector $w_0\otimes w_1 - qw_1\otimes w_0$ generates a direct summand of $W\otimes W$ which is isomorphic to $L(2n-1,2)$. Its complement $M$ is an indecomposable module with extension structure
$$0\to N\to M \to L(0,1)\oplus L(0,0)\to 0,$$
with $N=\ru_q(\fr{sl}_{2,\mI})(v_0\otimes v_0)$. Further, the highest weight vector $v_{1}'\otimes v_0-v_0\otimes v_1'$ generates a submodule isomorphic to $L(2n,2n+1)=L(0,1)$ inside of $N$ with extension structure 
$$0\to L(0,1)\to N \to L(0,2)\to 0.$$
\end{lemma}

\begin{lemma}
The endomorphism ring $\End_{\ru_q(\fr{sl}_{2,\mI})}(W\otimes W)$ has an alternative basis $\ide,\Psi,e$, where 
$$e\colon W\otimes W\twoheadrightarrow L(2n-1,2)\hookrightarrow W\otimes W$$
is the idempotent projecting onto the direct summand generated by $w_0\otimes w_1 - qw_1\otimes w_0$. 
\end{lemma}
Using the above Lemma, we see that 
\begin{equation}\label{eq:Skein2}
    \Psi+\Psi^{-1} + 2\ide+ (q-1)e=0. 
\end{equation}

\subsection{An invariant of links associated to the four-dimensional simple module}

To conclude this section, we include some knot invariants obtained from the ambidextrous trace on $W$. These invariants are constructed as in \cite{GPT} by cutting a diagram of a link $\cL$ at one strand thus producing a tangle diagram $r_\cL$ with one incoming and one outgoing strand. We label the diagram by the objects $W$ and $W^*$ coherently so that the cut strand was labelled by $W$. Thus, we can use the Reshetikhin--Tureav functor $F_W$ sending the tensor generator $(+)$ of the ribbon category of (framed) oriented tangles to $W$ and its dual to $W^*$ (cf. \cite{RT90}). Under this functor $F_W$, the ribbon diagram $r_\cL$ is evaluated to an endomorphism $F_W(r_\cL)$ of $W$. The generalized trace $\deri_W$ produces an invariant of framed links via
\begin{equation}
    \rI_W(\cL):=\deri_W(F_W(r_\cL))\in \Bbbk.
\end{equation}
This invariant is well-defined by \cite{GPT}*{Theorem~3}, i.e., independent of choice of cutting of the link diagram. The invariant $\rI_W(\cL)$ is indeed an invariant of oriented links (no framing required) since
$\theta_W=\ide_W$, see \Cref{lem:4dim-twist}.

\begin{remark}\label[remark]{rem:Links-Gould}
    This partial skein relation from \Cref{lem:4dim-skein} exactly recovers the skein relation of the Links--Gould invariant, a two-variable polynomial in $(p^2,q^2)$, obtained from the quantum superalgebra $U_q(\mathfrak{gl}(2|1))$ in \cites{GLZ,DLK} when one substitutes $(-1,q^{-1})$ for their parameters $(p^2,q^2)$. Comparing the knot invariants for $2_1^2,3_1,4_1$ with this substitution gives the same results, up to taking mirror image, as in \Cref{tab:invariants}.
\end{remark}

We expect the invariant $\rI_W$ to recover a specialization of the Links--Gould polynomial \Cref{rem:Links-Gould}.
\begin{conjecture}
The invariant $\rI_W$ is obtained from the Links--Gould invariant by setting $p^2=-1$ and $q^2=q^{-1}$. 
\end{conjecture}
The two invariants agree in all examples from the Rolfson table computed here. Further, the braiding matrices, comparing to $\sigma^{-1}$ in \cite{DLK}, are conjugate matrices.

\begin{example}
Let $\cL=\cT_{a,b}$ denote the $(a,b)$-torus knot on $a$ strands braided $b$ times with all components oriented in counter-clockwise direction. We label $\cT_{a,b}$ in such a way that the left-hand strands are all labelled by copies of $W$. Consequently, the right-hand strands are labelled by $W^*$.
For example, $\cT_{2,-1}$ is just the unknot, using $\theta_W=\ide$, $\cT_{2,-2}$ is the Hopf link, and $\cT_{2,-3}$ is a left-handed trefoil knot as displayed in \Cref{fig:trefoil} together with a choice of ribbon tangle obtained by cutting one strand. \Cref{tab:invariants} shows the values of the invariant $\rI_W$ for certain torus knots. We note that $\rI_{W}(\cT_{a,-b})$ can be obtained from $\rI_{W}(\cT_{a,b})$ by replacing $q$ with $q^{-1}$. Thus, the invariant $\rI_W$ distinguishes right-handed and left-handed versions of the torus knots.
\end{example}

\begin{figure}[htbp]
    \centering
\begingroup%
  \makeatletter%
  \providecommand\color[2][]{%
    \errmessage{(Inkscape) Color is used for the text in Inkscape, but the package 'color.sty' is not loaded}%
    \renewcommand\color[2][]{}%
  }%
  \providecommand\transparent[1]{%
    \errmessage{(Inkscape) Transparency is used (non-zero) for the text in Inkscape, but the package 'transparent.sty' is not loaded}%
    \renewcommand\transparent[1]{}%
  }%
  \providecommand\rotatebox[2]{#2}%
  \newcommand*\fsize{\dimexpr\f@size pt\relax}%
  \newcommand*\lineheight[1]{\fontsize{\fsize}{#1\fsize}\selectfont}%
  \ifx\svgwidth\undefined%
    \setlength{\unitlength}{150.02383441bp}%
    \ifx\svgscale\undefined%
      \relax%
    \else%
      \setlength{\unitlength}{\unitlength * \real{\svgscale}}%
    \fi%
  \else%
    \setlength{\unitlength}{\svgwidth}%
  \fi%
  \global\let\svgwidth\undefined%
  \global\let\svgscale\undefined%
  \makeatother%
  \begin{picture}(1,0.42411209)%
    \lineheight{1}%
    \setlength\tabcolsep{0pt}%
    \put(0,0){\includegraphics[width=\unitlength,page=1]{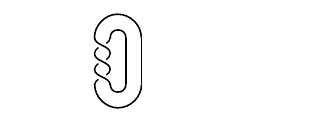}}%
    \put(-0.00221319,0.20393694){\makebox(0,0)[lt]{\lineheight{1.25}\smash{\begin{tabular}[t]{l}$\cT_{2,-3}=$\end{tabular}}}}%
    \put(0.47776147,0.20388677){\makebox(0,0)[lt]{\lineheight{1.25}\smash{\begin{tabular}[t]{l}$,$\end{tabular}}}}%
    \put(0,0){\includegraphics[width=\unitlength,page=2]{trefoil.pdf}}%
    \put(0.5525202,0.20438353){\makebox(0,0)[lt]{\lineheight{1.25}\smash{\begin{tabular}[t]{l}$r_{\cT_{2,-3}}=$\end{tabular}}}}%
    \put(0,0){\includegraphics[width=\unitlength,page=3]{trefoil.pdf}}%
    \put(0.87269888,0.4038549){\makebox(0,0)[lt]{\lineheight{1.25}\smash{\begin{tabular}[t]{l}$W$\end{tabular}}}}%
    \put(0.87269888,0.00391865){\makebox(0,0)[lt]{\lineheight{1.25}\smash{\begin{tabular}[t]{l}$W$\end{tabular}}}}%
    \put(0.25279729,0.35386294){\makebox(0,0)[lt]{\lineheight{1.25}\smash{\begin{tabular}[t]{l}$W$\end{tabular}}}}%
    \put(0.45276549,0.35386294){\makebox(0,0)[lt]{\lineheight{1.25}\smash{\begin{tabular}[t]{l}$W^*$\end{tabular}}}}%
  \end{picture}%
\endgroup%

    \caption{The trefoil knot $\cT_{2,-3}$ and a choice of tangle $r_{\cT_{2,-3}}$ such that $\rI_W(\cT_{2,-3})=\deri_W(r_{\cT_{2,-3}})$.}
    \label{fig:trefoil}
\end{figure}

\begin{example}
\Cref{tab:invariants}, in particular, contains the invariants associated to all prime knots with up to 7 crossings as listed in the Rolfsen Knot Table \cite{Rol}. As passing to the mirror image interchanges $q$ and $q^{-1}$, the table has been normalized in such a way that all invariants of knots from the Rolfsen table appear to start with a highest positive power of $q$.
\end{example}

Thus, we see that the invariant $\rI_W$ distinguishes all prime knots with up to $7$ crossings. The invariant also distinguishes these knots from their mirror images, besides those which are equivalent to their mirror images ($4_1$ and $6_3$). Note that The Links--Gould invariant distinguishes  \emph{all} prime knots with up to 10 crossings \cite{DW} but fails to distinguish some (non-mutant) knots with 12 crossings \cite{DL}.

We can give a closed formula for $\rI_{W}(\cT_{2,b})$ using the skein relation \eqref{eq:4dim-skein}.

\begin{lemma} \label[lemma]{lem:invariant-2strands}
For $b \geq 1$, we have
\begin{equation} \label{eq:invariant:2strands}
\rI_{W}(\cT_{2,-b})=(-1)^{b+1}\left( b+ 2 \sum_{i=1}^{b-1} (b-i) q^{-i}\right)
\end{equation}
\end{lemma} 

\begin{proof}
This holds for $1\leq b \leq 3$ by \Cref{tab:invariants}. Let $b\geq 4$ and assume the equality holds for all $1\leq k < b$. Then, using the skein relation \eqref{eq:4dim-skein} we have
\begin{align*}
\rI_{W}(\cT_{2,-b})
=&-\left[ (2+q^{-1}) \rI_{W}(\cT_{2,-b+1}) + (1+2q^{-1})\rI_{W}(\cT_{2,-b+2}) + q^{-1}\rI_{W}(\cT_{2,-b+3}) \right] \\
=& (-1)^{b+1} \Bigg[ (2+q^{-1})\left(b-1+2 \sum_{i=1}^{b-2} (b-1-i) q^{-i} \right) \\
&-(1+2q^{-1})\left(b-2+2 \sum_{i=1}^{b-3} (b-2-i) q^{-i} \right) + q^{-1}\left(b-3+2 \sum_{i=1}^{b-4} (b-3-i) q^{-i} \right)\Bigg] \\
=& (-1)^{b+1} \Bigg[ b +2 \sum_{i=1}^{b-2} (b-1-i) q^{-i-1}+ 4(b-2)q^{-1} -2 (b-3)q^{-1} \Bigg]\\
=&(-1)^{b+1}\left( b+ 2 \sum_{i=1}^{b-1} (b-i) q^{-i}\right).
\qedhere\end{align*}
\end{proof}

\begin{remark}
For $b \geq 1$, we have
\begin{equation} \label{eq:invariant:2strands2}
\rI_{W}(\cT_{2,b})=(-1)^{b+1}\left( b+ 2 \sum_{i=1}^{b-1} (b-i) q^{i}\right) = (-1)^{b+1} \left( \frac{2}{1-q^{-1}} [b]_q + \frac{b}{1-q} [2]_q\right).
\end{equation}
\end{remark}

\begin{question}
Is there a closed formula for $\rI_W(\cL_{a,b})$ in terms of $a,b$ for general torus knots?
\end{question}

\begin{table}[htbp]
    \centering
    \begin{tabular}{c|c|c}
    Name / Rolfsen table&Torus knot description&$\rI_W(\cL)$\\ \hline
   Unknot & $\cT_{2,1}$ &  $1$\\
   Hopf Link $2_1^2$ & $\cT_{2,2}$ & $-2q-2$\\
   (Right-handed) Trefoil $3_1$& $\cT_{2,3}$ & $2 q^{2}+4 q+3$\\
   Figure-eight knot $4_1$ & & $6 q+13+6q^{-1}$\\
   Solomon link $4_1^2$& $\cT_{2,4}$ & $-2q^{3}-4 q^{2}-6 q-4 $ \\
    Qinquefoil knot $5_1 $& $\cT_{2,5}$ & $2q^{4}+4 q^{3}+6 q^{2}+8 q+5$ \\
    Three-twist Gordian $5_2$&&$2q^{3}+14q^{2}+22q+11$\\
 Stevedore $6_1$&&$6q^{2}+26 q+35+ 14 q^{-1}$  
\\
$6_2$&&$6q^{3}+22 q^{2}+40q+ 39+14 q^{-1}$\\
$6_3$&&$ 10 q^{2}+42 q+65+42 q^{-1}+10q^{-2}$
\\
   $6_1^2$ &$\cT_{2,6}$ &$- 2q^{5}- 4q^{4}- 6q^3- 8q^2- 10q-6$\\
   $7_1$ & $\cT_{2,7}$ &$2 q^{6}+4 q^{5}+6 q^{4}+8 q^{3}+10 q^{2}+12 q+7$\\
   $7_2$&&$2 q^{4}+14 q^{3}+40 q^{2}+46 q+19$\\
   $7_3$ & &$2q^5 + 14q^4 + 32q^3 + 50q^2 + 50q + 21$\\
   $7_4$ &&$2 q^{4}+24 q^{3}+76 q^{2}+88q+35$\\
   $7_5$&&$2q^{5} + 20q^{4}+ 60q^{3} + 96q^{2}+ 82q +29$\\
   $7_6$&&$16q^{3}+76 q^{2}+134q+105 +30 q^{-1}$\\
   $7_7$&& $18q^{2}+ 96q+171    + 124q^{-1} +32q^{-2}$\\
    \end{tabular}
    \caption{Examples of links invariants obtained from the four-dimensional simple $\ru_q(\fr{sl}_{2,\mI})$-module $W=L(n,n+1)$}
    \label{tab:invariants}
\end{table}
It is known that the knots $5_1$ and $10_{132}$ have the same Jones, Alexander--Conway, and HOMFLYPT polynomials and several other invariants coincide for these two knots (however, they can be distinguished by the G2 invariant and, indeed, the Links--Gould invariant) \cite{KnotAtlas}. Here, we compute the invariant $\rI_W$ for this knot. 

\begin{example}[$10_{132}$]
Using MAGMA\textsuperscript{TM} we compute that
$$\rI_W(10_{132})=4q^2 + 4q - 3 + 10q^{-2} + 8q^{-3} + 2q^{-4}.$$
Thus, $\rI_W$ distinguishes the knots $10_{132}$ and $5_1$.
The tangle whose closure gives the knot $10_{132}$ is displayed in \Cref{fig:tangles}{\scshape (b)}.
\end{example}

To conclude, we evaluate the invariant $\rI_W$ on more complex links that the Jones polynomial cannot distinguish from unlinks. 

\begin{example}
In \cite{EKT}, the authors provide an infinite family of links $LL_2(l)$, for $l$ a positive integer and show that for even $l$ the Jones polynomial is that of the two-component unlink. The easiest links in this family are the closures of the tangles displayed in \Cref{fig:tangles}{\scshape (c)}. Again, using MAPLE\textsuperscript{TM} we compute the invariants $\rI_W$ of these links and see that  
\begin{align*}
    \rI_W(LL_2(1))=&2\left(q^7 + 2q^6 - q^5 - 3q^4 - 2q^3 + 2q + 2 - q^{-2}\right),\\
    \rI_W(LL_2(2))=&2 \left(8 q^{8}+30 q^{7}+49 q^{6}+62 q^{5}+84 q^{4}+77 q^{3}-18 q^{2}-105 q-99 -150 q^{-1}\right.\\&\left.-213 q^{-2}-113 q^{-3}+40 q^{-4}+129 q^{-5}+134 q^{-6}+70 q^{-7}+15q^{-8}\right).
\end{align*}
Thus, $\rI_W$ distinguishes $LL_2(1)$ and $LL_2(2)$ from the two-component unlink.
\end{example}

\begin{figure}[htbp]
    \centering
    \begin{subfigure}[htb]{0.3\textwidth}
\begingroup%
  \makeatletter%
  \providecommand\color[2][]{%
    \errmessage{(Inkscape) Color is used for the text in Inkscape, but the package 'color.sty' is not loaded}%
    \renewcommand\color[2][]{}%
  }%
  \providecommand\transparent[1]{%
    \errmessage{(Inkscape) Transparency is used (non-zero) for the text in Inkscape, but the package 'transparent.sty' is not loaded}%
    \renewcommand\transparent[1]{}%
  }%
  \providecommand\rotatebox[2]{#2}%
  \newcommand*\fsize{\dimexpr\f@size pt\relax}%
  \newcommand*\lineheight[1]{\fontsize{\fsize}{#1\fsize}\selectfont}%
  \ifx\svgwidth\undefined%
    \setlength{\unitlength}{42.01481055bp}%
    \ifx\svgscale\undefined%
      \relax%
    \else%
      \setlength{\unitlength}{\unitlength * \real{\svgscale}}%
    \fi%
  \else%
    \setlength{\unitlength}{\svgwidth}%
  \fi%
  \global\let\svgwidth\undefined%
  \global\let\svgscale\undefined%
  \makeatother%
  \begin{picture}(1,1.60657835)%
    \lineheight{1}%
    \setlength\tabcolsep{0pt}%
    \put(0.02244953,0.70429676){\makebox(0,0)[lt]{\lineheight{1.25}\smash{\begin{tabular}[t]{l}$\Psi_{W,W}^q$\end{tabular}}}}%
    \put(0,0){\includegraphics[width=\unitlength,page=1]{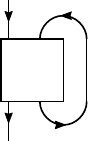}}%
  \end{picture}%
\endgroup%
 
      \centering
         \caption{For the torus links $\cT_{2,-q}$ with $q\geq 1$}
         \label{fig:torus}
     \end{subfigure}
     \begin{subfigure}[htb]{0.3\textwidth}
\begingroup%
  \makeatletter%
  \providecommand\color[2][]{%
    \errmessage{(Inkscape) Color is used for the text in Inkscape, but the package 'color.sty' is not loaded}%
    \renewcommand\color[2][]{}%
  }%
  \providecommand\transparent[1]{%
    \errmessage{(Inkscape) Transparency is used (non-zero) for the text in Inkscape, but the package 'transparent.sty' is not loaded}%
    \renewcommand\transparent[1]{}%
  }%
  \providecommand\rotatebox[2]{#2}%
  \newcommand*\fsize{\dimexpr\f@size pt\relax}%
  \newcommand*\lineheight[1]{\fontsize{\fsize}{#1\fsize}\selectfont}%
  \ifx\svgwidth\undefined%
    \setlength{\unitlength}{50.37188551bp}%
    \ifx\svgscale\undefined%
      \relax%
    \else%
      \setlength{\unitlength}{\unitlength * \real{\svgscale}}%
    \fi%
  \else%
    \setlength{\unitlength}{\svgwidth}%
  \fi%
  \global\let\svgwidth\undefined%
  \global\let\svgscale\undefined%
  \makeatother%
  \begin{picture}(1,1.70982213)%
    \lineheight{1}%
    \setlength\tabcolsep{0pt}%
    \put(0,0){\includegraphics[width=\unitlength,page=1]{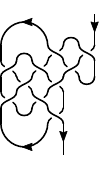}}%
    \put(0.82702883,1.64948967){\makebox(0,0)[lt]{\lineheight{1.25}\smash{\begin{tabular}[t]{l}$W$\end{tabular}}}}%
    \put(0.52924381,0.011671){\makebox(0,0)[lt]{\lineheight{1.25}\smash{\begin{tabular}[t]{l}$W$\end{tabular}}}}%
  \end{picture}%
\endgroup%
 
      \centering
         \caption{For the knot $10_{132}$}
         \label{fig:10132}
     \end{subfigure}
     \begin{subfigure}[htb]{0.3\textwidth}
\begingroup%
  \makeatletter%
  \providecommand\color[2][]{%
    \errmessage{(Inkscape) Color is used for the text in Inkscape, but the package 'color.sty' is not loaded}%
    \renewcommand\color[2][]{}%
  }%
  \providecommand\transparent[1]{%
    \errmessage{(Inkscape) Transparency is used (non-zero) for the text in Inkscape, but the package 'transparent.sty' is not loaded}%
    \renewcommand\transparent[1]{}%
  }%
  \providecommand\rotatebox[2]{#2}%
  \newcommand*\fsize{\dimexpr\f@size pt\relax}%
  \newcommand*\lineheight[1]{\fontsize{\fsize}{#1\fsize}\selectfont}%
  \ifx\svgwidth\undefined%
    \setlength{\unitlength}{48.88187744bp}%
    \ifx\svgscale\undefined%
      \relax%
    \else%
      \setlength{\unitlength}{\unitlength * \real{\svgscale}}%
    \fi%
  \else%
    \setlength{\unitlength}{\svgwidth}%
  \fi%
  \global\let\svgwidth\undefined%
  \global\let\svgscale\undefined%
  \makeatother%
  \begin{picture}(1,4.44665664)%
    \lineheight{1}%
    \setlength\tabcolsep{0pt}%
    \put(0,0){\includegraphics[width=\unitlength,page=1]{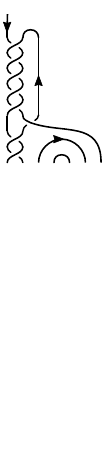}}%
    \put(-0.00679252,4.38448513){\makebox(0,0)[lt]{\lineheight{1.25}\smash{\begin{tabular}[t]{l}$W$\end{tabular}}}}%
    \put(0,0){\includegraphics[width=\unitlength,page=2]{LL22tangle.pdf}}%
    \put(-0.0049737,0.01202675){\makebox(0,0)[lt]{\lineheight{1.25}\smash{\begin{tabular}[t]{l}$W$\end{tabular}}}}%
  \end{picture}%
\endgroup%
 
      \centering
         \caption{For the link $LL_2(2)$}
         \label{fig:LL22}
     \end{subfigure}
    \caption{The tangle used to compute the invariant $\rI_W$, with $W=L(n,n+1)$. The diagrams are read from top to bottom.}
    \label{fig:tangles}
\end{figure}


\bibliography{biblio}

\begin{bibdiv}
\begin{biblist}

\bib{AA}{article}{
      author={Andruskiewitsch, Nicol\'{a}s},
      author={Angiono, Iv\'{a}n~Ezequiel},
       title={On finite dimensional {N}ichols algebras of diagonal type},
        date={2017},
        ISSN={1664-3607},
     journal={Bull. Math. Sci.},
      volume={7},
      number={3},
       pages={353\ndash 573},
         url={https://doi.org/10.1007/s13373-017-0113-x},
}

\bib{AGPS}{article}{
      author={{Aghaei}, N.},
      author={{Gainutdinov}, A.~M.},
      author={{Pawelkiewicz}, M.},
      author={{Schomerus}, V.},
       title={{Combinatorial Quantisation of $GL(1|1)$ {C}hern-{S}imons
  {T}heory {I}: {T}he {T}orus}},
        date={2018},
     journal={arXiv preprint arXiv:1811.09123 [math.hep-th]},
}

\bib{Ang1}{article}{
      author={Angiono, Iv\'{a}n~Ezequiel},
       title={On {N}ichols algebras of diagonal type},
        date={2013},
        ISSN={0075-4102},
     journal={J. Reine Angew. Math.},
      volume={683},
       pages={189\ndash 251},
         url={https://doi.org/10.1515/crelle-2011-0008},
}

\bib{Ang2}{article}{
      author={Angiono, Iv\'{a}n~Ezequiel},
       title={A presentation by generators and relations of {N}ichols algebras
  of diagonal type and convex orders on root systems},
        date={2015},
        ISSN={1435-9855},
     journal={J. Eur. Math. Soc. (JEMS)},
      volume={17},
      number={10},
       pages={2643\ndash 2671},
         url={https://doi.org/10.4171/JEMS/567},
}

\bib{AP}{article}{
      author={Andersen, Henning~Haahr},
      author={Paradowski, Jan},
       title={Fusion categories arising from semisimple {L}ie algebras},
        date={1995},
        ISSN={0010-3616},
     journal={Comm. Math. Phys.},
      volume={169},
      number={3},
       pages={563\ndash 588},
         url={http://projecteuclid.org/euclid.cmp/1104272854},
}

\bib{AS}{incollection}{
      author={Andruskiewitsch, Nicol\'{a}s},
      author={Schneider, Hans-J\"{u}rgen},
       title={Pointed {H}opf algebras},
        date={2002},
   booktitle={New directions in {H}opf algebras},
      series={Math. Sci. Res. Inst. Publ.},
      volume={43},
   publisher={Cambridge Univ. Press, Cambridge},
       pages={1\ndash 68},
         url={https://doi.org/10.2977/prims/1199403805},
}

\bib{AY}{article}{
      author={Angiono, Iv\'{a}n},
      author={Yamane, Hiroyuki},
       title={The {$R$}-matrix of quantum doubles of {N}ichols algebras of
  diagonal type},
        date={2015},
        ISSN={0022-2488},
     journal={J. Math. Phys.},
      volume={56},
      number={2},
       pages={021702, 19},
         url={https://doi.org/10.1063/1.4907379},
}

\bib{Bes}{article}{
      author={Bespalov, Yu.~N.},
       title={Crossed modules and quantum groups in braided categories},
        date={1997},
     journal={Appl. Categ. Structures},
      volume={5},
      number={2},
       pages={155\ndash 204},
}

\bib{BK}{book}{
      author={Bakalov, Bojko},
      author={Kirillov, Alexander, Jr.},
       title={Lectures on tensor categories and modular functors},
      series={University Lecture Series},
   publisher={American Mathematical Society, Providence, RI},
        date={2001},
      volume={21},
        ISBN={0-8218-2686-7},
}

\bib{KnotAtlas}{misc}{
      author={Bar-Natan, Dror},
      author={Morrison, Scott},
       title={The {K}not {A}tlas},
        date={2022},
        note={Available at \url{http://katlas.org}, retrieved on December 1,
  2022},
}

\bib{BT}{article}{
      author={Bellamy, Gwyn},
      author={Thiel, Ulrich},
       title={Highest weight theory for finite-dimensional graded algebras with
  triangular decomposition},
        date={2018},
        ISSN={0001-8708},
     journal={Adv. Math.},
      volume={330},
       pages={361\ndash 419},
         url={https://doi.org/10.1016/j.aim.2018.03.011},
}

\bib{Bur}{article}{
      author={Burciu, Sebastian},
       title={A class of {D}rinfeld doubles that are ribbon algebras},
        date={2008},
        ISSN={0021-8693},
     journal={J. Algebra},
      volume={320},
      number={5},
       pages={2053\ndash 2078},
         url={https://doi.org/10.1016/j.jalgebra.2008.05.021},
}

\bib{BW}{article}{
      author={Barrett, John~W.},
      author={Westbury, Bruce~W.},
       title={Spherical categories},
        date={1999},
        ISSN={0001-8708},
     journal={Adv. Math.},
      volume={143},
      number={2},
       pages={357\ndash 375},
         url={https://doi.org/10.1006/aima.1998.1800},
}

\bib{CDGG}{article}{
      author={Creutzig, Thomas},
      author={Dimofte, Tudor},
      author={Garner, Niklas},
      author={Geer, Nathan},
       title={{A {QFT} for non-semisimple {TQFT}}},
        date={2021},
     journal={arXiv preprint arXiv:2112.01559 [math.hep-th]},
}

\bib{CGP}{article}{
      author={Costantino, Francesco},
      author={Geer, Nathan},
      author={Patureau-Mirand, Bertrand},
       title={Quantum invariants of 3-manifolds via link surgery presentations
  and non-semi-simple categories},
        date={2014},
        ISSN={1753-8416},
     journal={J. Topol.},
      volume={7},
      number={4},
       pages={1005\ndash 1053},
         url={https://doi.org/10.1112/jtopol/jtu006},
}

\bib{CGR}{article}{
      author={Creutzig, Thomas},
      author={Gainutdinov, Azat~M.},
      author={Runkel, Ingo},
       title={A quasi-{H}opf algebra for the triplet vertex operator algebra},
        date={2020},
        ISSN={0219-1997},
     journal={Commun. Contemp. Math.},
      volume={22},
      number={3},
       pages={1950024, 71},
         url={https://doi.org/10.1142/S021919971950024X},
}

\bib{CLR}{article}{
      author={Creutzig, Thomas},
      author={Lentner, Simon},
      author={Rupert, Matthew},
       title={Characterizing braided tensor categories associated to
  logarithmic vertex operator algebras},
        date={2021},
     journal={Arxiv preprint arXiv:2104.13262 [math.QA]},
}

\bib{DGGPR}{article}{
      author={De~Renzi, Marco},
      author={Gainutdinov, Azat~M.},
      author={Geer, Nathan},
      author={Patureau-Mirand, Bertrand},
      author={Runkel, Ingo},
       title={3-{D}imensional {TQFT}s from non-semisimple modular categories},
        date={2022},
        ISSN={1022-1824},
     journal={Selecta Math. (N.S.)},
      volume={28},
      number={2},
       pages={Paper No. 42, 60},
         url={https://doi.org/10.1007/s00029-021-00737-z},
}

\bib{DGGPR2}{article}{
      author={De~Renzi, Marco},
      author={Gainutdinov, Azat~M.},
      author={Geer, Nathan},
      author={Patureau-Mirand, Bertrand},
      author={Runkel, Ingo},
       title={Mapping class group representations from non-semisimple {TQFTs}},
        date={2023},
     journal={Communications in Contemporary Mathematics},
      volume={25},
      number={01},
       pages={2150091},
         url={https://doi.org/10.1142/S0219199721500917},
}

\bib{DSS2}{article}{
      author={Douglas, Christopher~L.},
      author={Schommer-Pries, Christopher},
      author={Snyder, Noah},
       title={The balanced tensor product of module categories},
        date={2019},
        ISSN={2156-2261},
     journal={Kyoto J. Math.},
      volume={59},
      number={1},
       pages={167\ndash 179},
         url={https://doi.org/10.1215/21562261-2018-0006},
}

\bib{DSS}{article}{
      author={Douglas, Christopher~L.},
      author={Schommer-Pries, Christopher},
      author={Snyder, Noah},
       title={Dualizable tensor categories},
        date={2020},
        ISSN={0065-9266},
     journal={Mem. Amer. Math. Soc.},
      volume={268},
      number={1308},
       pages={vii+88},
         url={https://doi.org/10.1090/memo/1308},
}

\bib{DW}{article}{
      author={De~Wit, David},
       title={Automatic evaluation of the {L}inks-{G}ould invariant for all
  prime knots of up to 10 crossings},
        date={2000},
        ISSN={0218-2165},
     journal={J. Knot Theory Ramifications},
      volume={9},
      number={3},
       pages={311\ndash 339},
         url={https://doi.org/10.1142/S0218216500000153},
}

\bib{DL}{article}{
      author={De~Wit, David},
      author={Links, Jon},
       title={Where the {L}inks-{G}ould invariant first fails to distinguish
  nonmutant prime knots},
        date={2007},
        ISSN={0218-2165},
     journal={J. Knot Theory Ramifications},
      volume={16},
      number={8},
       pages={1021\ndash 1041},
         url={https://doi.org/10.1142/S0218216507005658},
}

\bib{DLK}{article}{
      author={De~Wit, David},
      author={Links, Jon~R.},
      author={Kauffman, Louis~H.},
       title={On the {L}inks-{G}ould invariant of links},
        date={1999},
        ISSN={0218-2165},
     journal={J. Knot Theory Ramifications},
      volume={8},
      number={2},
       pages={165\ndash 199},
         url={https://doi.org/10.1142/S0218216599000110},
}

\bib{EGNO}{book}{
      author={Etingof, Pavel},
      author={Gelaki, Shlomo},
      author={Nikshych, Dmitri},
      author={Ostrik, Victor},
       title={Tensor categories},
      series={Mathematical Surveys and Monographs},
   publisher={American Mathematical Society, Providence, RI},
        date={2015},
      volume={205},
}

\bib{EKT}{article}{
      author={Eliahou, Shalom},
      author={Kauffman, Louis~H.},
      author={Thistlethwaite, Morwen~B.},
       title={Infinite families of links with trivial {J}ones polynomial},
        date={2003},
        ISSN={0040-9383},
     journal={Topology},
      volume={42},
      number={1},
       pages={155\ndash 169},
         url={https://doi.org/10.1016/S0040-9383(02)00012-5},
}

\bib{EO}{inproceedings}{
      author={Etingof, Pavel},
      author={Ostrik, Victor},
       title={On semisimplification of tensor categories},
        date={2022},
   booktitle={{Representation Theory and Algebraic Geometry: A Conference
  Celebrating the Birthdays of Sasha Beilinson and Victor Ginzburg}},
      editor={Baranovsky, Vladimir},
      editor={Guay, Nicolas},
      editor={Schedler, Travis},
   publisher={Springer International Publishing},
     address={Cham},
       pages={3\ndash 35},
         url={https://doi.org/10.1007/978-3-030-82007-7_1},
}

\bib{FS1}{article}{
      author={Fuchs, J\"{u}rgen},
      author={Schweigert, Christoph},
       title={Consistent systems of correlators in non-semisimple conformal
  field theory},
        date={2017},
        ISSN={0001-8708},
     journal={Adv. Math.},
      volume={307},
       pages={598\ndash 639},
         url={https://doi.org/10.1016/j.aim.2016.11.020},
}

\bib{FScS}{book}{
      author={Frappat, L.},
      author={Sciarrino, A.},
      author={Sorba, P.},
       title={Dictionary on {L}ie algebras and superalgebras},
   publisher={Academic Press, Inc., San Diego, CA},
        date={2000},
        ISBN={0-12-265340-8},
}

\bib{FT}{article}{
      author={Feigin, B.},
      author={Tipunin, I.~Yu.},
       title={Logarithmic {CFT}s connected with simple lie algebras},
        date={2010},
     journal={Arxiv preprint arXiv:1002.5047 [math.QA]},
}

\bib{GG}{article}{
      author={Gordon, Robert},
      author={Green, Edward~L.},
       title={Graded {A}rtin algebras},
        date={1982},
        ISSN={0021-8693},
     journal={J. Algebra},
      volume={76},
      number={1},
       pages={111\ndash 137},
         url={https://doi.org/10.1016/0021-8693(82)90240-X},
}

\bib{GKP}{article}{
      author={Geer, Nathan},
      author={Kujawa, Jonathan},
      author={Patureau-Mirand, Bertrand},
       title={Generalized trace and modified dimension functions on ribbon
  categories},
        date={2011},
        ISSN={1022-1824},
     journal={Selecta Math. (N.S.)},
      volume={17},
      number={2},
       pages={453\ndash 504},
         url={https://doi.org/10.1007/s00029-010-0046-7},
}

\bib{GLO}{article}{
      author={Gainutdinov, Azat~M.},
      author={Lentner, Simon},
      author={Ohrmann, Tobias},
       title={Modularization of small quantum groups},
        date={2018},
     journal={arXiv preprint arXiv:1809.02116v2 [math.QA]},
}

\bib{GLZ}{article}{
      author={Gould, Mark~D.},
      author={Links, Jon~R.},
      author={Zhang, Yao-Zhong},
       title={Type-{I} quantum superalgebras, {$q$}-supertrace, and
  two-variable link polynomials},
        date={1996},
        ISSN={0022-2488},
     journal={J. Math. Phys.},
      volume={37},
      number={2},
       pages={987\ndash 1003},
         url={https://doi.org/10.1063/1.531422},
}

\bib{GN}{article}{
      author={Gannon, Terry},
      author={Negron, Cris},
       title={Quantum sl(2) and logarithmic vertex operator algebras at
  (p,1)-central charge},
        date={2021},
     journal={Arxiv preprint arXiv:2104.12821 [math.QA]},
}

\bib{GPT}{article}{
      author={Geer, Nathan},
      author={Patureau-Mirand, Bertrand},
      author={Turaev, Vladimir},
       title={Modified quantum dimensions and re-normalized link invariants},
        date={2009},
        ISSN={0010-437X},
     journal={Compos. Math.},
      volume={145},
      number={1},
       pages={196\ndash 212},
         url={https://doi.org/10.1112/S0010437X08003795},
}

\bib{GSTF}{article}{
      author={Ga\u{\i}nutdinov, A.~M.},
      author={Semikhatov, A.~M.},
      author={Tipunin, I.~Yu.},
      author={Fe\u{\i}gin, B.~L.},
       title={The {K}azhdan-{L}usztig correspondence for the representation
  category of the triplet {$W$}-algebra in logorithmic conformal field
  theories},
        date={2006},
        ISSN={0564-6162},
     journal={Teoret. Mat. Fiz.},
      volume={148},
      number={3},
       pages={398\ndash 427},
         url={https://doi.org/10.1007/s11232-006-0113-6},
}

\bib{Hec1}{article}{
      author={Heckenberger, I.},
       title={The {W}eyl groupoid of a {N}ichols algebra of diagonal type},
        date={2006},
        ISSN={0020-9910},
     journal={Invent. Math.},
      volume={164},
      number={1},
       pages={175\ndash 188},
         url={https://doi.org/10.1007/s00222-005-0474-8},
}

\bib{Hec2}{article}{
      author={Heckenberger, I.},
       title={Classification of arithmetic root systems},
        date={2009},
        ISSN={0001-8708},
     journal={Adv. Math.},
      volume={220},
      number={1},
       pages={59\ndash 124},
         url={https://doi.org/10.1016/j.aim.2008.08.005},
}

\bib{Hec3}{article}{
      author={Heckenberger, I.},
       title={Lusztig isomorphisms for {D}rinfel\cprime d doubles of
  bosonizations of {N}ichols algebras of diagonal type},
        date={2010},
        ISSN={0021-8693},
     journal={J. Algebra},
      volume={323},
      number={8},
       pages={2130\ndash 2182},
         url={https://doi.org/10.1016/j.jalgebra.2010.02.013},
}

\bib{HN}{article}{
      author={Holmes, Randall~R.},
      author={Nakano, Daniel~K.},
       title={Brauer-type reciprocity for a class of graded associative
  algebras},
        date={1991},
        ISSN={0021-8693},
     journal={J. Algebra},
      volume={144},
      number={1},
       pages={117\ndash 126},
         url={https://doi.org/10.1016/0021-8693(91)90132-R},
}

\bib{HS}{book}{
      author={Heckenberger, I.},
      author={Schneider, Hans-J\"urgen},
       title={Hopf {A}lgebras and {R}oot {S}ystems},
      series={Mathematical Surveys and Monographs},
   publisher={American Mathematical Society, Providence, RI},
        date={2020},
      volume={247},
        ISBN={978-1-4704-5680-1},
}

\bib{Kac-77}{article}{
      author={Kac, Victor},
       title={Lie superalgebras},
        date={1977},
        ISSN={0001-8708},
     journal={Adv. Math.},
      volume={26},
      number={1},
       pages={8\ndash 96},
         url={https://doi.org/10.1016/0001-8708(77)90017-2},
}

\bib{Kac-78}{inproceedings}{
      author={Kac, V.},
       title={Representations of classical {L}ie superalgebras},
        date={1978},
   booktitle={Differential geometrical methods in mathematical physics, {II}
  ({P}roc. {C}onf., {U}niv. {B}onn, {B}onn, 1977)},
      series={Lecture Notes in Math.},
      volume={676},
   publisher={Springer, Berlin},
       pages={597\ndash 626},
}

\bib{KL}{book}{
      author={Kerler, Thomas},
      author={Lyubashenko, Volodymyr~V.},
       title={Non-semisimple topological quantum field theories for 3-manifolds
  with corners},
      series={Lecture Notes in Mathematics},
   publisher={Springer-Verlag, Berlin},
        date={2001},
      volume={1765},
        ISBN={3-540-42416-4},
}

\bib{KR93}{article}{
      author={Kauffman, Louis~H.},
      author={Radford, David~E.},
       title={A necessary and sufficient condition for a finite-dimensional
  {D}rinfel\cprime d double to be a ribbon {H}opf algebra},
        date={1993},
        ISSN={0021-8693},
     journal={J. Algebra},
      volume={159},
      number={1},
       pages={98\ndash 114},
         url={https://doi.org/10.1006/jabr.1993.1148},
}

\bib{KT}{article}{
      author={Khoroshkin, S.~M.},
      author={Tolstoy, V.~N.},
       title={Universal {$R$}-matrix for quantized (super)algebras},
        date={1991},
        ISSN={0010-3616},
     journal={Comm. Math. Phys.},
      volume={141},
      number={3},
       pages={599\ndash 617},
         url={http://projecteuclid.org/euclid.cmp/1104248397},
}

\bib{Lau2}{article}{
      author={Laugwitz, Robert},
       title={Comodule algebras and 2-cocycles over the (braided) {D}rinfeld
  double},
        date={2019},
        ISSN={0219-1997},
     journal={Commun. Contemp. Math.},
      volume={21},
      number={4},
       pages={1850045, 46},
         url={https://doi.org/10.1142/S0219199718500451},
}

\bib{Lau3}{article}{
      author={Laugwitz, Robert},
       title={The relative monoidal center and tensor products of monoidal
  categories},
        date={2020},
        ISSN={0219-1997},
     journal={Commun. Contemp. Math.},
      volume={22},
      number={8},
       pages={1950068, 53},
         url={https://doi.org/10.1142/S0219199719500688},
}

\bib{Len}{article}{
      author={Lentner, Simon~D.},
       title={Quantum groups and {N}ichols algebras acting on conformal field
  theories},
        date={2021},
        ISSN={0001-8708},
     journal={Adv. Math.},
      volume={378},
       pages={Paper No. 107517, 71},
         url={https://doi.org/10.1016/j.aim.2020.107517},
}

\bib{LMSS}{book}{
      author={Lentner, Simon},
      author={Mierach, Svea~Nora},
      author={Schweigert, Christoph},
      author={Sommerh\"{a}user, Yorck},
       title={Hochschild cohomology, modular tensor categories, and mapping
  class groups. {I}},
      series={SpringerBriefs in Mathematical Physics},
   publisher={Springer, Singapore},
        date={2023},
      volume={44},
        ISBN={978-981-19-4644-8; 978-981-19-4645-5},
}

\bib{LO}{article}{
      author={Lentner, Simon},
      author={Ohrmann, Tobias},
       title={Factorizable {$R$}-matrices for small quantum groups},
        date={2017},
     journal={SIGMA Symmetry Integrability Geom. Methods Appl.},
      volume={13},
       pages={Paper No. 076, 25},
         url={https://doi.org/10.3842/SIGMA.2017.076},
}

\bib{LW1}{article}{
      author={Laugwitz, Robert},
      author={Walton, Chelsea},
       title={Braided commutative algebras over quantized enveloping algebras},
        date={2021},
        ISSN={1083-4362},
     journal={Transform. Groups},
      volume={26},
      number={3},
       pages={957\ndash 993},
         url={https://doi.org/10.1007/s00031-020-09599-9},
}

\bib{LW2}{article}{
      author={Laugwitz, Robert},
      author={Walton, Chelsea},
       title={Constructing non-semisimple modular categories with relative
  monoidal centers},
        date={2022},
        ISSN={1073-7928},
     journal={Int. Math. Res. Not. IMRN},
      number={20},
       pages={15826\ndash 15868},
         url={https://doi.org/10.1093/imrn/rnab097},
}

\bib{Lyu}{article}{
      author={Lyubashenko, Volodymyr~V.},
       title={Invariants of {$3$}-manifolds and projective representations of
  mapping class groups via quantum groups at roots of unity},
        date={1995},
        ISSN={0010-3616},
     journal={Comm. Math. Phys.},
      volume={172},
      number={3},
       pages={467\ndash 516},
         url={http://projecteuclid.org/euclid.cmp/1104274312},
}

\bib{Mue2}{article}{
      author={M\"{u}ger, Michael},
       title={From subfactors to categories and topology. {II}. {T}he quantum
  double of tensor categories and subfactors},
        date={2003},
        ISSN={0022-4049},
     journal={J. Pure Appl. Algebra},
      volume={180},
      number={1-2},
       pages={159\ndash 219},
         url={https://doi.org/10.1016/S0022-4049(02)00248-7},
}

\bib{Majid}{book}{
      author={Majid, Shahn},
       title={Foundations of quantum group theory},
   publisher={Cambridge University Press, Cambridge},
        date={2000},
        ISBN={0-521-64868-8},
         url={http://dx.doi.org/10.1017/CBO9780511613104},
        note={Paperback Edition, originally published 1995},
}

\bib{Maj99}{article}{
      author={Majid, Shahn},
       title={Double-bosonization of braided groups and the construction of
  {$U_q(\germ g)$}},
        date={1999},
     journal={Math. Proc. Cambridge Philos. Soc.},
      volume={125},
      number={1},
       pages={151\ndash 192},
}

\bib{Musson}{book}{
      author={Musson, Ian~M.},
       title={Lie superalgebras and enveloping algebras},
      series={Graduate Studies in Mathematics},
   publisher={American Mathematical Society},
        date={2012},
      volume={131},
        ISBN={0-8218-6867-5; 978-0-8218-6867-6},
}

\bib{Negron}{article}{
      author={Negron, Cris},
       title={Log-modular quantum groups at even roots of unity and the quantum
  {F}robenius {I}},
        date={2021},
        ISSN={0010-3616},
     journal={Comm. Math. Phys.},
      volume={382},
      number={2},
       pages={773\ndash 814},
         url={https://doi.org/10.1007/s00220-021-04012-2},
}

\bib{Rad}{book}{
      author={Radford, David~E.},
       title={Hopf algebras},
      series={Series on Knots and Everything},
   publisher={World Scientific Publishing Co. Pte. Ltd., Hackensack, NJ},
        date={2012},
      volume={49},
        ISBN={978-981-4335-99-7; 981-4335-99-1},
}

\bib{Rol}{book}{
      author={Rolfsen, Dale},
       title={Knots and links},
      series={Mathematics Lecture Series},
   publisher={Publish or Perish, Inc., Houston, TX},
        date={1990},
      volume={7},
        ISBN={0-914098-16-0},
        note={Corrected reprint of the 1976 original},
}

\bib{RT90}{article}{
      author={Reshetikhin, N.~Yu.},
      author={Turaev, V.~G.},
       title={Ribbon graphs and their invariants derived from quantum groups},
        date={1990},
        ISSN={0010-3616},
     journal={Comm. Math. Phys.},
      volume={127},
      number={1},
       pages={1\ndash 26},
         url={http://projecteuclid.org/euclid.cmp/1104180037},
}

\bib{RT91}{article}{
      author={Reshetikhin, N.~Yu.},
      author={Turaev, V.~G.},
       title={Invariants of {$3$}-manifolds via link polynomials and quantum
  groups},
        date={1991},
        ISSN={0020-9910},
     journal={Invent. Math.},
      volume={103},
      number={3},
       pages={547\ndash 597},
         url={https://doi.org/10.1007/BF01239527},
}

\bib{Shi1}{article}{
      author={Shimizu, Kenichi},
       title={Non-degeneracy conditions for braided finite tensor categories},
        date={2019},
        ISSN={0001-8708},
     journal={Adv. Math.},
      volume={355},
       pages={106778, 36},
         url={https://doi.org/10.1016/j.aim.2019.106778},
}

\bib{Shi2}{article}{
      author={Shimizu, Kenichi},
       title={Ribbon structures of the {D}rinfeld center of a finite tensor
  category},
        date={2023},
        ISSN={0386-5991},
     journal={Kodai Math. J.},
      volume={46},
      number={1},
       pages={75\ndash 114},
         url={https://doi.org/10.2996/kmj46106},
}

\bib{SW}{article}{
      author={Schweigert, Christoph},
      author={Woike, Lukas},
       title={Homotopy coherent mapping class group actions and excision for
  {H}ochschild complexes of modular categories},
        date={2021},
        ISSN={0001-8708},
     journal={Advances in Mathematics},
      volume={386},
       pages={107814},
  url={https://www.sciencedirect.com/science/article/pii/S000187082100253X},
}

\bib{Tak}{article}{
      author={Takeuchi, Mitsuhiro},
       title={Survey of braided {H}opf algebras},
        date={2000},
     journal={Contemp. Math.},
      volume={267},
       pages={301\ndash 324},
         url={http://nyjm.albany.edu:8000/j/1996/2_35.html},
}

\bib{TV}{book}{
      author={Turaev, Vladimir},
      author={Virelizier, Alexis},
       title={Monoidal categories and topological field theory},
      series={Progress in Mathematics},
   publisher={Birkh\"{a}user/Springer, Cham},
        date={2017},
      volume={322},
}

\bib{Vay}{incollection}{
      author={Vay, Cristian},
       title={On {H}opf algebras with triangular decomposition},
        date={2019},
   booktitle={Tensor categories and {H}opf algebras},
      series={Contemp. Math.},
      volume={728},
   publisher={Amer. Math. Soc., [Providence], RI},
       pages={181\ndash 199},
         url={https://doi.org/10.1090/conm/728/14662},
}

\bib{Yam}{article}{
      author={Yamane, Hiroyuki},
       title={Quantized enveloping algebras associated with simple {L}ie
  superalgebras and their universal {$R$}-matrices},
        date={1994},
        ISSN={0034-5318},
     journal={Publ. Res. Inst. Math. Sci.},
      volume={30},
      number={1},
       pages={15\ndash 87},
         url={https://doi.org/10.2977/prims/1195166275},
}

\end{biblist}
\end{bibdiv}
\bibliographystyle{amsrefs}

\end{document}